\documentclass[reqno,10pt]{amsart}
\usepackage{amssymb,amsmath,amsthm,}
\usepackage{a4wide}
\usepackage{amscd}
\usepackage{amsfonts}
\usepackage{amssymb}
\usepackage{latexsym}
\usepackage{color}
\usepackage{esint}
\usepackage{graphicx}
\usepackage{caption}
\usepackage{subcaption}
\usepackage{amsthm}
 \usepackage{verbatim}
\usepackage[colorlinks]{hyperref}

\usepackage[makeroom]{cancel}
\usepackage{mathtools}
\DeclarePairedDelimiter\abs{\lvert}{\rvert}
\DeclarePairedDelimiter\norm{\lVert}{\rVert}
\let\hide\iffalse
\let\unhide\fi

\newtheorem{theorem}{Theorem}[section]

\newtheorem{definition}[theorem]{Definition}
\newtheorem{lemma}[theorem]{Lemma}
\newtheorem{proposition}[theorem]{Proposition}
\newtheorem{remark}[theorem]{Remark}

\let\p=\partial
\let\t=\vartheta
\let\O=\Omega

\let\t=\theta
\let\b=\beta

\newcommand{\R}{\mathbb{R}}

\renewcommand{\S}{\mathbb{S}}

\newcommand{\be}{\begin{equation}}
\newcommand{\bm}{\begin{multline}}
\newcommand{\ee}{\end{equation}}

\newcommand{\xb}{x_{\mathbf{b}}}

\newcommand{\tb}{t_{\mathbf{b}}}

\newcommand{\Bes}{\begin{eqnarray*}}
	\newcommand{\Ees}{\end{eqnarray*}}
\newcommand{\Be}{\begin{equation}}
\newcommand{\Ee}{\end{equation}}

\def\p{\partial}

\def\O{\Omega}
\def\R{\mathbb{R}}

\def\B{\begin{equation}}
\def\E{\end{equation}}
\def\BN{\begin{eqnarray*}}
\def\EN{\end{eqnarray*}}

\newcommand{\om}{\omega}
\newcommand{\ka}{\kappa}
\newcommand{\CE}{\mathcal{E}}

\numberwithin{equation}{section}

\usepackage{tikz-cd}
\usepackage{ragged2e}
\usepackage{etoolbox}
\usepackage{lipsum}

\usepackage{color}

\begin{document}

\title[Boltzmann equation with large-amplitude data and specular reflection]{The Boltzmann equation with large-amplitude initial data and specular reflection boundary condition}

\author[R.-J. Duan]{Renjun Duan}
\address[RJD]{Department of Mathematics, The Chinese University of Hong Kong,
Shatin, Hong Kong, P.R.~China}
\email{rjduan@math.cuhk.edu.hk}

\author[G. Ko]{Gyounghun Ko}
\address[GK]{Department of Mathematics, Pohang University of Science and Technology, South Korea}
\email{gyounghun347@postech.ac.kr}

\author[D. Lee]{Donghyun Lee}
\address[DL]{Department of Mathematics, Pohang University of Science and Technology, South Korea}
\email{donglee@postech.ac.kr}

\begin{abstract}
For the Boltzmann equation with cutoff hard potentials, we construct the unique global solution converging with an exponential rate in large time to global Maxwellians  not only for the specular reflection boundary condition with the bounded convex $C^3$ domain but also for a class of large amplitude initial data where the $L^\infty$ norm with a suitable velocity weight can be arbitrarily large but the relative entropy need to be small. A key point in the proof is to introduce a delicate nonlinear iterative process of estimating the gain term basing on the triple Duhamel iteration along the linearized dynamics.
\end{abstract}

\date{\today}
\keywords{Boltzmann equation, specular boundary condition, large-amplitude initial data, triple iteration, Gronwall argument}
\maketitle

\thispagestyle{empty}
\setcounter{tocdepth}{2}
\tableofcontents

\section{Introduction}

\subsection{Problem} 
In the paper, we are concerned with the initial-boundary value problem on the Boltzmann equation with the specular reflection boundary condition in a $C^3$ uniformly convex bounded domain $\Omega\subset \R^3$. The Boltzmann equation takes the form of
\begin{equation}
\label{def.be}
\p_{t}F + v\cdot\nabla_{x} F = Q(F,F). 
\end{equation}
Here, the unknown $F=F(t,x,v)\geq 0$ stands for the density distribution function of gas particles with velocity $v\in \R^3$ and position  $x\in \Omega$ at time $t>0$. The bilinear Boltzmann collision operator acting only on velocity variables is given by
\begin{align}
Q(G,F)(v)&=\int_{\mathbb{R}^3}\int_{\mathbb{S}^2} B(v-u,\om)\left[G(u')F(v')
-G(u)F(v)\right]\,d\omega du,
\label{def.Q}
\end{align}
where the velocity pairs $(u',v')$ and $(u,v)$  satisfy the relation
\begin{equation*}
u'=u+[(v-u)\cdot\omega]\,\omega,\quad
v'=v-[(v-u)\cdot\omega]\,\omega,
\end{equation*}
with $\om\in \S^2$, according to conservations of momentum and energy of two particles for an elastic collision
\begin{equation*}
u+v=u'+v',\quad 
|u|^2+|v|^2=|u'|^2+|v'|^2.
\end{equation*}
For the collision kernel $B(v-u,\om)$, it depends only on the relative velocity $|v-u|$ and $\cos\theta:=(v-u)\cdot \om/|v-u|$, cf.~\cite{CIP}. Throughout the paper, we assume that  
\begin{equation}\label{1.4}
B(v-u,\om)=|v-u|^{\ka}q_0(\theta),\quad 
0\leq\ka\leq1,\quad 0\leq q_0(\t)\leq C|\cos\t |,
\end{equation}
for the case of hard potentials with angular cutoff. Under the cutoff assumption, it is convenient to use $Q_\pm(G,F)$ to denote the gain and loss terms on the right-hand side of \eqref{def.Q}, respectively. To solve \eqref{def.be}, we supplement it with initial data
\begin{equation}
\label{id}
F(0,x,v)= F_{0}(x,v),\quad (x,v)\in \Omega\times \R^3,
\end{equation}
and focus only on the specular reflection boundary condition 
\begin{equation} 
\label{specular}
F(t,x,v) = F(t,x,R_{x}v),\quad (t,x,v)\in [0,\infty)\times\partial\Omega\times \R^3, 
\end{equation}
with $R_x v = v- 2 (n(x)\cdot v)n(x)$, where $n(x)$ is the outward normal vector at $x\in \partial \Omega$. For later use, we define the boundary of the phase space as $\gamma:=\{(x,v)\in \partial \O\times \R^3\}$, and split $\gamma$ into an outgoing boundary $\gamma_+$, an incoming boundary $\gamma_-$, and a singular boundary $\gamma_0$:  
\begin{align*}
	\gamma_+=\{(x,v)\in \partial \O\times \R^3 : \quad n(x) \cdot v>0\},\\ 
	\gamma_-=\{(x,v)\in \partial \O\times \R^3 : \quad n(x) \cdot v<0\},\\
	\gamma_0=\{(x,v)\in \partial \O\times \R^3 : \quad n(x) \cdot v=0\}.
\end{align*}

To the end we are devoted to constructing the unique global solution $F(t,x,v)\geq 0$ converging in large time to the global Maxwellian 
$$
\mu=\frac{1}{(2\pi)^{\frac{3}{2}}}e^{-\frac{|v|^2}{2}}
$$ 
on the initial-boundary value problem \eqref{def.be}, \eqref{id} and \eqref{specular} for a class of initial data of large amplitude. In fact, it is well known that for general initial data $F_0(x,v)\geq 0$ with all the physical bounds, the global existence of renormalized solutions was established first by DiPerna-Lions \cite{D-Lion} for the whole space and later by Hamdache \cite{Ha} for a domain bounded or not with general boundary conditions including the specular reflection. In particular, for a bounded domain, a renormalized solution to the initial-boundary value problem \eqref{def.be}, \eqref{id} and \eqref{specular} should exist globally in time for any initial data with
\begin{equation*}
\int_{\Omega}\int_{\R^3} F_0 (1+|v|^2+|\log F_0|)\,dvdx<\infty.
\end{equation*}
However, it is a big open problem to obtain uniqueness of those solutions. Note that the long time asymptotics of such large-data solutions towards equilibrium was also studied in Desvillettes-Villani \cite{DV} under additional uniform-in-time regularity conditions on solutions. The goal of this paper is to show the global-in-time existence of unique solutions in $L^\infty$ space for any nonnegative initial data $F_0(x,v)$ that has a finite $L^\infty$ norm weighted by a suitable velocity function but admits a small enough relative entropy $\iint F_0\ln \frac{F_0}{\mu}$, and further prove the exponential convergence of solutions to $\mu$.

\subsection{Main result}

In terms of the standard perturbation $f$ such that $F(t,x,v)=\mu+\sqrt{\mu}f(t,x,v)$,  the Boltzmann equation \eqref{def.be} can be rewritten as
	\begin{equation}\label{f_eqtn}
	\p_{t}f + v\cdot \nabla_{x} f+Lf  
	= \Gamma(f,f) , 
	\end{equation} 
with the usual notations on the linearized term
 \begin{equation} \label{L_operator}
 	Lf = - \frac{1}{\sqrt{\mu}}\{ Q(\mu, \sqrt{\mu}f)+Q(\sqrt{\mu}f,\mu)\} = \nu(v) f -Kf = \nu(v) f - \int_{\R^3} k(v,\eta)f(\eta)\,d\eta,
 \end{equation}
 and the nonlinear term
 \begin{equation*}
     \Gamma(g,f) = \frac{1}{\sqrt{\mu}}Q(\sqrt{\mu} g, \sqrt{\mu}f).
 \end{equation*}
Here, we have denoted the collision frequency at the equilibrium state 
\begin{equation}
\label{def.nu}
\nu (v) \equiv \int_{\R^3}\int_{\S^2} \vert v-u \vert ^\kappa \mu(u) q_0(\theta)\, d\omega d u \sim  (1+\vert v \vert) ^{\kappa}
\end{equation}
 for $0\leq \kappa \leq 1$. The integral part $Kf=\Gamma(f,\sqrt{\mu})+\Gamma_+(\sqrt{\mu},f)$ in \eqref{L_operator} can be written as $Kf=K_2f-K_1f$ with 
\begin{align} \label{k_1 and k_2}
	\begin{split}
		(K_1f)(v) &= \Gamma_-(f,\sqrt{\mu}) = \int_{\R^3} k_1(v,\eta)f(\eta) \,d \eta, \\ 
		(K_2f)(v)&=  \Gamma_+(\sqrt{\mu},f)+\Gamma_+(f,\sqrt{\mu}) = \int_{\R^3} k_2(v,\eta)f(\eta) \,d \eta. 
	\end{split}
\end{align}
Note that $k(\eta,v)=k_2(v,\eta)-k_1(v,\eta)$. 

For the specular reflection boundary condition \eqref{specular}, it is well known that both mass and energy are conserved for \eqref{def.be}. Without loss of generality, we may always assume that the mass-energy conservation laws hold for any $t\geq 0$, namely, in terms of the perturbation $f$, 
\begin{equation} \label{conservation}
    \int_{\O} \int_{\R^3} \sqrt{\mu}f(t,x,v)  \, dvdx \, =\, 0, \quad 
    \int_{\O} \int_{\R^3} \vert v\vert^2 \sqrt{\mu}f(t,x,v) \, dvdx \, =\, 0.
\end{equation}
We assume that the domain $\O$ is connected and bounded, and there exists a $C^{3}$ function $\xi(x)$ such that $\O = \{ x\in\R^{3} : \xi(x) < 0 \}$. We can choose $\xi$ so that $\nabla\xi(x)\neq 0$ at the boundary and then the outward unit normal vector $n(x) = \nabla\xi(x)/|\nabla\xi(x)|$ is well-defined. Moreover, in this paper, we consider that $\O$ is uniformly convex,  i.e., there is $C_{\O}>0$ such that 
	$\sum_{i,j=1}^{3} \p_{ij}\xi(x)\zeta_{i}\zeta_{j} \geq C_{\O}|\zeta|^{2}$
for any $\zeta\in\R^{3}$.
Meanwhile, $\O$ is said to have the rotational symmetry if  there exist two vectors $x_0$ and $u_0$ such that
\Be \label{rot sym}
	[(x-x_0)\times u_0]\cdot n(x) = 0,\quad \forall x\in\p\O.
\Ee
Under the specular reflection boundary condition in case of a rotational symmetric domain, the angular momentum is also conserved for \eqref{def.be} and hence we always assume
\Be \label{angular conserv}
	\int_{\O} \int_{\R^3} [(x-x_0)\times u_0]\cdot v \sqrt{\mu} f(t,x,v)\, dv dx = 0.
\Ee
As shown in Lemma \ref{decay_entropy}, another uniform-in-time quantity that is non-increasing in time along the nonlinear Boltzmann solution is the relative entropy given by
\begin{equation}\label{Relative entropy}
	\mathcal{E}(F) := \int_{\O} \int_{\R^3} \left(\frac{F}{\mu} \ln \frac{F}{\mu} - \frac{F}{\mu}+1 \right) \mu \, dvdx, 
\end{equation} 
where the integrand is nonnegative since it holds that $a\ln a-a+1\geq 0$ for any $a> 0$. Notice that under the mass conservation, \eqref{Relative entropy} can be reduced to $\mathcal{E}(F)=\iint F\ln \frac{F}{\mu}\,dvdx$.


The main result of this paper is stated as follows.

\begin{theorem} \label{main thm}
Let $\Omega$ be a general $C^3$ bounded uniformly convex domain. Assume that $f_0$ satisfies \eqref{conservation} as well as \eqref{angular conserv} in addition if $\O$ has the rotational symmetry \eqref{rot sym}. Denote a velocity weight function
\begin{equation}
\label{def.vw}
w = w_{\rho}(v)=(1+\rho^2 \vert v \vert ^2) ^\beta e^{\varpi \vert v \vert^2 }
\end{equation}
with fixed constants $0 < \varpi < \frac{1}{64}$ and $\b \geq \frac{5}{2}$, where $\rho>0$ is a constant to be determined later. Then, for any $M_0 > 0$, there are $\rho=\rho(M_0)>0$ and $\epsilon_{0}=\epsilon(M_0)>0$ such that if initial data satisfy that $F_0(x,v)=\mu+\sqrt{\mu}f_0(x,v)\geq 0$ and  
\begin{equation}\label{thm.conid}
	\Vert wf_0 \Vert _{L^\infty} \leq M_0, \quad \mathcal{E}(F_0) \leq \epsilon_0, 
\end{equation}
then the initial-boundary value problem \eqref{def.be}, \eqref{id} and \eqref{specular}  on the Boltzmann equation under the assumption \eqref{1.4} admits a unique global-in-time solution $F(t,x,v)= \mu + \sqrt{\mu}f(t,x,v)\geq 0$ satisfying 
\begin{equation}\label{thm.td}
	\Vert wf(t) \Vert_{L^{\infty}} \leq C (M_0 + M_0^2) \exp\left \{\frac{4}{\nu_0} C (M_0 + M_0^2) \right \} e^{-\vartheta t} , 
\end{equation} 
for all $t\geq 0$, where $C\geq 1$ and $\vartheta = \min\{\frac{\nu_0}{8}, \lambda\}>0$ are generic constants with $\nu_0:= \inf_{v\in \R^3} \nu(v) >0$ and $\lambda>0$ given in Proposition \ref{LeeKim}.  Moreover, if $f_0(x,v)$ is continuous except on $\gamma_0$ and satisfies the initial-boundary compatibility condition
\[
	f_0(x,v) = f_0(x, R_x v),\quad \forall x\in\p\O,
\] 
then $f(t,x,v)$ is also continuous in $[0,\infty)\times \{\overline{\O}\times \R^{3}\backslash \gamma_0\}$. 
\end{theorem}

In contrast to the DiPerna-Lions' large-data theory for general initial data mentioned before, the global well-posedness has been well understood in the perturbation framework and many results exist in this direction. In particular, the problem was first solved by Ukai \cite{Ukai} in case when the space domain is a periodic box and initial data are chosen as $F_0(x,v)=\mu+\sqrt{\mu}f_0(x,v)\geq 0$ such that 
$\sup_{v}\langle v\rangle^\beta \|f_0(\cdot,v)\|_{H^N_x}$
is sufficiently small for suitably large $\beta$ and $N$. Later, Shizuta-Asano \cite{SA} announced an extension to general smooth bounded convex domains for specular reflection boundary condition without much details of the proof. For general bounded convex domains with specular reflection boundary condition, Guo \cite{Guo} gave a complete proof via a novel $L^2$-$L^\infty$ approach to construct the global unique solution provided that $\sup_{x,v}\langle v\rangle^\beta |f_0(x,v)|$ is sufficiently small for suitably large $\beta$ and the boundary of domain is a level set of a real analytic function. Notice that besides specular reflection boundary condition, other types of boundary conditions such as the inflow, bounce-back and diffuse reflection were also treated in \cite{Guo}. In case of the specular reflection boundary condition, Kim-Lee \cite{KL} removed the analyticity assumption in \cite{Guo} and replaced it by the general $C^3$ regularity for a bounded convex domain. Recently, Guo-Zhou \cite{GZ} constructed the global-in-time Boltzmann solution in the diffusive limit for the general Maxwell boundary condition with the full range of accommodation coefficients even in the case when the space domain is not necessarily convex and the boundary admits $C^2$ regularity only; see also Briant-Guo \cite{BG} on the construction of global close-to-equilibrium solutions for the finite Knudsen number and for the Maxwell boundary condition of the partial range of accommodation coefficients without the convexity and analyticity assumptions on $\Omega$. We also refer \cite{EGKM, EGKM-hy, GKTT, GL, KL-nc} for other recent results about non-convex domain, discontinuity, regularity,  and hyrodynamic limit results for the Boltzmann boundary problems.

On the other hand, in those aforementioned results, the initial perturbation $f_0(x,v)$ has to have a small oscillation in space variables due to smallness of $L^\infty$ norm. An extension of \cite{Guo} to a class of large amplitude initial data in the sense that the $L^\infty$ norm was allowed to be arbitrarily large but the $L^p$ norm for some $p<\infty$ is suitably small was recently made in Duan-Huang-Wang-Yang \cite{DHWY17} and Duan-Wang \cite{DW} either for the whole space and periodic box or for the bounded domain with diffusion reflection boundary condition, respectively. Inspired by \cite{DHWY17,DW,KL}, we are going to explore in the paper the possibility of constructing the global Boltzmann solution not only for the specular reflection boundary condition with the $C^3$ boundary regularity but also for a class of large amplitude initial data. It turns out that the direct combination of those approaches in \cite{DHWY17,DW,KL} does not work well. We will have to figure out new difficulties and strategies to treat the problem under consideration.  

%


\subsection{Example of large-amplitude initial data} 
Those conditions on initial data in Theorem \ref{main thm} can be satisfied for a class of functions of large amplitude for pointwise space and velocity variables. To give a typical example, we choose 
\begin{equation}
\label{def.exid}
F_0(x,v)=\mu+\sqrt{\mu} f_0(x,v)\ \text{ with }\ f_0(x,v):=\frac{\phi(x)-1}{w}\sqrt{\mu},
\end{equation}
where the velocity weight $w=w(v)$ is given in \eqref{def.vw} and $\phi(x)$ is to be chosen such that all conditions on $F_0(x,v)$ hold true. First of all, letting $\phi(x)\geq 0$, the nonnegativity of $F_0(x,v)$ is guaranteed in terms of \eqref{def.exid} by 
\begin{equation*}
F_0=\mu +\mu \frac{\phi(x)-1}{w}= \mu\left\{(1-\frac{1}{w}) +\frac{\phi(x)}{w}\right\}\geq 0
\end{equation*}
because $w\geq 1$. And, as long as 
\begin{equation}
\label{mass.con}
\int_\Omega (\phi(x)-1)\,dx=0,
\end{equation}
we have the conservations laws \eqref{conservation}  of mass and energy, as well as angular momentum conservation \eqref{angular conserv} in addition if $\O$ has the rotational symmetry \eqref{rot sym}. Letting $M_0>0$ be finite but arbitrarily large, in terms of the condition on $L^\infty$ norm in \eqref{thm.conid}, we choose $\phi(x)$ such that
\begin{equation*}
M_0=\|wf_0\|_{L^\infty}=\sup_x |\phi(x)-1|\cdot \sup_v \sqrt{\mu} \sim \sup_x |\phi(x)-1|.
\end{equation*} 
For the condition on the relative entropy in \eqref{thm.conid}, in terms of \eqref{def.exid} and \eqref{mass.con},  one has to further require
\begin{equation}
\label{rep1}
\CE(F_0)=\int_\Omega\int_{\R^3}\mu (1+\frac{\phi(x)-1}{w})\ln(1+\frac{\phi(x)-1}{w})\,dvdx
\end{equation}
to be as small as we wish. To verify this, using the convexity of $\Phi(s):=s\ln s$ over $s>0$ and noticing $0<\frac{1}{w}<1$, one has 
\begin{equation*}
\Phi (1+\frac{\phi(x)-1}{w})=\Phi ((1-\frac{1}{w})\cdot 1 +\frac{1}{w}\phi(x))\leq (1-\frac{1}{w})\Phi(1) +\frac{1}{w}\Phi(\phi(x)).
\end{equation*}
By $\Phi(1)=0$, it further holds that 
\begin{equation*}
\Phi (1+\frac{\phi(x)-1}{w})\leq \frac{1}{w}\Phi(\phi(x))=\frac{\phi(x)\ln \phi(x)}{w}.
\end{equation*}
Plugging this into \eqref{rep1} gives that
\begin{equation*}
\CE(F_0)\leq \int_{\Omega} \big(\phi (x)\ln \phi(x)-\phi(x)+1\big)\,dx\int_{\R^3}\frac{\mu}{w}\,dv,
\end{equation*}
where we have used \eqref{mass.con} to rewrite $\int \phi \ln \phi=\int (\phi \ln \phi -\phi+1)$ such that its integrand is nonnegative. Notice that similar to \eqref{def.intrho}, one has $\int_{\R^3}\frac{\mu}{w}\,dv\leq \frac{C_\beta}{\rho^3}$ that can be small for $\rho>0$ large enough. Hence, provided that $\int (\phi \ln \phi -\phi+1)$ is finite, the initial relative entropy $\CE(F_0)$ can be small for any $\rho\geq \rho_0$ with a constant $\rho_0>0$ depending only on $\int (\phi \ln \phi -\phi+1)$ and $\beta$. In sum, those conditions on initial data in Theorem \ref{main thm} can be satisfied for any initial data of the specific form \eqref{def.exid} with $\phi(x)$ satisfying
\begin{equation*}
\left\{\begin{aligned}
&\phi(x)\geq 0, \, x\in \Omega;\quad \sup_{x\in \Omega} \phi(x)<\infty;\\
&\int_{\Omega} (\phi(x)-1)\,dx=0;\\
& \int_{\Omega} \big(\phi (x)\ln \phi(x)-\phi(x)+1\big)dx<\infty.
\end{aligned}\right.
\end{equation*} 
It is obvious to see that such $\phi(x)$ can have an arbitrarily large amplitude, or equivalently it is the case for $F_0(x,v)$ of the form \eqref{def.exid}. Note that different from \cite{DW}, the relative entropy $\int \phi\ln \phi$ does not need to be small.

\begin{remark}
It should be pointed out that although $F_0(x,v)$ can be arbitrarily large for pointwise $x\in \Omega$ and $v\in \R^3$, the local moment functions for any order $m\geq 0$
\begin{equation}
\label{rem.idp}
\int_{\R^3}|v|^mF_0(x,v)\,dv=\int_{\R^3}|v|^m(\mu+\sqrt{\mu}f_0(x,v))\,dv=C_m+\int_{\R^3}|v|^m\sqrt{\mu}f_0(x,v)\,dv
\end{equation} 
have to admit positive lower and upper bounds uniform in $x\in \Omega$, in particular, the initial local density function $\int_{\R^3}F_0(x,v)\,dv$ has to be far from zero and hence is not allowed to contain any vacuum, which is different from the situation in \cite{DW}. In fact,  as seen from the proof, for instance, in terms of \eqref{rho} and \eqref{def.bM}, one can choose 
\begin{equation}
\label{def.consrho}
\rho=[8CC_\ast (M_0+M_0^2)]^{\frac{1}{3}}\exp \{\frac{4}{3\nu_0}C(M_0+M_0^2)\}
\end{equation}
for an arbitrarily large constant $M_0>0$, where $C$ and $C_\ast$ are generic positive constants. Then, recalling \eqref{def.vw}, by the condition $\|wf_0\|_{L^\infty}\leq M_0$, it holds that 
\begin{equation*}
\sup_{x\in\Omega} \left|\int_{\R^{3}}|v|^m\sqrt{\mu} f_0(x,v)\,dv\right|\leq M_0\int_{\R^3} \frac{|v|^m\sqrt{\mu}}{w}\,dv\leq C_{\beta,m}\frac{M_0}{\rho^3},
\end{equation*}
which can be small for $M_0>0$ large enough due to \eqref{def.consrho}, where the constant $C_{\beta,m}>0$ depends only on $\beta$ and $m$. Therefore our claim is true by \eqref{rem.idp}. 
%
\end{remark}

\subsection{Difficulties and strategy of the proof}\label{sec.dsp}

Before giving scheme of the proof of Theorem \ref{main thm}, we first  briefly review ideas of \cite{KL} and \cite{DW} in order to understand new difficulties to be overcome. In the $L^{2}$-$L^{\infty}$ bootstrap argument by \cite{Guo}, it is crucial to derive the following nondegeneracy condition
\Be \label{C}
\left \vert \det \left ( \frac{\p}{\p u} X(s^{\prime}; s, X(s;t,x,v), u) \right ) \right \vert \geq \varepsilon > 0,
\Ee
except for some small regions in phase space, when the double Duhamel iteration is applied.  In the case of the specular reflection boundary condition for a general $C^{3}$ convex domain, however, it is hard to obtain \eqref{C} because of the complicated trajectory form. In \cite{KL}, the main idea is to introduce the triple Duhamel iteration for ensuring another kind of the nondegeneracy condition 
\begin{equation} 
\label{ndc.kl}
\left \vert \det \left(\frac{\p}{\p (|u^{\prime}|, \xi_{1}, \xi_{2})} X(s^{\prime\prime}; s^{\prime}, X(s^{\prime};s,X(s;t,x,v),u), u^{\prime}) \right ) \right \vert \geq \varepsilon > 0
\end{equation}
for two distinct variables $
\{ \xi_{1}, \xi_{2} \} $ in $\{ |u|, \hat{u}_{1}, \hat{u}^{\prime}_{1}, \hat{u}^{\prime}_{2} \}
$, where $(\hat{u}_1,\hat{u}_2)$ and $(\hat{u}_1',\hat{u}_2')$ are the spherical coordinates of unit vectors $\hat{u}=\frac{u}{|u|}\in \S^2$ and  $\hat{u}'=\frac{u'}{|u'|}\in \S^2$, respectively. The above nondegeneracy condition \eqref{ndc.kl} can be verified using some idea of the geometric decomposition of particle trajectory, so a unique global mild solution can be constructed  and converges exponentially to zero in large time when initial data $f_0(x,v)$ is small enough in the velocity-weighted $L^{\infty}_{x,v}$ norm. This strongly implies that one may need to perform the Duhamel iteration thrice to make a proper change of variables and bound it by the $L^{p}$ norm of $f$. 

Meanwhile, by \cite{DW}, in case of diffuse reflection boundary condition, for a class of large-amplitude initial data in the sense that they are arbitrarily large in $L^{\infty}_{x,v}$ norm but small  $L^{2}_{x,v}$ norm, a unique global mild solution may exist and decay exponentially to zero in large time. In fact, \cite{DW} borrowed an idea from \cite{DHWY17} in order to bound the nonlinear term by the product of $L^\infty$ norm and an integrability norm; see \cite[Lemma 3.1, Eq.~(3.4), page 386]{DHWY17}. To overcome the velocity growth of the loss term $\Gamma_{-}(f,f)$ in case of hard potentials, $\Gamma_{-}(f,f)$ is added together with the linear local term $\nu(v)f$ to produce a quasilinear relaxation term $\int_{\R^{3}}\int_{\S^2} B(v-u,\omega) F(t,x,u)\, d\omega du\, f(t,x,v)$. A crucial observation is that although it could vanish initially, the integral part in such relaxation term can be uniformly positive after a certain time ${t}_0$ depending only on initial data by proving that  
\Be \label{Rf}
\int_{\R^{3}} e^{-\frac{|v|^{2}}{8}} |f(t,x,v)| dv \leq \text{some generic small constant},\quad \ \ \forall\,t > t_0,\forall\,x\in \Omega.
\Ee
This gives the linear exponential decay along the characteristic after time $t_0$.  Finally, one can obtain the uniform a priori estimate
\begin{equation} \label{Gwall}
\begin{split}
|h(t,x,v)| 
\lesssim  (\text{initial terms}) \big(1 + \int_{0}^{t} \underbrace{\|h(s)\|_{L^{\infty}}}_{(\ast)} ds \big)e^{-\frac{\nu_{0}}{8}t} + ( 1 + \|h\|^{3}_{L^{\infty}}) \mathcal{P}(\|f_{0}\|_{L^{2}})
\end{split}
\end{equation}
for $h=wf$, which gives the uniform bound of $\|h(t)\|_{L^{\infty}}$ over $[0,T_0]$ and the smallness of $\|h(T_0)\|_{L^\infty}$ for a suitably large time $T_0$ depending only on initial data. Since $\|h_0\|_{L^\infty}$ is not small, the smallness  of $\|h(T_0)\|_{L^\infty}$ is a consequence of smallness of $\|f_0\|_{L^2}$. Based on the small-data result in \cite{Guo}, the solution $f(t,x,v)$ can be further extended from $T_0$ to all time $(T_0,\infty)$ and hence tend exponentially to zero in large time; see the Figure \ref{fig1} below.

 \begin{figure}[h]
    \begin{tikzpicture} [scale=0.8]
    \draw[->,thick] (0,0)--(8,0) node[right] {$t$};
    \draw[->,thick] (0,0)--(0,6) node[above] {$\Vert wf(t) \Vert_{L^\infty}$};
    \draw 
    (0,0) node[below] {0}
    (0,1.5) node[left] {$\delta_1$}
    (0,3) node[left] {$M_0$}
    (0,5) node[left] {$\bar{M}=\bar{M}(M_0)$} 
    (5,0) node[below] {$T_0=T_0(M_0,\delta_1)$};
    \fill (0,3)  circle[radius=2pt];
    \draw [thick, red] plot [smooth,tension=0.5] coordinates{ (0,3) (0.3,3.5)  (0.7,3.49) (1,4.1) (1.4,4.1) (1.6,4.6) (2,4.5) (2.5,5) (3,3.51) (3.2,3.5) (3.6,3) (4,1.5) (5,1)} ;
    \draw [thick, blue] plot [smooth,tension=0.5] coordinates{ (5,1) (5.2,0.999) (5.4,1.5) (5.7, 1.8) (6,1.5) (6.3,1.3) (6.6, 1.4) (7,1) (7.5,0.5) (8,0.2)};
     \fill (5,1)  circle[radius=2pt];
    \draw[thick,dashed,black] (0,3) -- (5,3);
    \draw[thick,dashed,black] (0,5) -- (5,5); 
    \draw[thick,dashed,black] (5,0) -- (5,6);
    \draw[thick,dashed,black] (0,1.5) -- (8,1.5); 
    \end{tikzpicture}
     \caption{Time evolution of the norm $\|wf(t)\|_{L^\infty}$.}\label{fig1}
    \end{figure}
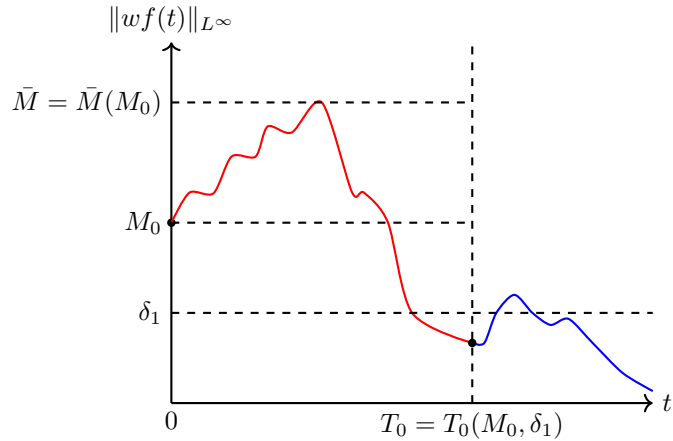
    
\hide
\noindent\textbf{Small amplitude problem in  $C^{3}$ convex domains under specular reflection, [KimLee18]}  
Let us first briefly explain small amplitude boundary problem [Guo10] and [KimLee18]. Fundamental idea of [Guo10] was to use decay of $\|f\|_{L^{2}}$ to obtain decay of $\|wf\|_{L^{\infty}}$. For small amplitude problem, we mainly focus on linearized problem
\[
	\p_{t}f + v\cdot\nabla_{x}f + Lf = 0,
\]
and treat nonlinear $\Gamma(f,f)$ as small perturbation with nonlinear estimate,
$	
	|w(v)\Gamma(f,f)| \lesssim \langle v \rangle \|wf\|_{L^{\infty}}^{2}.	
$
To derive $L^{\infty}$ decay of $h$, double Duhamel expansion along particle trajectory is applied to obtain
\[
	h(t,x,v) \lesssim h_{0} + \int_{0}^{t} \int_{u\in\R^{3}} \int_{u^{\prime}\in\R^{3}} k_{w}(V(s),u)  k_{w}(V(s^{\prime}) ,u^{\prime}) f(s^{\prime}, X(s^{\prime}), u^{\prime}) du du^{\prime}+ \cdots
\]
where $V(s) = V(s;t,x,v)$, $V(s^{\prime}) = V(s^{\prime}; s, X(s;t,x,v), u)$, and $X(s^{\prime}) = X(s^{\prime}; s, X(s;t,x,v), u)$. And key step is to obtain nondegeneracy of particle trajectory
\Be \label{C}
	\left \vert \det \left ( \frac{\p}{\p u} X(s^{\prime}; s, X(s;t,x,v), u) \right ) \right \vert \geq \varepsilon > 0,
\Ee
to apply change of variable $\int_{u\in\R^{3}}$ into $\int_{\O}$. Then, by H\"older, we get
\begin{equation} \label{double}
\begin{split}
h(t,x,v) 
&\lesssim e^{-\nu(v)t}h_{0} + \int_{0}^{t} \|f(s)\|_{L^{2}} ds + \text{small terms}.
\end{split}
\end{equation}
Then using semi-positivity of linearized operator $L$ and coercivity $\|\mathbf{P}f\|_{2} \lesssim \|\mathbf{I-P}f\|_{2}$ (where $\mathbf{P}f$ is hydrodynamics part), we get exponential decay of $\|h(t)\|_{\infty}$. In the case of specular reflection boundary condition with $C^{3}$ convex domain, however, nondegeneracy condition (\ref{C}) is not easy to obtain through double iteration because of complex behavior of trajectory under specular reflection boundary condition. \\

Key idea of [KimLee18] is to apply triple iteration with geometric decomposition of velocity. By decomposing $u$ into $(|u|, \hat{u}_{1}, \hat{u}_{2})$ (speed and directional components), we know that $\frac{\p X(s^{\prime})}{\p|u|}$ yields rank one. Full rank (rank 3) is not easy directly, but through double iteration, we can guarantee rank 2 by choosing $\{|u|, \hat{u}_{1}\}$ WLOG. Inspired by this rank 2 property, we can imagine two dimensional manifold around $X(s^{\prime};s, X(s;t,x,v), u)$ spanned by $\big\{ \p_{|u|}X(s^{\prime}), \p_{\hat{u}_{1}}X(s^{\prime})\big\}$. Applying one more iteration step, we consider nondegeneracy
\Be \label{triple_C}
\left \vert \det \left ( \frac{\p}{\p u^{\prime}} X(s^{\prime\prime}; s^{\prime}, X(s^{\prime};s,X(s;t,x,v),u), u^{\prime}) \right ) \right \vert \geq \varepsilon > 0.
\Ee
Speed component $\p_{|u^{\prime}|}$ gives rank 1 obviously. Among $\{ \p_{|u|}X(s^{\prime}), \p_{\hat{u}_{1}}X(s^{\prime}), \hat{u}^{\prime}_{1}, \hat{u}^{\prime}_{2} \}$, we can show that there exist 
$
	\{ \xi_{1}, \xi_{2} \} \subset \{ \p_{|u|}X(s^{\prime}), \p_{\hat{u}_{1}}X(s^{\prime}), \hat{u}^{\prime}_{1}, \hat{u}^{\prime}_{2} \}
$
such that 
\Be 
\left \vert \det \left(\frac{\p}{\p (|u^{\prime}|, \xi_{1}, \xi_{2})} X(s^{\prime\prime}; s^{\prime}, X(s^{\prime};s,X(s;t,x,v),u), u^{\prime}) \right ) \right \vert \geq \varepsilon > 0.
\Ee
By rank 2 argument choosing,  $\p_{|u|}X(s^{\prime}), \p_{\hat{u}_{1}}X(s^{\prime})$ is equivalent to choose $|u|$, $\hat{u}_{1}$ and hence we can similar estimate as (\ref{double}) with triple velocity integration $\int_{u^{\prime\prime}} \int_{u^{\prime}} \int_{u}$ which comes from triple iteration. \\

\noindent\textbf{Large amplitude problem under diffuse reflection, [DuanWang19]} 
Difficulty of large amplitude solution is quite different to small amplitude solution since we cannot treat the problem as perturbation of linearized Boltzmann equation. In particular, the nonlinear term $\Gamma(f,f)$ cannot be small around initial time.  \\

In [DuanWang19], they borrowed an idea from \cite{DHWY17} in order to bound the nonlinear term by the product of $L^\infty$ norm and an integrability norm; see \cite[Lemma 3.1, Eq.~(3.4), page 386]{DHWY17}. However, in case of hard potentials, in particular, $\kappa >0$, the estimate like  \cite[Eq.~(3.4), page 386]{DHWY17} prevents us using the new Gronwall argument to close the {\it a priori} estimates, because the right-hand terms of \cite[Eq.~(3.4), page 386]{DHWY17} contain an additional factor $\nu(v)\sim (1+|v|)^\kappa$ that may grow in velocity. Fortunately, for the gain term, we may use \cite[Lemma 2.2, Eq.~(2.11), page 52]{Duan} to replace the first estimate of \cite[Eq.~(3.4), page 386]{DHWY17}, but the second estimate of \cite[Eq.~(3.4), page 386]{DHWY17} for the loss term can not be refined such that the pointwise upper bound is bounded with respect to velocity variables. This then forces \cite{Duan} to make use of the quasilinear idea to treat the loss term, namely.  \\
\indent Because of above reason, $w(v)\Gamma_{-}(f,f)$ is conbined with collision frequency part $\nu(v)h$ to obtain $h R(f) := \nu(v)h - w(v)\Gamma_{-}(f,f) = h \iint_{\S^{2}\times\R^{3}} B(v-u,\omega) F(t,x,u) d\sigma du$. To obtain uniform decaying property with $h \mathcal{A}F$, we claim
\Be \label{Rf}
	\int_{\R^{3}} e^{-\frac{|v|^{2}}{8}} |f(t,x,v)| dv \leq \text{some generic constant}
\Ee
which guarantee uniform positivity of $R(f)(t,x,v)$ for $t \geq \tilde{t}$ for some $\tilde{t}_{f_{0}}$, i.e., uniformly decaying mechanism along particle trajectory after some time interval. We apply trajectory expansion once to (\ref{Rf}) and perform change of variable $\int_{\R^{3}}$ into $\int_{\O}$. For diffuse reflection case, we keep track particle trajectory until its first bounce (backward in time) which makes proving (\ref{C}) very easy. However, special treatment for boundary integral which comes from diffuse reflection condition is required. We refer [Guo10] and [DuanWang19] for detail. Some important idea is another way of treating $w(v)\Gamma_{\pm}(f,f)$. In particular, we can estimate $\Gamma_{+}$ as
\Be \label{Gamma_RJ}
	w\Gamma_{+}(f,f) \lesssim \|h\|_{L^{\infty}} \Big\{ \int_{u} k(u) h^{2}(t,x,u) du  \Big\}^{\frac{1}{2}},
\Ee
using Carlemann representation. Here, kernel $k$ has very good decay without singularity. Roughly we have double velocity integration and performing change of variable is available. Assuming smallness of $\|f_{0}\|_{L^{2}}$ in addition, we can derive proper control like (\ref{Rf}).  \\

Then, to derive apriori estimate, we use (\ref{Gamma_RJ}), change of variable with double Duhamel expansion, and smallness of $\|f_{0}\|_{L^{2}}$. For $t \geq \tilde{t}$, with uniform positivity of $R(f)$, we estimate (we ignore $K_{w}$ here)
\begin{equation} \label{Gwall}
\begin{split}
	h(t,x,v) &\lesssim e^{-\frac{\nu_{0}}{2}t}h_{0} + \int_{0}^{t} e^{-\frac{\nu_{0}}{2}(t-s)} \|h\|_{L^{\infty}}  \\
	&\quad \times 
	\Big\{ \int_{\R^{3}}k(u)
	\Big[ h^{2}_{0} + \int_{0}^{s} e^{-\frac{\nu_{0}}{2}(s-s^{\prime})} \|h(s^{\prime})\|^{2}_{L^{\infty}}\int_{\R^{3}} k(u^{\prime}) h^{2}(s^{\prime}, X(s^{\prime}), u^{\prime}) du^{\prime} \Big] du \Big\}^{\frac{1}{2}} \\
	&\lesssim  (\text{initial terms}) \big(1 + \int_{0}^{t} \underbrace{\|h(s)\|_{L^{\infty} } }_{(*)} ds \big)e^{-\frac{\nu_{0}}{8}t} + ( 1 + \|h\|^{3}_{L^{\infty}}) \mathcal{P}(\|f_{0}\|_{L^{2}}).
\end{split}
\end{equation}
Finally we derive uniform decay of $\|h(t)\|_{L^{\infty}}$ from Gronwall type inequality and bootstrap argument.  \\
\unhide

Now let us explain the main scheme of the proof of Theorem \ref{main thm}. This is inspired by the above two results. Moreover, it contains a new idea of controlling several $\Gamma_{+}(f,f)$ terms that appear in each iteration step in the exact order. 

\medskip
\noindent\textbf{\textit{Scheme of proof.}} As in \cite{DW} using smallness of $L^2$ norm, we will use smallness of the relative entropy to control the nonlinear term along the pointwise dynamics in case when $\|h(t)\|_{\infty}$ is no longer small. We note that the relative entropy can come out once we have the double integrations in phase space as explained in Lemma \ref{entropy_est}. In fact, as a viewpoint of the Duhamel expansion, the integration on $\O\times \R^{3}_{v}$ \textit{formally} comes from the double Duhamel iteration and nondegeneracy of the mapping $v\mapsto X(s;t,x,v)$. 

In the case of a general $C^{3}$ convex domain with specular reflection boundary condition, however, it is very hard to perform change of variables via the nondegeneracy of $v\mapsto X(s;t,x,v)$, cf.~\cite{Guo,KL}. From the idea of \cite{KL}, we need the triple Duhamel iterations, instead of the double iteration, to guarantee the nondegeneracy \eqref{ndc.kl} so as to perform change of variables. Then, the integral will be bounded by the relative entropy $\mathcal{E}(F(t))$ which can be further bounded by its initial value $\mathcal{E}(F_{0})$ and hence the smallness of $\mathcal{E}(F_{0})$ can be used to control  growth of the nonlinear term. Unfortunately, one more iteration process generates many difficulties. To understand the issue, as in \cite{DW} to treat the large amplitude solutions for hard potentials case, one has to write the Boltzmann equation \eqref{def.be} as
\Be \label{Rf_expan}
\p_{t}h + v\cdot\nabla_{x}h + R(F)h = K_{w}h + w\Gamma_{+}(f,f), 
\Ee
where $R(F) := \int_{\R^{3}}\int_{\S^2} B(v-u,\omega) F(t,x,u) \,d\omega du$ and $K_{w}h := wK\frac{h}{w}$. For the moment, we assume
\Be \label{Rf_apriori}
R(F)(t,x,v) \gtrsim 1,
\Ee
in order to guarantee the linear decay structure for performing the triple Duhamel iterations in terms of \eqref{Rf_expan}. Then we will see the appearance of a nonlinear term $\|h(s)\|^{2}_{L^{\infty}}$ in the same place of the underbraced $(\ast)$ in \eqref{Gwall}. Hence, even if a proper uniform lower bound of $R(F)$ as in \eqref{Rf_apriori} could hold true, one cannot derive the proper a priori estimate as in \eqref{Gwall} such that the Gronwall argument can be employed. 

Let us see how the above bad term comes out more specifically. In fact, it is induced from the combination of $\|h(s)\|_{L^{\infty}}$ and initial data $h_{0}$ which appears whenever we proceed extra iteration steps. For instance, if one applies the only nonlinear bound
\Be \label{Gamma_RJ}
	w\Gamma_{+}(f,f) \lesssim \|h\|_{L^{\infty}} \Big\{ \int_{u} k(u) h^{2}(t,x,u) du  \Big\}^{\frac{1}{2}},
\Ee
then one should get a higher order term $\|h\|^{2}_{L^{\infty}}$ after making iterations as follows
\Be \label{Fail}
\begin{split}
	w\Gamma_{+}(f,f) 
	&\lesssim  \|h\|_{L^{\infty}} \Big\{ h_{0}^{2} + \int_{0}^{t}\int_{u} \int_{u^{\prime}} k(u) k(u^{\prime}) h^{2} du^{\prime} du  \Big\}^{\frac{1}{2}} + \cdots \\
	&\lesssim  \|h\|_{L^{\infty}} h_{0} + \|h\|_{L^{\infty}} \Big\{ \int_{0}^{t}\int_{u} \int_{u^{\prime}} k(u) k(u^{\prime}) h^{2} du^{\prime} du  \Big\}^{\frac{1}{2}}  + \cdots \\
	&\lesssim   \|h\|_{L^{\infty}} h_{0}  + \underbrace{ \|h\|^{2}_{L^{\infty}} h_{0} }_{\text{higher order}}+ \|h\|_{L^{\infty}} \Big\{ \int_{0}^{t}\int_{u} \int_{u^{\prime}}  \int_{u^{\prime\prime}} k(u) k(u^{\prime}) k(u^{\prime\prime}) h^{2} du^{\prime\prime} du^{\prime} du  \Big\}^{\frac{1}{2}} + \cdots . \\
\end{split}  
\Ee

To resolve the trouble, we shall mix two different ways of treating $\Gamma_{+}$. Indeed, in Lemma \ref{Gamma_new}, we introduce the following estimate 
\begin{equation} \label{Gamma_N}
\abs{ w(v) \Gamma_+ (f,f) (v)}  \leq C \int_{\R^3} \abs{\tilde{k}_2(\eta,v)} \; \abs{h(\eta)}^2 d \eta
\end{equation}
where $\tilde{k}_{2}(\eta, v)$ is integrable and has singularity $|\eta-v|^{-1}$ 
in the hard potential case when $0 \leq \kappa \leq 1$. To prove \eqref{Gamma_N}, the exponential part with the strictly positive $\varpi>0$ in the velocity weight \eqref{def.vw} is crucially necessary to be used.   Compared with (\ref{Gamma_RJ}), the estimate (\ref{Gamma_N}) does not generate any extra large factor $\|h\|_{L^{\infty}}$ but only the strength of singularity in the kernel will be doubled by $|h(\eta)|^{2}$ in the next iteration. For each of three Duhamel iterations, we apply $(\ref{Gamma_RJ})$-$(\ref{Gamma_N})$-$(\ref{Gamma_RJ})$ for $\Gamma_{+}$ estimates in order. We note that this order is very sharp and cannot be changed into another order.

\hide
We also note that for $w(v)\Gamma_{-}(f,f)$, there is no way to avoid $\|h(t)\|_{\infty}$ because $\Gamma_{-}(f,f)$ contains local term. It will be sufficiently large in general near initial time at least. Hence in expansion like (\ref{Fail}), it is also very clear why we should combine $\Gamma_{-}(f,f)$ with $\nu(v)f$. \\
\unhide

To simplify explanation, let us ignore $K_{w}h$ and time integration for simplicity, and focus on large factor $\|h\|_{L^{\infty}}$ and velocity integration because linearized operator $K_{w}h := wKf$ does not generate any large factor. Also, the first expansion with (\ref{Gamma_RJ}) informally gives that 
\[
h(t,x,v) \sim h_{0} + \|h\|_{L^{\infty}} \Big\{ \int_{u} k(u) h^{2}(u) du \Big\}^{\frac{1}{2}}. 
\]
Here, we have ignored time integration for simplicity. Using (\ref{Gamma_N}) for the second iteration, we obtain
\begin{equation*}
\begin{split}
h(t,x,v) \sim h_{0} + \|h\|_{L^{\infty}}h_{0} + \|h\|_{L^{\infty}}  \Big\{ \int_{u} k(u) \Big[ \int_{u^{\prime}} \tilde{k}(u,u^{\prime}) h^{2}(u^{\prime})\Big]^{2} \Big\}^{\frac{1}{2}}
\end{split}
\end{equation*} 
to get the low order of large factor in underbraced term $(*)$ unlike to (\ref{Fail}). Since we have only the double velocity integration, we need one more iteration step to adopt technique and idea of (\ref{ndc.kl}) in \cite{KL}. Applying  (\ref{Gamma_RJ}) again, it gives that
\begin{equation*}
\begin{split}
h(t,x,v) &\sim h_{0} + \|h\|_{L^{\infty}}h_{0} + \|h\|_{L^{\infty}} \Big\{ \int_{u} k(u) \Big[ \int_{u^{\prime}} \tilde{k}(u,u^{\prime}) \big[ \|h\|_{L^{\infty}}^{2} \int_{u^{\prime\prime}} k(u^{\prime\prime}) h^{2}(u^{\prime\prime}) \big] \Big]^{2} \Big\}^{\frac{1}{2}} \\
&\sim h_{0} + \|h\|_{L^{\infty}}h_{0} + \|h\|^{3}_{L^{\infty}} 
\Big\{ \underbrace{ \int_{u} k(u) \int_{u^{\prime}} \tilde{k}^{2}(u,u^{\prime}) \int_{u^{\prime\prime}} k^{2}(u^{\prime\prime}) h^{4}(u^{\prime\prime}) }_{(**)}  \Big\}^{\frac{1}{2}}.
\end{split}
\end{equation*} 
For the underbraced term above, we note that singularity of $\tilde{k}^{2}(u, u^{\prime})$ is still integrable near $u^{\prime}=u$ for all $0\leq \kappa \leq 1$. Hence using $L^{1}$ and $L^{2}$ control by relative entropy from Lemma \ref{entropy_est}, one has  
\[
(**) \sim \|h\|_{L^{\infty}} \int_{u} k(u) \int_{u^{\prime}} \tilde{k}^{2}(u,u^{\prime}) \int_{u^{\prime\prime}} k^{2}(u^{\prime\prime}) h^{2}(u^{\prime\prime}) \lesssim \|h\|_{L^{\infty}}\mathcal{E}(F_{0}) + \text{other small terms}.
\]
Using the above estimate, we then derive the desired crucial a priori estimate similar to (\ref{Gwall}) to continue with the bootstrap argument. Hence, as explained in Figure \ref{fig1}, the global existence and large time behavior of solutions follow basing on the small-amplitude result \cite{KL}.

Now, let us return to the assumption (\ref{Rf_apriori}). To justify it, the key point is to use the velocity weight $w=w_{\rho}(v)$ in \eqref{def.vw} where the parameter $\rho$ will be chosen in terms of $M_0$ that is the upper bound of $\|h_0\|_{L^\infty}$. In fact, under the {\it a priori} assumption that $\|h(s)\|_{L^{\infty}}\leq \bar{M}$ where $\bar{M}$ is given in \eqref{def.bM} depending only on $M_0$ but not on $\rho$, it holds that 
\Be \label{Rf_simple}
\int_{\R^{3}} e^{-\frac{|v|^{2}}{8}} f(t,x,v) dv \leq \|h\|_{\infty} \int_{\R^{3}} e^{-\frac{|v|^{2}}{8}} \frac{1}{(1 + \rho^{2}|v|^{2})^{\b}}dv \lesssim \frac{\bar{M}}{\rho^{3}} \lesssim \text{a generic small constant},
\Ee
provided that $\rho$ is suitably large relative to $M_0$ only. Then, the above estimate guarantees \eqref{Rf_apriori}. We notice that there is an obvious difference between \eqref{Rf_simple} and \eqref{Rf}. In fact, as mentioned before, \eqref{Rf} holds true only for large enough time $t>t_0$ because the integral on the left-hand side of \eqref{Rf} can be large around initial time, while \eqref{Rf_simple} holds true uniformly in time starting from $t=0$. Notice that it is crucial to establish the local-in-time $L^{\infty}$ estimate that has nothing to do with the parameter $\rho$ in the case of specular reflection boundary condition, and hence the a priori upper bound $\bar{M}$ of $\|h(t)\|_{\infty}$ can be chosen independent of $\rho$ whereas $\bar{M}$ depends on $\rho$ in the case of diffuse reflection boundary condition case \cite{DW}. In the end we also remark that it is unclear for us whether one can remove the restrictive dependence of $\rho$ on $M_0$ by carrying out a similar idea in \cite{DW} due to the use of triple iterations for the $C^3$ boundary regularity.    


\hide

We end this section with several remarks.

\begin{remark}
	In the case of specular reflection boundary condition, the parameter $\rho$ has nothing to do with the local-in-time $L^{\infty}$ estimate. Hence, the a priori bound $\|h(t)\|_{\infty} \leq \bar{M}$ can be chosen independent of $\rho$ whereas $\bar{M}$ depends on $\rho$ in the case of diffuse reflection boundary condition case \cite{DW}. Since $\bar{M}$ can be determined only by $M_0$,  choosing $\rho$ to depend on $M_0$ does not give rise to any contradiction.
\end{remark}
\begin{remark}
	Let us treat $\int_{\R^{3}} e^{-\frac{|v|^{2}}{8}}$ as like an integral operator on the LHS of \eqref{Rf_simple}. Then, in the case of [DuanWang19] only one iteration yields double velocity integration and hence corresponding change of variable yield smallness. Combination of $\bar{M}$ and $h_0$ does not happen in that case. Therefore, we can control $R(f)$ as generic small using time decaying factor, say $e^{-\nu_0 t}$, after some time $t > \tilde{t}$. However, in the triple iteration case, we still need two more iterations in \eqref{Rf_simple}. So the bad combination of $\|h\|_{\infty}$ and $h_0$ happens again. This cannot be controlled by generic small factor. This is the main reason why we have chosen large $\rho$ depending on $M_0$, in weight $w_{\rho}(v)$.
\end{remark}
\begin{remark}
	If we consider local mass perturbation $\int \sqrt{\mu}f dv = \int F dv - \int \mu dv$, we know that mass perturbation,
	\[
	\int \sqrt{\mu}f dv \leq \int_{\R^{3}} e^{-\frac{|v|^{2}}{8}} f(t,x,v) dv.
	\]
	From \eqref{Rf_simple}, it must hold from initial time $t=0$. By some computation, initial data will satisfy
	\[
	\sup_{x} \int_{\R^{3}} e^{-\frac{|v|^{2}}{8}} f_0(x,v) dv \leq C \frac{\|h_0\|_{\infty}}{\bar{M}} \leq \text{small generic constant}.
	\]
	Since $\bar{M}$ is determined by only on $\|h_0\|_{\infty}$ and has faster growth than linear growth, local mass perturbation goes to zero as $\|h_0\|_{\infty}$ goes to infinity. Therefore, for very large initial data $\|h_0\|_{\infty}$, local mass $\int F_0(x,v) dv$ is very close to $\int \mu(v) dv$ and hence initial local vacuum is not allowed in general.
\end{remark}

\textbf{\textit{Scheme of proof}} First, from the idea of [KimLee18], we need triple velocity integration to perform change of variable and then it will be  bounded by relative entropy $\mathcal{E}(F(t))$ which monotonically decreases from initial small relative entropy, i.e., $\mathcal{E}(F(t)) < \mathcal{E}(F_{0}) \ll 1$. However, unlike small amplitude problem, one more iteration process generates large factor $\|h(s)\|^{2}_{L^{\infty}}$ in underbraced $(*)$ in (\ref{Gwall}) by (\ref{Gamma_RJ}). Hence, although we assume proper $R(f)$ estimate, we cannot derive decay in apriori estimate.   \\

Main difficulty comes from combination of $\|h(s)\|_{L^{\infty}}$ and initial data $h_{0}$ which appear whenever we add iteration step and does not contain any small factor. To see this effect, let us apply iteration for the integrand of (\ref{Gamma_RJ}). Then whenever we apply iteration we get higher order of $\|h\|_{L^{\infty}}$,
\Be \label{Fail}
\begin{split}
w\Gamma_{+}(f,f) 
&\lesssim  \|h\|_{L^{\infty}} \Big\{ h_{0}^{2} + \int_{0}^{t}\int_{u} \int_{u^{\prime}} k(u) k(u^{\prime}) h^{2} du^{\prime} du  \Big\}^{\frac{1}{2}} + \cdots \\
&\lesssim  \underbrace{ \|h\|_{L^{\infty}} h_{0} }_{\text{large factor}} + \|h\|_{L^{\infty}} \Big\{ \int_{0}^{t}\int_{u} \int_{u^{\prime}} k(u) k(u^{\prime}) h^{2} du^{\prime} du  \Big\}^{\frac{1}{2}}  + \cdots \\
&\lesssim   \|h\|_{L^{\infty}} h_{0}  + \underbrace{ \|h\|^{2}_{L^{\infty}} h_{0} }_{\text{higher order}}+ \|h\|_{L^{\infty}} \Big\{ \int_{0}^{t}\int_{u} \int_{u^{\prime}}  \int_{u^{\prime\prime}} k(u) k(u^{\prime}) k(u^{\prime\prime}) h^{2} du^{\prime\prime} du^{\prime} du  \Big\}^{\frac{1}{2}} + \cdots  \\
\end{split}  
\Ee

To resolve this problem we mix two different ways of treating $\Gamma_{+}$. In Lemma \ref{Gamma_new}, we estimate $w\Gamma_{+}(f,f)$ as 
\begin{equation} \label{Gamma_N}
	\abs{ w(v) \Gamma_+ (f,f) (v)}  \leq C \int_{\R^3} \abs{\tilde{k}_2(\eta,v)} \; \abs{h(\eta)}^2 d \eta
\end{equation}
where $\tilde{k}_{2}(\eta, v)$ is integrable and has singularity $\frac{1}{|\eta-v|}$ in the case of hard potential $0 \leq \kappa \leq 1$. Compared with (\ref{Gamma_RJ}), (\ref{Gamma_N}) does not generate extra large factor $\|h\|_{L^{\infty}}$ but strength of singularity will be stronger after iteration by $|h(\eta)|^{2}$ in integrand. For triple Duhamel iteration, we apply $(\ref{Gamma_RJ})$-$(\ref{Gamma_N})$-$(\ref{Gamma_RJ})$ for $\Gamma_{+}$ in order in each steps. We note that this order is very sharp and cannot be changed into the another order. \\
We also note that for $w(v)\Gamma_{-}(f,f)$, there is no way to avoid $\|h(t)\|_{\infty}$ because $\Gamma_{-}(f,f)$ contains local term. It will be sufficiently large in general near initial time at least. Hence in expansion like (\ref{Fail}), it is also very clear why we should combine $\Gamma_{-}(f,f)$ with $\nu(v)f$. \\

 Now, let us assume (\ref{Rf}) so that $R(f) \geq \frac{\nu(v)}{2}$ and trajectory expansion gives proper exponential decay factor. Now, to simplify explanation, let us ignore $K_{w}h$ and time integration, and focus on large factor $\|h\|_{L^{\infty}}$ and velocity integration because linearized operator $K_{w}h := wKf$ does not generate any large factor. Also,The first expansion with (\ref{Gamma_RJ}) gives
\[
	h(t,x,v) \sim h_{0} + \|h\|_{L^{\infty}} \Big\{ \int_{u} k(u) h^{2}(u) du \Big\}^{\frac{1}{2}} 
\]
Using (\ref{Gamma_N}) for the second iteration, we obtain
\begin{equation*}
\begin{split}
	h(t,x,v) \sim h_{0} + \|h\|_{L^{\infty}}h_{0} + \underbrace{\|h\|_{L^{\infty}} } \Big\{ \int_{u} k(u) \Big[ \int_{u^{\prime}} \tilde{k}(u,u^{\prime}) h^{2}(u^{\prime})\Big]^{2} \Big\}^{\frac{1}{2}}
\end{split}
\end{equation*} 
to get low order of large factor in underbraced term $(*)$ unlike to (\ref{Fail}). Since we have only double velocity integration, we need more iteration step to adopt technique and idea of (\ref{triple_C}) in [KimLee18]. Applying  (\ref{Gamma_RJ}) again, it gives
\begin{equation} \label{apriori_scheme}
\begin{split}
h(t,x,v) &\sim h_{0} + \|h\|_{L^{\infty}}h_{0} + \|h\|_{L^{\infty}} \Big\{ \int_{u} k(u) \Big[ \int_{u^{\prime}} \tilde{k}(u,u^{\prime}) \big[ \|h\|_{L^{\infty}}^{2} \int_{u^{\prime\prime}} k(u^{\prime\prime}) h^{2}(u^{\prime\prime}) \big] \Big]^{2} \Big\}^{\frac{1}{2}} \\
&\sim h_{0} + \|h\|_{L^{\infty}}h_{0} + \|h\|^{3}_{L^{\infty}} 
\Big\{ \underbrace{ \int_{u} k(u) \int_{u^{\prime}} \tilde{k}^{2}(u,u^{\prime}) \int_{u^{\prime\prime}} k^{2}(u^{\prime\prime}) h^{4}(u^{\prime\prime}) }_{(***)}  \Big\}^{\frac{1}{2}} \\
\end{split}
\end{equation} 
For the underbraced term above, we note that singularity of $\tilde{k}^{2}(u, u^{\prime})$ is still integrable near $u^{\prime}=u$ for all $0\leq \kappa \leq 1$. Hence using $L^{1}$ and $L^{2}$ control by relative entropy (See Lemma \ref{entropy_est}),  
\[
	(***) \sim \|h\|_{L^{\infty}} \int_{u} k(u) \int_{u^{\prime}} \tilde{k}^{2}(u,u^{\prime}) \int_{u^{\prime\prime}} k^{2}(u^{\prime\prime}) h^{2}(u^{\prime\prime}) \lesssim \|h\|_{L^{\infty}}\mathcal{E}(F_{0}) + \text{some small terms}.
\]
Using above estimate, we derive apriori estimate simlar as (\ref{Gwall}) to continue with bootstrap argument.  \\

Now, let us return to our assumption (\ref{Rf}), uniform positivity of $R(f)$. Our simple observation is that we can use $\int_{\R^{3}} e^{-\frac{|v|^{2}}{8}}$ in (\ref{Rf}) as first iteration. Moreover, we can treat $e^{-\frac{|v|^{2}}{8}}$ as $k(v)$ which has very nice properties. Therefore, only double iteration for $|h(t,x,v)|$ in integrand is enough to derive triple iteration so that we can perform change of variable. Main difference to (\ref{apriori_scheme}) is that we do not combine $\Gamma_{-}$ with collision frequency to use exponential decay property. For $\Gamma_{-}$, we may use
\Be \label{Gamma_-}
	w(v)\Gamma_{-}(f,f) \lesssim \|h\|_{L^{\infty}} \int_{u} \overline{k}(v,u) h(u) du
\Ee
which is similar to (\ref{Gamma_RJ}) where kernel $\overline{k}$ is integrable and has no singularity. Or it is also possible to use
\Be \label{Gamma_-_v}
	w(v)\Gamma_{-}(f,f) \lesssim \|h\|_{L^{\infty}} \langle v \rangle^{\kappa} \int_{u} e^{-c|u|^{2}} f(u) du,\quad c> 0,
\Ee
since $\langle v \rangle^{\kappa}$ is not harmful anymore by extra decay outside, e.g. $e^{-\frac{1}{8}|v|^{2}}$ in the LHS of (\ref{Rf}). \\
 
 We first expand $h$ with (\ref{Gamma_N}) and (\ref{Gamma_-}). For second iteration, we now use (\ref{Gamma_RJ}) and (\ref{Gamma_-}) as we used (\ref{Gamma_RJ}) in third iteration of (\ref{apriori_scheme}). As a last step, applying $\int_{\R^{3}} e^{-\frac{|v|^{2}}{8}}$ and performing change of variable with triple velocity integration, we can derive a version of (\ref{apriori_scheme}) with $\int_{\R^{3}} e^{-\frac{|v|^{2}}{8}} |h(t,\cdot,v)|_{L^{\infty}_{x}} dv$, i.e., 
\begin{equation} \label{Rf_new}
\begin{split}
\int_{\R^{3}} e^{-\frac{|v|^{2}}{8}} |h(t,\cdot,v)|_{L^{\infty}_{x}} dv 
&\lesssim  (\text{initial terms}) \big(1 + \int_{0}^{t} \Big[ \int_{\R^{3}} e^{-\frac{|v|^{2}}{8}} |h(s,\cdot,v)|_{L^{\infty}_{x}} dv \Big] ds \big)e^{-\frac{\nu_{0}}{2}t} + \text{small terms}.
\end{split}
\end{equation}
Hence, by Gronwall's inequality, (\ref{Rf}) is justified after some time $\tilde{t}$.  \\
\unhide

\subsection{Organization of the paper}
The rest of this paper is organized as follows. In the next Section \ref{sec2}, we make preliminary preparations on the estimates for the operator $K$, the nonlinear gain term $\Gamma_+(f,f)$, the relative entropy, and the nondegeneracy of specular trajectory. In the main part Section \ref{sec3}, we proceed obtaining series of the a priori estimate in order to get the key inequality of the form \eqref{Gwall}. Then, in Section \ref{sec4}, Theorem \ref{main thm} is proved by the Gronwall argument as in \cite{DW} together with the global existence result of the small-amplitude global solution in \cite{KL}. In Section \ref{sec5}, we give an appendix on the local-in-time existence of solutions for the completeness of the proof.    

\section{Preliminaries}\label{sec2}
Define the velocity weight function 
\begin{equation} \label{weight}
	w=w(v):=(1+\rho^2\abs{v}^2)^{\beta}e^{\varpi \abs{v}^2},
\end{equation} 
with $\rho>0$, $\varpi >0$ and $\beta>0$. Here, the strict positivity of $\varpi$ will be only used in Lemma \ref{Gamma_new} that is crucial for the iteration procedure we shall introduce in this paper.

\subsection{{Operator $K$}}
Recall the definition of operator $K$ by
\begin{equation*}
	(Kf)(v)= \int_{\R^3} k(v,\eta) f(\eta)\, d\eta, 
\end{equation*} 
where $k(v,\eta)=k(\eta,v)$ is a symmetric integral kernel of $K$. 
\begin{lemma} \label{Lemma_k}
\cite{Guo} There is $C>0$ such that 
\begin{equation} \label{k_esti.0}
	\vert k(v,\eta)\vert \leq C\{ \vert v-\eta \vert + \vert v-\eta \vert^{-1} \} e ^{-\frac{\vert v- \eta \vert^2}{8}}e^{-\frac{1}{8}\frac{\vert \vert v \vert^2 - \vert \eta \vert ^2 \vert^2}{\vert v - \eta \vert^2} },
\end{equation} 
for any $v,\eta \in R^3$ with $v\neq \eta$. Recall $w$ in \eqref{weight}. Let $0< \varpi <\frac{1}{4}$. Then, there exist $0\leq \epsilon(\varpi)<1 $ and $C_{\varpi}>0$ such that for $0\leq \epsilon < \epsilon(\varpi)$, 
\begin{equation} \label{k_esti.1}
	\int_{\R^3} \{ \vert v-\eta \vert + \vert v -\eta \vert ^{-1}  \}e^{-\frac{1-\epsilon}{8} \vert v-\eta \vert^2 -\frac{1-\epsilon}{8}\frac{\vert \vert v \vert^2 -\vert \eta \vert^2 \vert^2 }{\vert v-\eta \vert^2}} \frac{w(v)}{w(\eta)} \, 
	d\eta \leq \frac{C_\varpi}{1+\vert v\vert}, 
\end{equation} 
for any $v\in \R^3$.
\end{lemma}

\begin{lemma} \label{Lemma_nega}
\cite{DW}
There is $C>0$ such that 
\begin{equation*} \label{k_esti.2}
	\vert k(v,\eta)\vert \leq C\{ \vert v-\eta \vert + \vert v-\eta \vert^{-1} \} e ^{-\frac{\vert v- \eta \vert^2}{8}}e^{-\frac{1}{8}\frac{\vert \vert v \vert^2 - \vert \eta \vert ^2 \vert^2}{\vert v - \eta \vert^2} },
\end{equation*} 
for any $v,\eta \in \R^3$ with $v\neq \eta$. Moreover, let $\, 0<\varpi \leq \frac{1}{64}, \, \alpha \geq 0$. There exists $C_{\beta,\varpi}>0$  such that  
\begin{equation} \label{k_esti.3}
	\int_{\R^3} \{ \vert v-\eta \vert + \vert v -\eta \vert ^{-1}  \}e^{- \frac{\vert v-\eta \vert^2}{16}}e^{ -\frac{\vert \vert v \vert^2 -\vert \eta \vert^2 \vert^2 }{16\vert v-\eta \vert^2}} \frac{w(v)}{w(\eta)} \, d\eta \leq 
	\frac{C_{\beta,\varpi}}{1+\vert v\vert}. 
\end{equation} 
\end{lemma}
\begin{proof}
We can see \eqref{k_esti.3} is exactly an inequality in the case of $\epsilon =1/2 $ in \eqref{k_esti.1}. Thus, we will check  the exact form of $\epsilon (\varpi)$. We first notice that for some $C_{\rho,\beta}>0$, 
\begin{equation*}
	\left \vert \frac{w(v)}{w(\eta)}\right \vert \leq C_{\rho,\beta} [1+\abs{v-\eta}^2]^{\abs{\beta}}e^{-\varpi \{ \abs{\eta}^2-\abs{v}^2 \}}.
\end{equation*}
Let $v-\eta = u$ and $\eta =v-u$ in the integral of \eqref{k_esti.1}. We now compute the toatal exponent in \eqref{k_esti.1}. 
\begin{align*}
	-&\frac{1-\epsilon}{8} \abs{u}^2 -\frac{1-\epsilon}{8}\frac{\abs{\abs{u}^2-2v\cdot u}^2}{\abs{u}^2} -\varpi \{\abs{v-u}^2-\abs v^2\}\\
	&=-\frac{1-\epsilon}{4}\abs{u}^2 + \frac{1-\epsilon}{2}v \cdot u -\frac{1-\epsilon}{2} \frac{\abs{v\cdot u}^2}{\abs{u}^2} - \varpi\{\abs{u}^2 - 2 v \cdot u\} \\ 
	&=(-\varpi-\frac{1-\epsilon}{4})\abs{u}^2 +(\frac{1-\epsilon}{2}+2\varpi)v\cdot u -\frac{1-\epsilon}{2}\frac{\{v\cdot u\}^2}{\abs{u}^2}. 
\end{align*}
The discriminant of the above quadratic form of $\abs{u}$ and $\frac{v\cdot u}{\abs{u}}$ is 
\begin{align*}
	\Delta = (\frac{1-\epsilon}{2}+2\varpi)^2 +2(-\varpi-\frac{1-\epsilon}{4})(1-\epsilon)=(\frac{1-\epsilon}{2}+2\varpi)(\frac{1-\epsilon}{2}+2\varpi+\epsilon-1).
\end{align*}
Hence, the quadratic form is negative definite for $0\leq \epsilon <1-4\varpi$. Actually, the exact form of $\epsilon (\varpi)$ is $1-4\varpi$. Since $0\leq \varpi \leq \frac{1}{64}$, we have $\frac{15}{16} \leq \epsilon(\varpi) \leq 1$ and ,for $\epsilon =\frac{1}{2}$, that there is $C_{\varpi} >0$ such that the following quadratic form is negative definite 
\begin{align} \label{n_definite}
	-&\frac{1}{16} \abs{u}^2 -\frac{1}{16}\frac{\abs{\abs{u}^2-2v\cdot u }^2}{\abs{u}^2} - \varpi \{ \abs{u}^2 -2v\cdot u\} \nonumber \\ 
	&\leq - C_{\varpi} \bigg\{ \abs{u}^2+\frac{\abs{v\cdot u}^2}{\abs{u}^2}\bigg\}=- C_{\varpi} \bigg\{ \frac{\abs{u}^2}{2}+\bigg(\frac{\abs{u}^2}{2}+\frac{\abs{v\cdot u}^2}{\abs{u}^2}\bigg)\bigg\} \nonumber \\
	&\leq - C_{\varpi} \bigg\{\frac{\abs{u}^2}{2}+\abs{v\cdot u} \bigg\}. 
\end{align}
For $\abs{v}\geq 1$, we make change of variable $u_{\parallel}=\{u\cdot \frac{v}{\abs{v}}\}\frac{v}{\abs{v}} $, and $u_{\perp}=u-u_{\parallel}$ so that $\abs{v\cdot u}=\abs{v}\abs{u_{\parallel}}$ and $\abs{v-\eta}\geq \abs{u_{\perp}}$. $(1+\abs{u}^2)^{\beta}, \abs{u}(1+\abs{u}^2)^{\beta}$ are firstly absorbed by $e^{\frac{C_{\varpi}}{4} \abs{u}^2}$, and then the integral in \eqref{k_esti.3} is bounded by \eqref{n_definite}:
\begin{align*}
	C_{\beta} & \int_{\R^2} \bigg(\frac{1}{\abs{u_{\perp}}}+ 1\bigg) e^{-\frac{C_{\varpi}}{4} \abs{u}^2} \bigg \{ \int_{-\infty}^{\infty} e^{-C_{\varpi}\abs {v} \times \abs{u_{\parallel}}} \, d \abs{u_{\parallel}} \bigg \} \, du_{\perp} \\ 
	&\leq \frac{C}{\abs{v}} \int _{\R^2}\bigg(\frac{1}{\abs{u_{\perp}}}+1 \bigg) e^{-\frac{C_{\varpi}}{4} \abs{u_{\perp}}^2}\bigg \{ \int _{-\infty} ^{\infty} e^{-C_{\varpi} \abs {y} } \, dy \bigg \} \, du_{\perp} \quad 
	(y=\abs {v} \times \abs{u_{\parallel}}). 
\end{align*}  
Since both integrals are finite, we can deduce our lemma.
\end{proof}

\subsection{{Pointwise estimate on nonlinear term}}
The below Lemma \ref{Gamma+} implies that one can bound in the pointwise sense the $w$-weighted gain term by the product of the $w$-weighted $L^{\infty}$ norm and the weighted $L^2$ norm with another fixed velocity weight. In $C^3$ convex bounded domain, from the idea of \cite{KL}, we need triple velocity integration to perform a change of variable. However, unlike small data problem, one more iteration generates large factor $L^\infty$ norm as the estimate in Lemma \ref{Gamma+} and then we cannot apply the Gronwall type inequality. Through Lemma \ref{Gamma_new}, the large factor $L^\infty$ norm is no longer generated. These estimates play a crucial role in treating non-linear term whenever the solution could have a large amplitude in $C^3$ convex domain.

\begin{lemma} \label{Gamma+} \cite{DW}
There is a generic $C_{\beta}>0$ such that
\begin{equation} \label{gamma_gain 1}
	\abs{w(v)\Gamma_+ (f,f)(v)}\leq \frac{C_{\beta}\norm{wf}_{L^\infty _v}}{1+\abs{v}}\bigg(\int_{\R^3} (1+\abs{\eta})^4 \abs{e^{\varpi \abs{\eta}^2}f(\eta)}^2\, d\eta\bigg)^{\frac{1}{2}},
\end{equation}
for all $v\in \R^3$. In particular, for $\rho \geq 1$ and $\beta \geq 2 $, one has 
\begin{equation} \notag
	\abs{w(v)\Gamma_+(f,f)(v)}\leq \frac{C_{\beta}\norm{wf}_{L^\infty}^2}{1+\abs{v}},
\end{equation}
for all $v\in \R^3$, where the constant $C_{\beta}>0$ is independent of $\rho$ and $\varpi$. 

\end{lemma}
\begin{lemma} \label{Gamma_new} 
Let $0<\varpi\leq \frac{1}{64}$. There is a generic constant $C>0$ such that 
\begin{equation} \label{gamma 2}
	\abs{ w(v) \Gamma_+ (f,f) (v)}  \leq C \int_{\R^3} \abs{\tilde{k}_2(v,\eta)h^2(\eta)} \,d \eta,
\end{equation}
for all $v \in \R^3$, where $\tilde{k}_2$ is given by 
\begin{align} \notag
	\tilde{k}_2(v,\eta) = k'_2(v,\eta) \frac{w(v)}{w(\eta)}, 
\end{align}

\begin{align}  \label{k'_2}
	\abs{k'_2(v,\eta)} \leq \frac{C}{\abs{v-\eta}} e^{-C_{1,\varpi}\abs{v-\eta}^2}e^{-C_{1,\varpi}\frac{\abs{\abs{v}^2-\abs{\eta}^2}^2}{\abs{v-\eta}^2}}.
\end{align}
for some $C_{1,\varpi} > 0$. 
Moreover, there exists $C_{\varpi}>0$ such that 
\begin{align} \label{new kernel}
	\int_{\R^3} \abs{v-\eta}^{-1} e^{-C_1(1-\epsilon)\abs{v-\eta}^2} e^{-C_1(1-\epsilon)\frac{\abs{\abs{v}^2-\abs{\eta}^2}^2}{\abs{v-\eta}^2}} \frac{w(v)}{w(\eta)}
	\; d\eta \leq \frac{C_{\varpi}}{1+\abs{v}},
\end{align}
whenever $0\leq C_1 < C_1(\varpi)$ and $0 \leq \epsilon< \epsilon(C_1,\varpi)$ where $C_1(\varpi)$ and $\epsilon(C_1,\varpi)$ will be determined in the proof. This implies that the new kernel $\tilde{k}_2$ is still integrable and has singularity $\frac{1}{\abs{v-\eta}}$. 
\end{lemma}
\begin{proof}
By the definition of the gamma gain term in \eqref{k_1 and k_2}, we get 
\begin{align*}
	\left \vert w(v)\Gamma_+(f,f)(v) \right \vert &=\left \vert w(v)  \int_{\R^3} \int_{\S^2} B(v-u,\omega) \sqrt{\mu(u)} f(u')f(v') \,d\omega du \right \vert \\ 
	&=  \left \vert w(v)\int_{\R^3} \int_{\S^2} B(v-u,\omega) \sqrt{\mu(u)} w^{-1}(u') w^{-1}(v') h(u')h(v') \,d\omega du \right \vert \\ 
	&\leq w(v) \int_{\R^3} \int_{\S^2} B(v-u,\omega) \sqrt{\mu(u)} w^{-1}(u')w^{-1}(v') h^2(u') \,d\omega du \\ 
	&\quad +w(v) \int_{\R^3} \int_{\S^2} B(v-u,\omega) \sqrt{\mu(u)} w^{-1}(u')w^{-1}(v') h^2(v') \,d\omega du \\
	&=w(v) \Gamma_+ \left(w^{-1}, \frac{h^2}{w}\right) + w(v)\Gamma_+\left(\frac{h^2}{w}, w^{-1}\right) \\ 
	&\leq Cw(v)\int_{\R^3 } \left \vert k'_2(v,\eta)  \frac{h^2(\eta)}{w(\eta)} \right \vert d\eta.
\end{align*}
where it is well-known that $k'_2(v,\eta)$ satisfies (\ref{k'_2}) when $\varpi > 0$, cf.~\cite{Glassey}.

Since $w^{-1}$ has a similar role as $\sqrt{\mu}$  in \eqref{k_1 and k_2}, the last inequality above can be deduced. Note that $\varpi$ must be strictly positive for obtaining the last inequality. Defining
\begin{equation*}
\tilde{k}_2(v,\eta):=k'_2(v,\eta)\frac{w(v)}{w(\eta)},
\end{equation*} and then \eqref{gamma 2} follows. 
Therefore, from \eqref{k'_2}, there are $C>0$ and $C'_{\varpi}>0$ such that  
\begin{align} \label{new kernel_0}
	\abs{k'_2(v,\eta)}\leq C\abs{v-\eta}^{-1} e^{-C'_{\varpi}\abs{v-\eta}^2} e^{-C'_{\varpi} \frac{\abs{\abs{v}^2-\abs{\eta}^2}^2}{\abs{v-\eta}^2}},
\end{align}
for any $v,\eta \in \R^3$ with $v \neq \eta$. Similar to \eqref{k_esti.1}, we will find $C_1(\varpi)>0,\epsilon(C_1,\varpi)>0$ and $C_\varpi>0$ such that for $0 \leq C_1 <C_1(\varpi)$ and $0\leq \epsilon <\epsilon(C_1,\varpi)$, 
\begin{align} \label{new kernel_1}
	\int_{\R^3} \abs{v-\eta}^{-1}e^{-C_1(1-\epsilon)\abs{v-\eta}^2} e^{-C_1(1-\epsilon)\frac{\abs{\abs{v}^2-\abs{\eta}^2}^2}{\abs{v-\eta}^2}} \frac{w(v)}{w(\eta)}\; d\eta
	 \leq \frac{C_{\varpi}}{1+\abs{v}}.
\end{align} 
Similar to proof in Lemma 2.2, we firstly change the variable $u=v-\eta$, and then compute the total exponent in \eqref{new kernel_1}:
\begin{align*}
(-\varpi-2C_1(1-\epsilon))\abs{u}^2 +(4C_1(1-\epsilon)+2\varpi) v\cdot u -4C_1(1-\epsilon) \frac{\abs{v\cdot u}^2}{\abs{u}^2}.
\end{align*}
The discriminant of the above quadratic form of $\abs{u}$ and $\frac{v\cdot u }{\abs{u}}$ is 
\begin{align*}
	 \Delta= (4C_1(1-\epsilon)+2\varpi)((1-\epsilon)(4C_1-1)+2\varpi). 
\end{align*}
Thus, the quadratic form is negative definite for $0 \leq C_1< \frac{1-2\varpi}{4}$ and $0\leq \epsilon < \frac{1-4C_1-2\varpi}{1-4C_1}$. Taking $C_1(\varpi)=\min\{C'_{\varpi},\frac{1-2\varpi}{4}\}$ and $\epsilon(C_1,\varpi) = \frac{1-4C_1-2\varpi}{1-4C_1}$, we can deduce \eqref{new kernel} by the same argument in the proof of Lemma \ref{Lemma_nega}. This ends the proof of Lemma \ref{Gamma_new}. 
\end{proof}

\subsection{{Relative entropy}}
Recall again the definition of the relative entropy \eqref{Relative entropy} by
\begin{equation*}
\mathcal{E}(F) = \int_{\O}\int_{\R^3} \left(\frac{F}{\mu} \ln \frac{F}{\mu} - \frac{F}{\mu}+1 \right) \mu \, dvdx.
\end{equation*}
The below lemma is the global in time a priori estimate related with the relative entropy, cf.~\cite{CIP}. 

\begin{lemma} \label{decay_entropy}
	[Global in time a priori estimate] {It holds that}
	\begin{equation} \label{lem.dep1}
	\mathcal{E}({F(t)}) \leq \mathcal{E}(F_0), 
	\end{equation}
	for any $t\geq0$, where $F$ satisfies the Boltzmann equation \eqref{def.be} and the boundary condition \eqref{specular}. 
\end{lemma}
\begin{proof}
	Define a function
	\begin{align}\label{def.Psi}
	\Psi(s)=s\ln s -s +1 
	\end{align} 
	for $s>0$. Then, $\Psi$ is  a nonnegative and convex function on $(0,\infty)$ with 
	$\Psi'(s)=\ln s$. 
	From \eqref{def.be}, one can deduce that 
	\begin{align*}
	\partial_t [\mu \Psi\left(\frac{F}{\mu}\right) ]+ \nabla_x \cdot [ v\mu \Psi\left(\frac{F}{\mu}\right)] = Q(F,F) \ln \frac{F}{\mu}.
	\end{align*}
	Taking integration for $v\in \R^3$ and then for $x$ in $\O$, we get 
	\begin{align}\label{lem.dep2}
	\frac{d}{dt} \int _{\O} \int_{\R^3} \Psi\left(\frac{F}{\mu}\right) \mu \, dvdx + \int_{\partial \O} \int_{\R^3} 
	\Psi\left(\frac{\mu}{F}\right) \mu v \cdot n(x) \, dv dS_x = \int _{\O} \int_{\R^3} Q(F,F) \ln F \, dvdx.   
	\end{align}
	For $x \in \partial \O$, let us consider 
	\begin{align*}
	I_x := \int_{\R^3} \Psi\left(\frac{F}{\mu}\right) \mu v\cdot n(x) \, dv. 
	\end{align*}
	We can write 
	\begin{align*}
	I_x= \int_{v\cdot n(x)>0} \Psi\left(\frac{\gamma_{+} F}{\mu}\right) \, d\sigma_x - \int _{v\cdot n(x) <0} 
	\Psi\left(\frac{\gamma_{-}F}{\mu}\right) \, d\sigma_x, 
	\end{align*}
	where $d\sigma_x=d\sigma_x(v):=\mu(v)\abs{v\cdot n(x)}\,dv$. Applying the specular boundary condition \eqref{specular} to the negative part, 
	\begin{align*}
	I_x = \int_{v\cdot n(x)>0} \Psi\left(\frac{\gamma_{+} F}{\mu}\right) \, d\sigma_x - \int _{v\cdot n(x) <0} 
	\Psi\left(\frac{L_{\gamma_+}F}{\mu}\right) \, d\sigma_x, 
	\end{align*}
	where $L_{\gamma_+} F =F(R_x v)$. Since the negative part is the same as the positive part by taking change of variable, one has $I_x=0$ for $x\in \partial \O$. Further by the fact that $\int Q(F,F)F\,dv\leq 0$, it follows from \eqref{lem.dep2} that 
	\begin{align*}
	\frac{d}{dt} \int_{\O} \int_{\R^3} \Psi\left(\frac{F}{\mu}\right) \mu \, dvdx \leq 0,
	\end{align*}
	which proves \eqref{lem.dep1} after taking integration in time. 
	This ends the proof of Lemma \ref{decay_entropy}.
\end{proof}

Next, the following lemma shows that the relative entropy is used to control $L^1$ norm and $L^2$ norm of $F-\mu$ in the different regions, cf.~\cite{Guo-QAM}.

\begin{lemma} \label{entropy_est}
	[Control $L^1$ and $L^2$ by relative entropy]
	{It holds that} 
	\begin{equation}\label{lem.12re}
	{\int_{\O}\int_{\R^3}} \frac{1}{4\mu} \abs{F-\mu}^2 \cdot \mathbf{1} _{\abs{F-\mu}\leq \mu} \,{dvdx} + {\int_{\O}\int_{\R^3}} \frac{1}{4} \abs{F-\mu} \cdot \mathbf{1}_{\abs{F-\mu}>\mu} \, {dvdx} \leq \mathcal{E}(F_0), 
	\end{equation} 
	for any $t\geq0$, where $F$ satisfies the Boltzmann equation \eqref{def.be} and the boundary condition \eqref{specular}. Furthermore, if we write $F$ in terms of the standard perturbation $f$, then 
	\begin{equation} \label{lem.12repf}
	{\int_{\O}\int_{\R^3}} \frac{1}{4}\abs{f}^2 \cdot \mathbf{1}_{\abs{f} \leq \sqrt{\mu}} \, {dvdx} + {\int_{\O}\int_{\R^3}} \frac{\sqrt{\mu}}{4} \abs{f} \cdot \mathbf{1}_{\abs{f}>\sqrt{\mu}} \,{dvdx} \leq \mathcal{E}(F_0). 
	\end{equation} 
\end{lemma} 
\begin{proof}
	It is noticed that 
	\begin{align*}
	F\ln F - \mu \ln \mu = (1+\ln \mu) (F-\mu) + \frac{1}{2\tilde{F}} \abs{F-\mu}^2,
	\end{align*}
	where $\tilde{F}$ is between $F$ and $\mu$ form Taylor expansion. Then, we compute 
	\begin{align*}
	\frac{1}{2\tilde{F}} \abs{F-\mu}^2 &= F\ln F -\mu \ln \mu - (1+\ln \mu) (F-\mu) = \Psi\left(\frac{F}{\mu}\right) \mu,
	\end{align*}
	where $\Psi$ is given in \eqref{def.Psi}.
	Thus, 
	\begin{align*}
	\int_{\O}\int_{\R^3} \frac{1}{2\tilde{F}} \abs{F-\mu}^2 \, dvdx = \int_{\O}\int_{\R^3} \Psi(\frac{F}{\mu}) \mu
	\, dvdx, 
	\end{align*}
	which is uniformly in time bounded in terms of Lemma 3.1. For the left-hand side, we write 
	\begin{align*}
	1= \textbf{1}_{\abs{F-\mu} \leq \mu} + \textbf{1}_{\abs{F-\mu} > \mu}.
	\end{align*}
	Over $\{ \abs {F-\mu} > \mu\}$, we have $F > 2\mu$ and hence 
	\begin{align*}
	\frac{\abs{F-\mu}}{\tilde{F}} = \frac{F-\mu}{\tilde{F}} \geq \frac{F-\frac{1}{2}F}{F} =\frac{1}{2}.
	\end{align*}
	Over $\{\abs{F-\mu} \leq \mu\}$, we have $0 \leq F \leq 2\mu$ and hence 
	$\frac{1}{\tilde {F}} \geq \frac{1}{2\mu}$.
	Therefore, we obtain 
	\begin{multline*}
	\int_{\O} \int_{\R^3} \frac{1}{4\mu} \abs{F-\mu}^2 \cdot \textbf{1}_{\abs{F-\mu} \leq \mu} \, dvdx 
	+ \int_{\O}\int_{\R^3} \frac{1}{4}\abs{F-\mu}\cdot \textbf{1}_{\abs{F-\mu}>\mu}{\,dvdx} \\
	\leq \int_{\O} \int_{\R^3} \Psi\left(\frac{F}{\mu}\right) \mu \, dvdx \leq \int_{\O} \int_{\R^3} \Psi\left(\frac{F_0}{\mu}\right) \mu \, dvdx =\mathcal{E}(F_0),
	\end{multline*}
	for any $t\geq 0$. This gives the desired estimates \eqref{lem.12re} and \eqref{lem.12repf}, and hence ends the proof of Lemma \ref{entropy_est}.   
\end{proof}

\subsection{{Nondegeneracy of specular trajectory}}
We first recall some notations introduced in \cite{Guo}. Let $\O$ be {a uniformly convex domain}. Given $(t,x,v)$, let $[X_{cl}(s), V_{cl}(s)]=[X_{cl}(s;t,x,v), V_{cl}(s;t,x,v)]$ be the trajectory (position and velocity) of a particle at time $s$ which was at phase space $(t,x,v)$ under specular reflection boundary condition, $V_{cl}(s;t,x,v)=V_{cl}(s;t,x,Rv)$ when $X_{cl}(s;t,x,v)\in\p\O$. By definition, we have 
\begin{equation*}
	\frac{dX_{cl}(s)}{ds}=V(s), \quad \frac{dV_{cl}(s)}{ds}=0,
\end{equation*}
with the initial condition: $[X_{cl}(t;t,x,v),V_{cl}(t;t,x,v)]=[x,v]$. 
\begin{definition}[Backward exit time] For $(x,v)$ with $x\in \bar{\O}$ such that there exists some $\tau >0, x-sv\in \O$ for $0\leq s \leq \tau$, we define $\tb(x,v)>0$ to be the last moment at which the back-time straight line $[X(s;0,x,v),V(s;0,x,v)]$ remains in the interior of $\O$:
\begin{equation*}
	\tb(x,v):=\sup\{\tau>0:x-sv\in \O \textrm{ for all } 0\leq s \leq \tau\}.
\end{equation*}
Clearly, for any $x \in \O$, $\tb(x,v)$ is well-defined for all $v\in \R^3$. If $x \in \partial \O$, $\tb(x,v)$ is well-defined for all $v\cdot n(x)>0$. For any (x,v), we use $\tb(x,v)$ whenever it is well-defined. We have $x-\tb v \in \partial \O$. We also defined 
\begin{equation*}
	\xb(x,v)=x(\tb)=x-\tb v \in \partial \O.
\end{equation*} 
\end{definition}

\hide
\begin{lemma}{\textit{[1]}}(Velocity lemma 1)
Define the kinetic distance along the trajectories $\frac{dX(s)}{ds}=V(s), \, \frac{dV(s)}{ds}=0$ as: 
\begin{equation}
	\alpha(s)\equiv \xi^2(X(s))+[V(s)\cdot \nabla \xi(X(s))]^2 -2\{V(s)\cdot \nabla^2\xi(X(s))\cdot V(s)\} \xi(X(s)).
\end{equation}
Let $X(s) \in \bar{\O}$ for $t_1 \leq s \leq t_2$. Then, there exists constant $C_{\xi} >0$ such that 
\begin{align*}
	e^{C_\xi (\abs{V(t_1)}+1)t_1}\alpha(t_1) &\leq  e^{C_\xi (\abs{V(t_1)}+1)t_2}\alpha(t_2);\\
	e^{-C_\xi (\abs{V(t_1)}+1)t_1}\alpha(t_1)&\geq e^{-C_\xi (\abs{V(t_1)}+1)t_2}\alpha(t_2). 
\end{align*}
\end{lemma}
\begin{lemma}{\textit{[1]}}(Velocity lemma 2)
Let $(t,x,v)$ be connected with $(t-\tb,\xb, v) $ backward in time through a trajectory of  $\frac{dX(s)}{ds}=V(s), \, \frac{dV(s)}{ds}=0. \\$
Let $x_i \in \partial \O$, for $i=1,2$, and let $(t_1,x_1,v)$ and $(t_2,x_2,v)$ be connected with the trajectory $\frac{dX(s)}{ds}=V(s), \, \frac{dV(s)}{ds}=0$ which lies inside $\bar{\O}$. Then, there exists a constant $C_{\xi}>0$ such that 
\begin{equation*}
	\abs{t_1-t_2} \geq \frac{\abs{n(x_1)\cdot v}}{C_{\xi}\abs{v}^2}.
\end{equation*}
\end{lemma}
\unhide

$\newline$
Fix any point $(t,x,v)\notin \gamma_0 \cap \gamma_-$, and define $(t_0,x_0,v_0)=(t,x,v)$, and for $k\geq 1$ 
\begin{equation} \label{bounce_pt}
	(t_{k+1},x_{k+1},v_{k+1})=(t_k-\tb(t_k,x_k,v_k),x_b(x_k,v_k),R(x_{k+1}v_k),
\end{equation}
where $R(x_{k+1})v_k = v_k - 2(v_k\cdot n(x_{k+1}))n(x_{k+1})$. And we define the specular back-time cycle 
\begin{equation} \notag
	X_{cl}(s) \equiv \sum_{k=1}^{\infty} \textbf{1}_{[t_{k+1},t_k)}(s)\{x_k-(t_k-s)v_k\}, \quad V_{cl}(s)\equiv \sum_{k=1}^{\infty}  \textbf{1}_{[t_{k+1},t_k)}(s)v_k.
\end{equation}
We denote the weighted function 
\begin{equation} \notag
	h(t,x,v)=w(v)f(t,x,v), 
\end{equation}
and study the linearized Boltzmann equation in terms of $h$: 
\begin{equation} \notag
	\partial_t h +v\cdot \nabla_x h +\nu(v)h =K_w h, \quad h(0,x,v)=h_0(x,v)\equiv wf_0, 
\end{equation}
where 
\begin{equation*} \label{kw}
	K_wh(v) = wK\bigg(\frac{h}{w}\bigg)(v) := \int \textbf{k}_{w}(v,u) h(u) \,du.
\end{equation*}
We now fix any point $(t,x,v)$ so that $(x,v) \notin \gamma_0$. Let the back-time specular cycle of $(t,x,v)$ be denoted by $[X_{cl}(s_1),V_{cl}(s_1)]$. We will use twice the Duhamel's principle to derive that
\begin{align}
	h(t,x,v)&= e^{-\nu(v)t}h_0 (X_{cl}(0),V_{cl}(0)) \nonumber \\
	&\quad+\int_0 ^t e^{-\nu(v)(t-s_1)}\int \textbf{k}_w(V_{cl}(s_1),v')h(s_1,X_{cl}(s_1),v')\, dv'ds_1 \nonumber \\
	&= e^{-\nu(v)t}h_0 (X_{cl}(0),V_{cl}(0)) \nonumber \\
	&\quad+\int_0^t e^{-\nu(v)(t-s_1)}\int \textbf{k}_w(V_{cl}(s_1),v')e^{-\nu(v')s_1}h_0(X_{cl}^{'}(0),V_{cl}^{'}(0))\, dv'ds_1 \nonumber \\ 
	&\quad+\int_0^t \int _0 ^{s_1} \int e^{-\nu(v)(t-s_1)-\nu(v')(s_1-s)}\textbf{k}_w(V_{cl}(s_1),v')\textbf{k}_w(V_{cl}'(s),v{''})h(X_{cl}'(s),v'')\, dv'dv{''}dsds_1,\notag
\end{align}
where the back-time specular cycle from $(s_1,X_{cl}(s_1),v')$ is denoted by
\begin{equation} \notag
	X_{cl}' (s)=X_{cl}(s;s_1,X_{cl}(s_1),v'), \quad V_{cl}' (s)=V_{cl}(s;s_1,x_{cl}(s_1),v').
\end{equation}
More explicitly, let $t_k$ and $t_{k'}' $ be the corresponding times for both specular cycles as in \eqref{bounce_pt}. For $t_{k+1}\leq s_1 <t_k, \, t_{k'+1}' \leq s <t_{k'}'$
\begin{equation*}
	X_{cl}'(s)=X_{cl}(s;s_1,X_{cl}(s_1),v')\equiv x_{k'}'-(t_{k'}'-s)v_{k'}',
\end{equation*}
where $x_{k'}'=X_{cl}(t_{k'};s_1,x_k-(t_k-s_1)v_k,v'), \, v_{k'}' =V_{cl}(t_{k'};s_1,x_k-(t_k-s_1)v_k,v').$

Now we recall the following key lemma and proposition for the  nondegeneracy condition from \cite{KL}.  

\begin{lemma} \cite{KL}
	Assume $\O$ is convex. Choose $N \gg 1, \; 0<\delta \ll 1$. There exist collections of open subsets 
	$\{ \mathcal{O}_i\}_{i=1}^{I_{\O,N,\delta}}$ of $\O$ and $\{ \mathcal{V}_i(\textbf{q}_1,\textbf{q}_2) \}_{i=1}^{I_{\O,N,\delta}}$ of 
	$\R^3$, where $\textbf{q}_1$ and $\textbf{q}_2$ are two independent vectors in $\R^3$, with $I_{\O, N, \delta}<\infty$ such that 
	$\bar{\O} \subset  \bigcup_{i} \mathcal{O}_i$ and $\int _{\R^3\backslash \mathcal{V}_i(\textbf{q}_1,\textbf{q}_2)} e^{-\frac{\abs{v}^2}{100}} \; dv 
	\leq O_{\O} (\frac{1}{N})$. Moreover, 
	\begin{align} \label {Lee D .1} 
		\begin{split}
			K_i := \sup \{k \in \mathbb{N} : &t^k(t,x,v)\geq T, \\ 
			& (t,x,v) \in [T,T+1] \times \mathcal{O}_i \times \R^3 \backslash \mathcal{V}_i (\textbf{q}_1,\textbf{q}_2) \} < \infty. 
		\end{split}
	\end{align}
	If $(x,v) \in \mathcal{O}_i \times \R^3\backslash \mathcal{V}_i ( \textbf{q}_1, \textbf{q}_2) $ for some $i$, then 
	\begin{equation} \notag
		\abs{n(x^1(t,x,v)) \cdot v^1(t,x,v) } >  C_{\O,N} \delta, 
	\end{equation}
	and 
	\begin{equation} \notag
		\abs{(\textbf{q}_1  \times \textbf{q}_2 ) \cdot v } \geq \frac{1}{N}. 
	\end{equation}
\end{lemma}

The below Proposition implies that the zero set $\bigg\{\det \bigg(\frac{d X_{cl}(s;s_1,X_{cl}(s_1;t,x,v),v')}{dv'} \bigg)=0\bigg\}$ is very small if $\O$ is both $C^3$ and convex.

\begin{proposition} \label{prop_COV}
\cite{KL}
 Fix arbitrary $(t,x,v)\in [T,T+1] \times \O \times \R^3$. Recall $N, \delta, $ and $\mathcal{O}_i,\mathcal{V}_i(\hat{\textbf{e}}_1,\hat{\textbf{e}}_2)$, which are chosen in Lemma 2.8. For each $i=1,2,\dots, I_{\O,N,\delta},$
 there exist $\delta_1>0$ and a $C^1$-function $\psi^{l_0,\vec{l},i,k}$ for $k \leq K_i$ in \eqref{Lee D  .1} where $\psi^{l_0,\vec{l},i,k}$ is defined locally around $(T+\delta_1 l_0, X(T+\delta_1 l_0; t,x,v),\delta_1 \vec{l})$ with 
 	\begin{align*}
 		(l_0,\vec{l}) = (l_0,l_1,l_2,l_3) \in \bigg \{ &0,1, \dots, \bigg \lfloor \frac{1}{\delta_1} \bigg \rfloor+1 \bigg \} \\
		&\times \bigg \{ -\bigg \lfloor \frac{N}{\delta_1} \bigg \rfloor -1, \dots,0,\dots, \bigg \lfloor \frac{N}{\delta_1} \bigg \rfloor +1 \bigg \}^3
 	\end{align*}
and $\Vert \psi^{l_0,\vec{l},i,k} \Vert_{C^1} \leq C_{N,\O,\delta,\delta_1}<\infty$. 

Moreover, if 
\begin{equation}\label{Lee D .2} 
	(X(s;t,x,v),u) \in \mathcal{O}_i \times \R^3 \backslash \mathcal{V}_i (\hat{\textbf{e}}_1, \hat {\textbf{e}}_2)  \quad \textrm{for } i=1,2,\dots, 
	I_{\O,N,\delta}, 
\end{equation}

\begin{equation} \label{Lee D .3} 
	(s,u) \in [T,(l_0-1)\delta_1, T+(l_0+1)\delta_1] \times B(\delta_1 \vec{l}; 2\delta_1), 
\end{equation}

\begin{align} \label{Lee D .4}
	\begin{split}
		s' \in \bigg [ &t^{k+1} (T+\delta_1l_0; X(T+ \delta_1 l_0; t,x,v) ,\delta_1 \vec{l} +\frac{1}{N}, \\ 
		& t^k(T+\delta_1l_0; X(T+ \delta_1 l_0; t,x,v) ,\delta_1 \vec{l})-\frac{1}{N}\bigg], 	
	\end{split}
\end{align}
and 
\begin{align} \label{Lee D .5} 
	\begin{split}
		\bigg \vert s'-\psi^{l_0,\vec{l},i,k} (T+\delta_1l_0, X(T+\delta_1 l_0 ;t,&x,v), \delta_1 \vec{l})\bigg \vert > N^2 (1+\norm{\psi^{l_0,\vec{l},i,k}}_{C^1} ) \delta_1, 
	\end{split}
\end{align}
then
\begin{align}   \notag
	\begin{split}
		\bigg \vert \partial_{\vert u \vert} X(s'; s, X(s;t,x,v) ,u) \times \partial _{\hat{u}_1} X(s'; s,X(s;t,x,&v),u) \bigg \vert >  \epsilon_{\O,N,\delta, \delta_1}.
	\end{split}
\end{align}
Here $\epsilon_{\O,N,\delta, \delta_1} >0 $ does not depend on $T, t, x,$ or $v$ and $\hat{u}_1 =u_1/\abs{u}$.  

For each $j=1,2,\dots,I_{\O,N,\delta}$ in Lemma 2.8, there exists $\delta_2 > 0$ and $C^1$-functions
\begin{equation} \notag
	\psi_1^{l_0,\vec{l},i,k,j,m_0,\vec{m},k'}, \; \psi_2^{l_0,\vec{l},i,k,j,m_0,\vec{m},k'}, \; \psi_3^{l_0,\vec{l},i,k,j,m_0,\vec{m},k'}, 
\end{equation}
for $k' \leq K_j$ in \eqref{Lee D .1} where $\psi_n ^{l_0,\vec{l},i,k,j,m_0,\vec{m},k'}$ is defined locally around 
\begin{equation*}
	(T + \delta_2 m_0; X(T + \delta_2 m_0 l; T + \delta_1 l_0 , X(T + \delta_1l_0;t,x,v) ,\delta_1 \vec{l}), \delta_2 \vec{m}) 
\end{equation*}
for some 
\begin{align*}
	(m_0,\vec{m}) = (m_0,m_1,m_2,m_3) \in \bigg \{ 0&,1,\dots,\bigg \lfloor \frac{1}{\delta_2}\bigg \rfloor +1 \bigg \} \\ 
	 &\times \bigg \{ - \bigg \lfloor \frac{N}{\delta_2} \bigg \rfloor -1, \dots, 0, \dots, \bigg \lfloor \frac{N}{\delta_2} \bigg \rfloor +1 \bigg \}^3 
\end{align*}
with $0< \delta_2 \ll 1$.  

Moreover, if we assume \eqref{Lee D .2}, \eqref{Lee D .3}, \eqref{Lee D .4}, and \eqref{Lee D .5}, 
\begin{align}  \notag
	\begin{split}
		(X(s'; &s, X (s;t,x,v),u),u' ) \in \mathcal{O}_j \times \R^3 \backslash \mathcal{V}_j ( \partial _{\vert u \vert} X, \partial_{\hat{u}_1} X ) \\ 
		& \textrm{for some} \; j=1,2, \dots, I_{\O,N,\delta, } \; \textrm{in Lemma 2.8},
	\end{split}
\end{align}

\begin{align} \label{Lee D .6}
	\begin{split}
		s'' \in \bigg [ &t^{k'+1} (T + \delta_2m_0 ; X (T + \delta_2 m_0 ; T + \delta _1 l_0 , X (T + \delta_1 l_0;t,x,v), \delta_1 \vec{l}), \delta_2 \vec{m}) + \frac{1}{N}, \\ 
		&t^{k'} \underbrace{(T +\delta_2 m_0 ; X (T + \delta_2 m_0 ; T + \delta_1 l_0, X (T + \delta_1 l_0; t,x,v) , \delta_1 \vec{l}) ,\delta _2 \vec{m})}_{(**)} - \frac{1}{N} 
		\bigg ] ,  		
	\end{split}
\end{align}
and 
\begin{align} \notag
	\begin{split}
		\min_{n=1,2,3}  \bigg \vert  s'' - \psi_n & ^{l_0,\vec{l},i,k,j ,  m_0  ,\vec{m},k'} (**)  \bigg \vert >  N ^2 (1+\max_{n=1,2,3} \norm{ \psi_n ^{l_0 , \vec{l} , i,k,j,m_0 , \vec{m}, k' } } _{C^1} ) \delta_2, 
	\end{split}
\end{align}
where {$(**)$} is defined in \eqref{Lee D .6}. Then for each $l_0, \vec{l}, i,k,j,m_0,\vec{m},$ and $k'$ we can choose two distinct variables 
$\{ \xi_1,\xi_2\} \subset \{\vert u \vert, \hat{u}_1 ,\hat{u}_1', \hat{u}_2'\}$ such that $(\vert u' \vert, \xi_1, \xi_2) \mapsto X(s'';s',X(s';s,X(s;t,x,v),u),u')$ is one-to-one 
locally and 
\begin{equation} \notag
	\bigg \vert \det \bigg ( \frac{\partial X (s'';s',X(s';s,X(s;t,x,v),u),u')}{\partial (\vert u' \vert, \xi_1 ,\xi_2)} \bigg) \bigg \vert > \epsilon_{\O, N, \delta, \delta_1,\delta_2}'. 
\end{equation} 
Here  $\epsilon_{\O, N, \delta, \delta_1,\delta_2}'>0$ does not depend on $T,t,x,$ or $v$ and $\hat{u}'_i=u'_i/\abs{u'}$ for each $i=1,2$. 
\end{proposition}


\section{A priori estimates}\label{sec3}
This section is devoted to obtaining the {\it a priori} estimates along the way explained in Section \ref{sec.dsp}. To do so, for an arbitrary $T>0$, we let
\begin{equation*}
	F(t,x,v)=\mu(v) +\sqrt{\mu(v)} f(t,x,v)\geq0,
\end{equation*}
be a solution to the IBVP \eqref{def.be}, \eqref{id} and \eqref{specular} over the time interval $[0,T)$ with initial data $F_0(x,v)=\mu(v)+\sqrt{\mu(v)}f_0(x,v)\geq 0$. Set $h(t,x,v)=w(v)f(t,x,v)$ for a weight function $w$ given in \eqref{weight}. Note that the size of initial data $h_0(x,v)=w(v)f_0(x,v)$ can be large in $L^\infty$, namely, it holds that $\norm{h_0}_{L^\infty} \leq M_0$ for an arbitrarily given constant $M_0 \geq 1$. To the end, we impose the $\textit{a priori}$ assumption
\begin{equation} \label{a priori assumption}
	\sup_{0\leq t \leq T} \norm{h(t)}_{L^{\infty}} \leq \bar{M},
\end{equation}
where $\bar{M} \geq 1 $ is a large positive constant to be chosen depending only on $M_0$, not on the time $T$.  \\

\subsection{{Estimate in $L_x^\infty L_v^1$}} 
When treating the solution of small amplitude in $L^\infty$, it is a usual way to take the Boltzmann equation of the form
\begin{equation*} 
	\partial_t h+ v\cdot \nabla_x h +wLf = w \Gamma(f,f),
\end{equation*}
for $h=wf$, since nonlinearity can be estimated by $|w\Gamma(f,f)| \lesssim \langle v \rangle^{\kappa}\|h\|_{\infty}^{2}$. In case of the large amplitude problem (cf.~\cite{DHWY17,DW}), it is hard to follow the same strategy due to  velocity growth of $\langle v \rangle^{\kappa}$ for $\kappa>0$. Instead, we shall decompose $\Gamma(f,f)$ into $\Gamma_{\pm}(f,f)$  so as to estimate the gain and loss terms in the way as in \cite{DW}.

First, from Lemma \ref{Gamma+} and Lemma \ref{Gamma_new}, we can estimate $\Gamma_{+}$ by some weighted $L^{\infty}$ and $L^2$ norms of $h$. In particular, we would emphasize that the $L^{2}$ norm of $h$ plays an important role, because we expect to use the smallness of the relative entropy to control the $L^2$ and $L^1$ norms of the solution in the sense of Lemma \ref{entropy_est}. To avoid the appearance of the high order terms of $\|h\|_{\infty}$, we choose to apply the Lemma \ref{Gamma_new}. Estimate for $\Gamma_{-}(f,f)$, however, is very different to the one for $\Gamma_{+}(f,f)$ because the former always generates $\|h\|_{\infty}$ in each iteration step by the fact that $\Gamma_-(f,f)$ contains the local term $f(t,x,v)$. Thus, for the large-amplitude problem, it is beneficial to use the formulation 
\begin{equation} \label{reform_B}
	\partial_t h +v\cdot \nabla_x h + h\mathcal{A}F= K_w h +w \Gamma_+(f,f),
\end{equation}
where we have denoted $\mathcal{A}F(t,x,v)= \int_{\R^3} \int_{\S^2} B(v-u,\omega)F(t,x,u)\, d\omega du$. Note that $h\mathcal{A}F$ comes from the combination of $wLf$ and $w\Gamma_- (f,f)$. 

Although it is no longer a problem to control the gain term $\Gamma_{+}(f,f)$, one has to obtain the time-decay property of the linearized solution operator $G^f(t,s)$ corresponding to the linear equation 
\begin{equation} \label{Af}
	\partial_t h + v\cdot \nabla_x h + h\mathcal{A}F =0, 
\end{equation} 
with the corresponding specular reflection boundary for a given function $F=\mu +\sqrt{\mu} f\geq0$. Fortunately, under the {\it a priori} assumption \eqref{a priori assumption} with the parameter $\rho>0$ in the weight function $w$ to be chosen suitably large, one can still get the uniform exponential time-decay of the linearized solution operator $G^f$. The following lemma plays a key role in deriving such property. 

\begin{lemma} \label{Rf est}
Assume \eqref{a priori assumption}. Let $\rho=\rho(\bar{M})$ in the weight function $w(v)$ suitably large such that 
\begin{equation} \label{rho}
	C_{\beta}\frac{\bar{M}}{\rho^3} \leq \frac{1}{2C_*}
\end{equation}
for some generic positive constants $C_*$ and $C_\beta$ to be given in the proof, then it holds that 
\begin{equation} \label{R_f lb}
	R(f)(t,x,v)\geq \frac{1}{2}\nu(v) \quad \quad for \; all \quad (t,x,v)\in [0,T_0)\times \O \times \R^3,
\end{equation}
where $T_0\in (0,T)$ is arbitrary, and
\[
	R(f)(t,x,v)= \displaystyle \int _{\R^3}\int_{\S^2} B(v-u,\omega) \left( \mu(u)+\sqrt{\mu(u)}f(t,x,u) \right )\, d\omega du.
\]
\end{lemma} 

\begin{proof}
By \eqref{def.nu}, $R(f)(t,x,v)$ has the lower bound as 
\begin{align*}
	R(f)(t,x,v)&
			 =\nu(v) + \int _{\R^3}\int_{\S^2} B(v-u,\omega) \sqrt{\mu(u)}f(t,x,u)\, d\omega du\\
			 &\geq \nu(v) \left [1-C_* \int_{\R^3} e^{-\frac{\abs{u}^2}{8}}\abs{f(t,x,u)} \, du\right],
\end{align*}
for a generic constant $C_* \geq 1$. Therefore, to show the desired estimate \eqref{R_f lb}, it suffices to prove that 
\begin{equation} \label{G_f decay_goal}
	\int_{\R^3} e^{-\frac{\abs{v}^2}{8}}\abs{f(t,x,v)} \, dv \leq \frac{1}{2C_*},
\end{equation}
for all $(t,x) \in [0,T_0)\times \O$. We easily check that the integral in \eqref{G_f decay_goal} is bounded by
\begin{align} 	\label{G_f decay bound_1}
	\int_{\R^3} e^{-\frac{\abs{v}^2}{8}} \abs{f(t,x,v)}\, dv &=\int_{\R^3} e^{-\frac{\abs{v}^2}{8}}\frac{1}{w(v)}\abs{h(t,x,v)} \, dv\notag\\  
	&\leq \sup_{0\leq s \leq t} \norm{h(s)}_{L^\infty} \int_{\R^3} \frac{1}{(1+\rho^2 \abs{v}^2)^\beta} \, dv  \notag\\
	&\leq \bar{M} \int_{\R^3} \frac{1}{(1+\rho^2\abs{v}^2)^\beta} \,dv.
\end{align}
Notice that 
\begin{equation}
\label{def.intrho}
\int_{\R^3} \frac{1}{(1+\rho^2\abs{v}^2)^\beta}\,dv=\frac{1}{\rho^3}\int_{\R^3} \frac{1}{(1+|v|^2)^\beta}\,dv=\frac{C_\beta}{\rho^3},
\end{equation}
where $C_\beta>0$ is a constant depending only on $\beta$. Plugging  \eqref{def.intrho} to \eqref{G_f decay bound_1} and using the condition \eqref{rho}, we get the estimate \eqref{G_f decay_goal} and then complete the proof of Lemma 3.1. 
\end{proof}

\subsection{{Estimate in $L_{x,v}^{\infty}$}}
To obtain the nonlinear $L^\infty$ estimate for solutions with large amplitude, the time-decay property of the linear solution operator over the time interval $[0,T_0]$ is essential. Here, the time $T_0$ depending only on $M_0$ will be determined at the end of the proof. The following lemma implies that our linear solution operator has the time-decay property. Note that the lemma is an immediate consequence of Lemma \ref{Rf est}. 

\begin{lemma}\label{G_f decay}  
Recall the solution operator $G^f(t,s)$ to the linear equation \eqref{Af}. Assume that $\rho>0$ in the weight function $w(v)$ satisfies \eqref{rho}. 
Then, under the a pirori assumption \eqref{a priori assumption}, it holds that 
	\begin{equation} \label{G_f esti}
		\norm{G^f(t,s)h_0} _{L^\infty} \leq C e^{-\frac{1}{2} \nu_0 (t-s)} \norm{h_0} _{L^{\infty}}  \quad \textrm{ for all } \quad 0\leq s\leq t \leq T_0,
	\end{equation}
where $C>0$ is a generic constant and $\nu_0=\inf_{v}\nu(v)$ is strictly positive.
\end{lemma}

\begin{proof}
    	It follows from \eqref{R_f lb} in Lemma 3.1 that 
	\begin{align} \label{Rf_split}
		R(f)(t,x,v)\geq\frac{1}{2}\nu(v), 		
	\end{align}
for all $(t,x,v) \in [0,T_0)\times\O\times \R^3$. We only need to consider the estimate for $G^f(t,0)$ because the estimate for $G^f(t,s)$ can be deduced similarly. 
Note that
\begin{align*}
	G^f(t,s)= \exp \left \{ -\int_s^t R(f)(\tau,X_{cl}(\tau),V_{cl}(\tau)) \,d\tau \right\}.
\end{align*}
We can obtain from \eqref{Rf_split} that 
\begin{align*}
	G^f(t,0)\leq \exp\left \{-\int_0^t \frac{\nu_0}{2} \, d\tau \right \} = e^{-\frac{\nu_0}{2}t}.
\end{align*}
Moreover, it is obvious to see 
\begin{align} \label{G_f 1}
	\norm {G^f(t,0)h_0}_{L^\infty} \leq  e^{-\frac{1}{2}\nu_0 t}\norm{h_0}_{L^\infty}.
\end{align}
Using \eqref{G_f 1}, we get the estimate \eqref{G_f esti} and then complete the proof of Lemma 3.2.  
\end{proof} 
 
We try to apply Lemma \ref{G_f decay} to treat the nonlinear $L^\infty$ estimate over the time interval $[0,T_0]$. For this purpose, we recall the Boltzmann equation \eqref{reform_B} and use the Duhamel principle to rewrite it in terms of $h(t,x,v)=w(v)f(t,x,v)$ as 
\begin{equation} \label{Duhamel}
	h(t,x,v)=G^f(t,0)h_0+\int_0^t G^f(t,s)K_wh(s) \, ds + \int_0^t G^f(t,s)w\Gamma_+(f,f)(s)\,  ds,
\end{equation}
where $G^f(t,s)$ be the solution operator to the linear equation \eqref{Af}. 

\begin{lemma} \label{L_inf_esti}
Let $h(t,x,v)$ satisfy the equation \eqref{reform_B}. Assume that \eqref{a priori assumption} holds true and $\rho>0$ in \eqref{weight} satisfies \eqref{rho}. Denote $A_0=\sqrt{\mathcal{E}(F_0)}+\mathcal{E}(F_0)$. 
Let $t \in (0,T_0]$, 
then it holds that 
\begin{align} \label{L_inf_esti.result}
	\|h(t)\|_{L^\infty}
	&\leq Ce^{-\frac{\nu_0}{4}t} \left(\int_0^t \norm{h(s)}_{L^\infty} +1 \right)\left[\norm{h_0}_{L^\infty}+\norm{h_0}_{L^\infty}^2 \right]  \notag\\
	&\quad	 +C_{\O}\left(\delta+\frac{1}{\sqrt{N}}\right)\sup_{0\leq s \leq t }\norm{h(s)}_{L^\infty}^5 + C_{N,\delta,\O}A_0,
\end{align}
where $0 < \delta \ll 1$ and $N\gg 1$ can be chosen arbitrary small and large, respectively.
\end{lemma}

\begin{proof}
Take $(t,x,v)\in  (0,T_0] \times \O \times \R^3$. It is direct to deduce from \eqref{Duhamel}, Lemma \ref{G_f decay}, and \eqref{gamma_gain 1} in Lemma \ref{Gamma+} that
\begin{align} \label{L_inf_esti.1}
	\abs{h(t,x,v)} &\leq \abs{G^f(t,0)h_0} + \int_0^t \abs{G^f(t,s)} \bigg[\abs{K_wh(s)} + \abs{w\Gamma_+(f,f)(s)}\bigg] \, ds \nonumber \\
	&\leq C e^{-\frac{\nu_0}{2} t} \norm{h_0}_{L^\infty}  \notag\\ 
	&\quad+ C \int_0^t e^{-\frac{\nu_0}{2} (t-s)} \int_{\R^3} \left \vert k_w(V_{cl}(s),u)h(s,	X_{cl}(s),u) \right \vert \, duds  \notag\\ 
	&\quad + C\int_0^t e^{-\frac{\nu_0}{2} (t-s)} \frac{\norm{h(s)}_{L^\infty}}{1+\abs{v}} \bigg(\int_{\R ^3} (1+\abs{u})^{-4\beta+4} \abs{h(s,X_{cl}(s),u)}^2 \, du\bigg)^{\frac{1}{2}}\, ds\notag\\ 
	&:=I_1+I_2+I_3, 
\end{align} 
where $I_1,I_2$ and $I_3$ denote those terms on the right-hand side, respectively.
For $\abs{v} \geq N$, we have the following estimates about the two integral terms in \eqref{L_inf_esti.1} 
\begin{align} \label{I_2,I_3}
	\begin{split}
	I_2 \leq \frac{C}{N}
	\sup_{0\leq s \leq t}\norm{h(s)}_{L^\infty}, \quad 
	I_3 
	\leq \frac{C}{N}  \sup_{0\leq s \leq t}\norm{h(s)}_{L^\infty}^2, 
	\end{split}
\end{align}
where Lemma \ref{Lemma_nega} has been used.
From now on, we only conisder $\abs{v} \leq N$ for $I_{2}$ and $I_{3}$. 

\medskip
\noindent\underline{\textit{Estimate  on $I_2$}}: 
Firstly, we divide the integral domain of $I_2$ in \eqref{L_inf_esti.1} into $\{\abs{u}\leq 2N\}\cup \{\abs{u} \geq 2N\}$ so as to write that
\begin{align} \label{L_inf_split.1(n)}
	I_2=\int_0^t \int_{\abs{u}\geq 2N} + \int_0^t \int_{\abs{u}\leq 2N}.  
\end{align}
In the domain $\{\abs{v} \leq N, \abs{u}\geq 2N\}$, using \eqref{k_esti.3} in Lemma \ref{Lemma_nega} yields that 
\begin{align} \label{k_prop.1(n)}
	\begin{split}
		\int_{\abs{u} \geq 2N} \abs{k_w(v,u)} \,du \leq e^{-\frac{N^2}{32}} \int_{\abs{u} \geq 2N}  \abs{k_w(v,u)}  e^{\frac{\abs{v-u}^2}{32}} \,du \leq \frac{C}{N}.
	\end{split}
\end{align}
Thus, we can bound the first part in \eqref{L_inf_split.1(n)} from \eqref{k_prop.1(n)} in the way that
\begin{align} \label{I_2.1}
	\int_0^t \int_{\abs{u}\geq 2N} \leq C \int_0^t e^{-\frac{\nu_0}{2} (t-s)} \int_{\abs{u} \geq 2N} \left \vert k_w(V_{cl}(s),u)h(s,X_{cl}(s),u) \right \vert \, duds \leq \frac{C}{N} 
	\sup_{0\leq s \leq t} \norm{h(s)}_{L^\infty}.
\end{align}
Applying the Duhamel formula \eqref{Duhamel} to the second part in \eqref{L_inf_split.1(n)}, we get
\begin{align} \label{L_inf_esti.3}
	 \int_0^t \int_{\abs{u}\leq 2N} 
	 &\leq C\int_0^t e^{-\frac{\nu_0}{2} (t-s)} \int_{\abs{u}\leq 2N} \abs{k_w(V_{cl}(s),u)h(s,X_{cl}(s),u)} \, duds \nonumber \\
	 & \leq  C \int_0^t e^{-\frac{\nu_0}{2} (t-s) } \int _{\abs{u}\leq 2N} \abs{k_w(V_{cl}(s),u)G^f(s,0)h_0}\, duds \nonumber \\
	 &\quad+C \int_0^t \int_{\abs{u}\leq 2N}  e^{-\frac{\nu_0}{2} (t-s) }\abs{k_w(V_{cl}(s),u)}\left \vert \int_0^s G^f(s,s')K_wh(s')ds' \right \vert \, du ds \nonumber \\ 
	 &\quad+ C  \int_0^t \int_{\abs{u}\leq 2N}e^{-\frac{\nu_0}{2}  (t-s) } \abs{k_w(V_{cl}(s),u)} \left \vert \int_0^s G^f(s,s')w\Gamma_+(f,f)(s') ds'\right \vert \, du ds \nonumber \\
	 &\leq C\norm{h_0}_{L^\infty} \int_0^t e^{-\frac{\nu_0}{2}t}\int_{\abs{u}\leq 2N} \abs{k_w(V_{cl}(s),u)} duds \nonumber \\ 
	 &\quad+C\int_0^t \int_0^s \int_{\abs{u}\leq 2N}e^{-\frac{\nu_0}{2}(t-s')}\abs{k_w(V_{cl}(s),u)} 
	 \int_{\R^3} \abs{k_w(V'_{cl}(s'),u')h(s',X'_{cl}(s'),u')}du'duds'ds \nonumber \\ 
	 &\quad+ C\int_0^t \int_0^s \int_{\abs{u}\leq 2N}e^{-\frac{\nu_0}{2}(t-s')}\abs{k_w(V_{cl}(s),u)} 
	\int_{\R^3} \abs{\tilde{k}_2(V'_{cl}(s'),u')h^2(s',X'_{cl}(s'),u')}du'duds'ds\nonumber \\ 
	 &:= I_{21} + I_{22} + I_{23},
\end{align}
where the last inequality comes from \eqref{gamma 2} in Lemma \ref{Gamma_new} and Lemma \ref{G_f decay}.  
To estimate those terms in \eqref{L_inf_esti.3}, from 
\begin{align} \label{k_prop.1} 
	\int_{\R^3} \abs{k_w(V_{cl}(s),u)} \, du \leq C,
\end{align}
the term $I_{21}$ in \eqref{L_inf_esti.3} is bounded by 
\begin{equation} \label{L_inf_esti.4}
	I_{21} \leq C  e^{-\frac{\nu_0}{4}t} \norm{h_0} _{L^\infty}.
\end{equation}
We further split the integration domain of $I_{22}$ in \eqref{L_inf_esti.3} as 
\begin{align} \label{L_inf_split.2(n)}
	I_{22} = \int_0^t \int_0^s \int_{\abs{u}\leq 2N} \int_{\abs{u'}\geq 3N} + \int_0^t \int_0^s \int_{\abs{u}\leq 2N} \int_{\abs{u'} \leq 3N}. 
\end{align}
Using \eqref{k_prop.1(n)} and \eqref{k_prop.1}, the first part in \eqref{L_inf_split.2(n)} is further bounded by
\begin{multline} \label{I_22.1}
	C \int_0^t\int_0^s \int_{\abs{u}\leq 2N} \int_{\abs{u'}\geq 3N} e^{-\frac{\nu_0}{2} (t-s') }\abs{k_w(V_{cl}(s),u)k_w(V'_{cl}(s'),u')h(s',X'_{cl}(s'),u')} \, du'duds' ds \\ 
	\leq \frac{C}{N} \sup_{0\leq s \leq t} \norm{h(s)}_{L^\infty}. 
\end{multline}
We apply \eqref{L_inf_esti.1} to $h(\tau,X_{cl}'(\tau),u')$ in the second part of \eqref{L_inf_split.2(n)}, and hence use Lemma \ref{G_f decay} to arrive at
{\footnotesize\begin{align} \label{L_inf_split.3(n)}
	&C \int_0^t\int_0^s \int_{\abs{u}\leq 2N} \int_{\abs{u'}\leq 3N} e^{-\frac{\nu_0}{2} (t-s') }\abs{k_w(V_{cl}(s),u)k_w(V'_{cl}(s'),u')h(\tau,X'_{cl}(s'),u')} \,du'duds' ds   \nonumber \\
	&\leq C\norm{h_0}_{L^\infty} \int_0^t \int_0^s\int_{\abs{u}\leq 2N} \int_{\abs{u'}\leq 3N} 
	e^{-\frac{\nu_0}{2} t} \abs{k_w(V_{cl}(s),u)k_w(V'_{cl}(s'),u')} \, du'duds' ds\notag\\
	&\quad+C\int_0^t \int_0^s  \int_0^{s'} \int_{\abs{u}\leq 2N} \int_{\abs{u'}\leq 3N} 
	e^{-\frac{\nu_0}{2}(t-s'')} \abs{k_w(V_{cl}(s),u)k_w(V'_{cl}(s'),u')} \int_{\R^3} \abs{k_w(V'_{cl}(s''),u'')h(s'',X'_{cl}(s''),u'')} du''\notag\\
	&\quad+C\int_0^t \int_0^s  \int_0^{s'} \int_{\abs{u}\leq 2N} \int_{\abs{u'}\leq 3N} 
	e^{-\frac{\nu_0}{2}(t-s'')} \abs{k_w(V_{cl}(s),u)k_w(V'_{cl}(s'),u')}\notag\\
	&\quad \quad \times \frac{\norm{h(s')}_{L^\infty}}{1+\abs{u'}} \left ( \int_{\R^3} (1+\abs{u''})^{-4\beta +4} 
	\abs{h(s'',X'_{cl}(s''),u'')}^2 du'' \right)^{\frac{1}{2}}du'duds''ds' ds\notag\\
	&:=I_{221}+I_{222}+I_{223}.
\end{align}}
To estimate terms in \eqref{L_inf_split.3(n)}, it is direct to deduce from \eqref{k_prop.1} that 
\begin{align}\label{I_221}
	I_{221} \leq Ce^{-\frac{\nu_0}{4} t} \norm{h_0}_{L^\infty}. 
\end{align}
For the large velocity region $\{ \abs{u''} \geq 4N\}$, it follows from \eqref{k_prop.1(n)} and \eqref{k_prop.1} that 
\begin{align} \label{I_222.1}
	&C \int_0^t \int_0^s  \int_0^{s'} \int_{\abs{u}\leq 2N} \int_{\abs{u'}\leq 3N} 
	e^{-\frac{\nu_0}{2}(t-s'')} \abs{k_w(V_{cl}(s),u)k_w(V'_{cl}(s'),u')} \notag\\
&\qquad\qquad\qquad\qquad\qquad\qquad\times\int_{\abs{u''} \geq 4N} 
	\abs{k_w(V'_{cl}(s''),u'')h(s'',X'_{cl}(s''),u'')} du'' \leq \frac{C}{N}\sup_{0\leq s \leq t} \norm{h(s)}_{L^\infty}. 
\end{align}
Thus, we only consider the case $\{ \abs{u''} \leq 4N \}$ in $I_{222}$. Using Proposition \ref{prop_COV}, we see that
\begin{equation*}
	\begin{split}
		&\exists \; i_s \in \{1,2,\dots, I_{\O,N}\} \; \textrm{such that} \; X(s;t,x,v) \in \mathcal{O}_{i_s}, \\
		&\exists \; j_{s,s'} \in \{1,2, \dots, I_{\O,N} \} \; \textrm{such that} \; X(s';s,X(s;t,x,v),u) \in \mathcal{O}_{j_{s,s'}},
	\end{split}
\end{equation*}
and then we can define the following sets for fixed $n,\vec{n}, i,k,m,\vec{m},j$, and $k'$ as in Proposition \ref{prop_COV}
\begin{equation*}
	\begin{split}
		&R_1:= \left \{u \; \vert \; u \notin  B\left(\vec{n} \delta ;2\delta \right) \cap \left \{ \R^3 \backslash \mathcal{V}_{i}(\hat{e}_1,\hat{e}_2) \right \} \right \}\\ 
		&R_2:= \left \{ s' \; \vert \; \abs{s-s'} \leq \delta \right \} \\
		&R_3:= \left \{ s' \; \vert \; \abs{s'-\psi_1^{n,\vec{n},i,k,m,\vec{m},j,k'}(n\delta, X(n\delta;t,x,v),\vec{n}\delta)} \lesssim_{N} \delta \norm{\psi_1}_{C^1}  \right \} \\
		&R_4:= \left \{ u^{\prime} \; \vert \; u^{\prime} \notin B(\vec{m}\delta;2\delta) \cap \left \{ \R^3 \backslash \mathcal{V}_{j_{s,s'}} (\partial_{\abs{u}} X,\partial_{\hat{u}_1} X \right \} \right \} \\
		&R_5:= \left \{ s'' \; \vert \; \abs{s'-s''} \leq \delta \right \} \\
		&R_6:= \left \{ s'' \; \vert \; \min_{r=1,2} \abs{ s'' - \psi_{r}^{n,\vec{n},i,k,m,\vec{m},j,k'}\left (m\delta,X(m\delta;n \delta,X(n\delta;t,x,v),\vec{m}\delta),\vec{n}\delta \right)}\lesssim_{N} \delta  \min_{r=1,2} \norm{\psi_r}_{C^1} \right \}. 
	\end{split}
\end{equation*}
Therefore, 
it holds that
{\footnotesize\begin{align} \label{G_f decay split.2(n)}
		I_{222}\textbf{1}_{\abs{u''}\leq 4N}	
		&= \sum_{n=1}^{[t / \delta]} \sum_{\abs{\vec{n}} \leq N} \sum_{m=1}^{[ t/\delta]} \sum_{\abs{\vec{m}} \leq N} 
		\sum_{k}^{K_{i_s}} \sum_{k'}^{K'_{j_{s,s'}}}\int_{(n-1)\delta}^{(n+1)\delta} \int_{t^{k+1}+\delta} ^{t^k-\delta} \int_{t^{k'+1}+\delta}^{t^{k'}-
		\delta} e^{-\frac{\nu_0}{2}(t-s'')} \notag\\ 
		&\quad \times \int_{\abs{u}\leq 2N, \abs{u'}\leq 3N, \abs{u''}\leq 4N}  \abs{k_w(V_{cl}(s),u) k_w(V'_{cl}(s'),u') k_w(V'_{cl}(s''),u'')h(s'',X'_{cl}(s''),u'')}  \textbf{1}_{R_1^{c} \cap \cdots \cap R_6^{c}}\notag\\
		&\qquad + \mathcal{B}_{1} + \mathcal{R}_{1},
\end{align}}
where the $\mathcal{B}_{1}$ term corresponds to where the trajectory is near bouncing points and $\mathcal{R}_{1}$ corresponds to where $(u,s',u',s'')$ is in one of $R_1$ through $R_6$. So we have the following smallness estimates for $\mathcal{B}_{1}$ and $\mathcal{R}_{1}$: 
{\footnotesize\begin{align} \label{B,R}
		\mathcal{B}_{1} &\leq \int_0^t \int_0^s \int_0^{s'}e^{-\frac{\nu_0}{2}(t-s'')} \int_{\abs{u}\leq 2N}  \int_{\abs{u'} \leq 3N} 
		\abs{k_w(V_{cl}(s),u)k_w(V'_{cl}(s'),u')} \notag\\ 
		& \quad\times \int_{\abs{u''}\leq 4N} \abs{k_w(V'_{cl}(s''),u'')h(s'',X'_{cl}(s''),u'')} 
		\textbf{1}_{\{t^{k+1} \leq s' \leq t^{k+1}+\delta\} \cup \{t^k-\delta \leq s' \leq t^k\}} 
		\textbf{1}_{\{t^{k'+1} \leq s'' \leq t^{k'+1}+\delta\} \cup \{t^{k'}-\delta \leq s'' \leq t^{k'} \}} \notag\\
		&\leq C\delta \sup_{0 \leq s \leq t} \norm{h(s)}_{L^\infty}, \notag\\ 
		\mathcal{R}_{1} &\leq  \int_0^t \int_0^s\int_0^{s'} e^{-\frac{\nu_0}{2} (t-s'')} \int_{\abs{u}\leq 2N}  \int_{\abs{u'}\leq 3N} \abs{k_w(V_{cl}(s),u)k_w(V'_{cl}(s'),u')} \notag\\
		 &\quad\times\int_{\abs{u''} \leq 4N} \abs{k_w(V'_{cl}(s''), u'') h(s'',X'_{cl}(s''),u'')} \textbf{1}_{R_1 \cup R_2 \cup \cdots \cup R_6}\notag\\ 
		&\leq C \delta \sup_{0\leq s \leq t} \norm{h(s)}_{L^\infty}. 
\end{align}}
Now let us consider the first term on the right-hand side of \eqref{G_f decay split.2(n)}. 
Under the conditions of 
\begin{align*}
	(u,s',u',s'') &\in R_1^c \cap R_2^c \cap \cdots \cap R_6^c, \\ 
	t & \in \left [(n-1)\delta, (n+1)\delta \right], \\ 
	X(s;t,x,v) &\in \mathcal{O}_{i_s}, \\
	X(s';s,X(s;t,x,v),u) & \in \mathcal{O}_{j_{s,s'}}, \\
	u & \in B(\vec{n}\delta ;2\delta) \cap \R^3 \backslash \mathcal{V}_{i_s} (\hat{e}_1,\hat{e}_2), \\ 
	u' & \in B(\vec{m} \delta; 2\delta) \cap \R^3 \backslash \mathcal{V}_{j_{s,s'}} (\partial_{\abs{u}} X, \partial_{\hat{u}_1} X),  
\end{align*} 
applying Proposition \ref{prop_COV} gives a lower bound of determinant which comes from a change of variables:
\begin{align*}
	\left \vert \det \left( \frac{\partial (X(s''))}{\partial(\abs{u'},\xi_1,\xi_2)}\right) \right \vert \geq \epsilon'_{\delta},
\end{align*}
where $\epsilon'_{\delta}$ does not depend on time. Note that $ \xi_1$ and $\xi_2$ are the variables selected from $\{ \abs{u}, \hat{u}_1 , \hat{u}'_1, \hat{u}'_2\}$ in Proposition \ref{prop_COV} and the remaining variables among $\{ \abs{u}, \hat{u}_1 , \hat{u}'_1, \hat{u}'_2\}$ are $\xi_3$ and $\xi_4$. Let $\mathcal{P}$ be the projection map from $B(\vec{n}\delta ; 2\delta) \cap \R^3 \backslash \mathcal{V}_{i_s} (\hat{e}_1,\hat{e}_2) \times B(\vec{m}\delta; 2\delta) \cap \R^3 \backslash \mathcal{V}_{j_{s,s'}} (\partial_{\abs{u}} X, \partial_{\hat{u}_1} X) $ into $\R^3$, which corresponds to the $(\abs{u'},\xi_1,\xi_2)$ components. For the sufficiently small $\delta>0$, there exist small $r_{\delta,n,\vec{n},i,k,m,\vec{m}, j,k'} $ such that there exists a one-to-one mapping
\begin{multline*}
	\mathcal{M}: \; \mathcal{P}\left ( B(\vec{n}\delta ;2\delta) \cap \R^3 \backslash \mathcal{V}_{i_s} (\hat{e}_1,\hat{e}_2) \times
	B(\vec{m}\delta ; 2\delta) \cap \R^3 \backslash \mathcal{V}_{j_{s,s'}} (\partial_{\abs{u}} X, \partial_{\hat{u}_1}X) \right) \\ 
	\mapsto B \left ( X(s'';s',X(s';s,X(s;t,x,v),u),u') ,r_{\delta,n,\vec{n},i,k,m,\vec{m},j,k'} \right). 
\end{multline*}
Now, we perform a change of variable for the first term on the right-hand side of \eqref{G_f decay split.2(n)}. Note that \eqref{k_esti.0} in Lemma 2.1 yields
\begin{align} \label{k_prop.2(n)}
	\int_{\abs{u} \leq N} \abs{k_w(v,u)}^2 du \leq C. 
\end{align}
Applying the H\"{o}lder's inequality and \eqref{k_prop.2(n)}, we get the following estimate for the first term on the right-hand side of \eqref{G_f decay split.2(n)} as 
{\footnotesize\begin{align} \label{I_222.2}
	\sum_{n=1}^{[t/\delta]} & \sum_{\abs{\vec{n}} \leq N} \sum_{m=1}^{[t/\delta]} \sum_k^{K_{i_s}} \sum_{k'}^{K'_{j_{s,s'}}} \int_{(n-1)\delta}^{(n+1)\delta} \int_{t^{k+1}+\delta}^{t^k-\delta}  \int_{t^{k'+1}+\delta} ^{t^{k'}-\delta}  e^{-\frac{\nu_0}{2}(t-s'')} \notag\\
	&\times \left ( \int_{\abs{u}\leq 2N, \abs{u'}\leq 3N, \abs{u''}\leq 4N}  \abs{k_w(V_{cl}(s),u)k_w(V'_{cl}(s'),u') k_w(V'_{cl}(s''),u'')}^2 \; du'' du' du \right )^{\frac{1}{2}}  \notag\\
	&\times \left(\int_{\abs{u}\leq 2N, \abs{u'}\leq 3N,\abs{u''}\leq 4N} \abs{h(s'',X'_{cl}(s''),u'')}^2 
	 \textbf{1}_{R_1^c \cap \cdots \cap R_6^c} \; du'' du' du \right)^{\frac{1}{2}} ds'' ds'ds\notag\\
	 &\leq C \sum_{n=1}^{[t/\delta]} \sum_{\abs{\vec{n}} \leq N} \sum_{m=1}^{[t/\delta]} \sum_k^{K_{i_s}} \sum_{k'}^{K'_{j_{s,s'}}} \int_{(n-1)\delta}^{(n+1)\delta} \int_{t^{k+1}+
	 \delta}^{t^k-\delta} \int_{t^{k'+1}+\delta} ^{t^{k'}-\delta}  e^{-\frac{\nu_0}{2}(t-s'')} \notag\\
	&\times \left( \int_{u'' \in \R^3} \int_{\hat{u}_2,\xi_3,\xi_4} \textbf{1}_{\{ \abs{u}\leq 2N, \abs{u'}\leq 3N, \abs{u''}\leq 4N \}} \; d\hat{u}_2d\xi_3 d\xi_4
	  \right)^{\frac{1}{2}}
	  \left ( \int_{\abs{u'},\xi_1,\xi_2} \abs{h(s'',X'_{cl}(s''),u'')}^2  \; d\abs{u'}d\xi_1d\xi_2
	   \right) ^{\frac{1}{2}} ds''ds'ds \notag\\ 
	   &\leq C_{N,\delta,\O} \sum_{n=1}^{[t/\delta]} \sum_{\abs{\vec{n}} \leq N} \sum_{m=1}^{[t/\delta]} \sum_k^{K_{i_s}} \sum_{k'}^{K'_{j_{s,s'}}}
	    \int_{(n-1)\delta}^{(n+1)\delta}\int_{t^{k+1}+\delta}^{t^k-\delta} \int_{t^{k'+1}+\delta} ^{t^{k'}-\delta}  e^{-\frac{\nu_0}{2}(t-s'')}\notag\\
	    &\times \left(\int_{\abs{u''}\leq 4N } \int_{B(X(s''),r_{\delta,n,\vec{n},i,k,m,\vec{m},j,k'})} \abs{h(s'',y,u'')}^2 
	    \frac{1}{\epsilon'_{\O,N,\delta}}  \; dydu'' \right)^{\frac{1}{2}}\notag\\
	    &\leq C_{N,\delta,\O} \int_0^t \int_0^s\int_0^{s'} e^{-\frac{\nu_0}{2}(t-s'')} \left( \int_{y \in \O} \int_{\abs{u''}\leq 4N} \abs{h(s'',y,u'')}^2 \; du'' dy\right)^{\frac{1}{2}} \; ds''ds'ds.
\end{align}}
From Lemma \ref{entropy_est} and Young's inequality, one obtains that
{\footnotesize\begin{align} \label{entropy.1(n)}
	&C_{N,\delta,\O} \int_{y \in \O} \int_{\abs{u''}\leq 4N} \abs{h(s'',y,u'')}^2 du''dy \notag\\ 
	&=C_{N,\delta,\O} \left(\int_{y\in \O} \int_{\abs{u''} \leq 4N} \abs{h(s'',y,u'')}^2 \textbf{1}_{\abs{F-\mu}\leq \mu} \; du'' dy
	+\int_{y\in \O} \int_{\abs{u''} \leq 4N} \abs{h(s'',y,v'')}^2 \textbf{1}_{\abs{F-\mu}> \mu} \; du'' dy\right) \notag\\
	&\leq C_{N,\delta,\O} \left(\int_{y\in \O} \int_{\abs{u''} \leq 4N} \abs{f(s'',y,u'')}^2 \textbf{1}_{\abs{F-\mu}\leq \mu} \; du'' dy 
	+\sup_{0 \leq s \leq t} \norm{h(s)}_{L^\infty} \int_{y\in \O} \int_{\abs{u''} \leq 4N} \abs{h(s'',y,u'')} \textbf{1}_{\abs{F-\mu}>\mu} \; du'' dy \right) \notag\\
	&\leq C_{N,\delta,\O} \mathcal{E}(F_0) + C_{N,\delta,\O}\sup_{0 \leq s \leq t} \norm{h(s)}_{L^\infty} \int_{y\in \O} \int_{\abs{u''} \leq 4N} \sqrt{\mu(u'')}\abs{f(s'',y,u'')} 
	\textbf{1}_{\abs{F-\mu}>\mu} \; du'' dy \notag\\
	&\leq C_{N,\delta,\O} \mathcal{E}(F_0) + C_{N,\delta,\O} \sup_{0\leq s \leq t}\norm{h(s)}_{L^\infty} \mathcal{E}(F_0)\notag\\
	&\leq \frac{C_{\O}}{N} \sup_{0 \leq s \leq t} \norm{h(s)}_{L^\infty}^2 + C_{N,\delta,\O} [\mathcal{E}(F_0)+\mathcal{E}(F_0)^2].
\end{align}}
Thus, from \eqref{G_f decay split.2(n)}, \eqref{B,R}, \eqref{I_222.2} and \eqref{entropy.1(n)}, we have the following estimate 
\begin{equation} \label{I_222.3}
		I_{222} \textbf{1}_{\abs{u''}\leq 4N} \leq C_{\O}\left(\delta+\frac{1}{\sqrt{N}}\right)\sup_{0\leq s \leq t} \norm{h(s)}_{L^\infty} + C_{N,\delta,\O} A_0. 
\end{equation}
Here and to the end we have denoted $A_0=\sqrt{\mathcal{E}(F_0)}+\mathcal{E}(F_0)$.
Hence, it follows from \eqref{I_222.1} and \eqref{I_222.3} that
\begin{align} \label{I_222}
	I_{222} \leq C_{\O} \left(\delta+\frac{1}{\sqrt{N}}\right) \sup_{0\leq s \leq t} \norm{h(s)}_{L^\infty} + 
	C_{N,\delta,\O} A_0. 
\end{align}
It remains to estimate $I_{223}$ in \eqref{L_inf_split.3(n)}. Firstly, over $\{\abs{u''}\geq 4N\}$, $I_{223}$ is further bounded by 
\begin{align} \label{I_223.1}
	&I_{223} \textbf{1}_{\{\abs{u''}\geq 4N\}}\nonumber \\
	&\leq C\sup_{0 \leq s \leq t} \norm{h(s)}_{L^\infty}^2 
	 \int_0^t \int_0^s  \int_0^{s'}e^{-\frac{\nu_0}{2}(t-s'')}\left( \int_{\abs{u''} \geq 4N} 
	 (1+\abs{u''})^{-4\beta+4} du''\right)^{\frac{1}{2}} ds''ds'ds \nonumber \\ 
	 &\leq \frac{C}{N} \sup_{0 \leq s \leq t} \norm{h(s)}_{L^\infty}^2, 
\end{align}
where we have used $\beta \geq \frac{5}{2}$ and \eqref{k_prop.1}. Firstly, we can further bound $I_{223}$ over $\{\abs{u''} \leq 4N\}$ as 
\begin{align} \label{I_223.2} 
	&I_{223}\textbf{1}_{\{\abs{u''}\leq 4N\}}  \notag\\ 
	&\leq C \int_0^t \int_0^s  \int_0^{s'} \int_{\abs{u}\leq 2N} \int_{\abs{u'}\leq 3N} 
	e^{-\frac{\nu_0}{2}(t-s'')} \abs{k_w(V_{cl}(s),u)k_w(V'_{cl}(s'),u')}  \notag\\
	&\quad \times \norm{h(s')}_{L^\infty} 
	\left(\int_{\abs{u''}\leq 4N} (1+\abs{u''})^{-4\beta+4} \abs{h(s'',X'_{cl}(s''),u'')}^2 du'' \right)^{\frac{1}{2}}du'duds''ds'ds. 
\end{align}
Due to $(1+\abs{u''})^{-4\beta+4} \in L^1(\R^3)\cap L^2(\R^3)$ for $\beta \geq 5/2$, we can further bound \eqref{I_223.2} by using similar arguments \eqref{G_f decay split.2(n)}, \eqref{B,R}, \eqref{I_222.2} and \eqref{entropy.1(n)} as 
\begin{align} \label{I_223.3} 
 	I_{223}\textbf{1}_{\{\abs{u''} \leq 4N \}} \leq C_{\O} \left(\delta +\frac{1}{\sqrt{N}}\right) 
	\sup_{0\leq s \leq t} \norm{h(s)}_{L^\infty} ^2 + C_{N,\delta,\O}A_0^2.
\end{align}
Here the details are omitted for brevity.
By \eqref{I_223.1}, \eqref{I_223.2} and \eqref{I_223.3}, it holds that  
\begin{align} \label{I_223}
	I_{223} \leq  C_{\O} \left(\delta +\frac{1}{\sqrt{N}} \right) \sup_{0\leq s \leq t} \norm{h(s)}_{L^\infty}^2 +
	C_{N,\delta,\O} A_0^2.
\end{align}
In sum, combining \eqref{L_inf_split.2(n)}, \eqref{I_22.1}, \eqref{L_inf_split.3(n)}, \eqref{I_221}, \eqref{I_222} and \eqref{I_223}, we obtain 
\begin{align} \label{I_22}
	I_{22} \leq C e^{-\frac{\nu_0}{4} t} \norm{h_0}_{L^\infty}+ C_{\O} \left(\delta+\frac{1}{\sqrt{N}}\right) \sup_{0\leq s \leq t}\norm{h(s)}_{L^\infty}^2 + C_{N,\delta,\O} A_0.
\end{align}
Next, we estimate $I_{23}$ in \eqref{L_inf_esti.3} to complete estimates on $I_2$ in \eqref{L_inf_esti.1}.
Similarly, we firstly divide the integration domain of $I_{23}$ in \eqref{L_inf_esti.3} as 
\begin{align} \label{L_inf_split.4(n)}
	I_{23}= \int_0^t \int_0^s \int_{\abs{u} \leq 2N} \int_{\abs{u'} \geq 3N} + \int_0^t\int_0^s \int_{\abs{u} \leq 2N} \int_{\abs{u'} \leq 3N}.
\end{align}
Note that \eqref{new kernel} and \eqref{new kernel_0} imply that there is a suitably small $C_2=C_2(\varpi)>0$ such that 
\begin{equation} \label{k_prop.3(n)}
	\int_{\R^3} \abs{\tilde{k}_2(u,u')} e^{C_2\abs{u-u'}^2} \, du' \leq C, 
\end{equation}
where $C$ is a generic constant. 
In the domain $\{\abs{u}\leq 2N,\abs{u'}\geq 3N\}$, it follows from \eqref{k_prop.3(n)} that 
\begin{align} \label{k_prop.4(n)} 
	\int_{\abs{u'} \geq 3N} \abs{\tilde{k}_2(u,u')} du' \leq e^{-C_2 N^2} \int_{\abs{u'}\geq 3N} \abs{\tilde{k}_2(u,u')}e^{C_2\abs{u-u'}^2} \, du' \leq \frac{C}{N}. 
\end{align}
 By \eqref{k_prop.1} and \eqref{k_prop.4(n)}, it is directly deduced that the first part in \eqref{L_inf_split.4(n)} is further bounded by
\begin{align} \label{I_23.1}
	\begin{split}
	 &C\int_0^t \int_0^s \int_{\abs{u}\leq 2N}e^{-\frac{\nu_0}{2}(t-s')}\abs{k_w(V_{cl}(s),u)} 
	\int_{\abs{u'}\geq 3N} \abs{\tilde{k}_2(V'_{cl}(s'),u')h^2(s',X'_{cl}(s'),u')}du'duds'ds \\
	&\leq \frac{C}{N}  \sup_{0\leq s \leq t} \norm{h(s)}_{L^\infty}^2.
	\end{split}
\end{align} 
By \eqref{L_inf_esti.1}, the second part in \eqref{L_inf_split.4(n)} can be further bounded by
{\footnotesize\begin{align} \label{L_inf_split.5(n)}
	 &C\int_0^t \int_0^s \int_{\abs{u}\leq 2N}e^{-\frac{\nu_0}{2}(t-s')}\abs{k_w(V_{cl}(s),u)} 
	\int_{\abs{u'}\leq 3N} \abs{\tilde{k}_2(V'_{cl}(s'),u')h^2(s',X'_{cl}(s'),u')}du'duds'ds \notag\\
	&\leq C\norm{h_0}_{L^\infty}^2 \int_0^t \int_0^s\int_{\abs{u}\leq 2N}  
	e^{-\frac{\nu_0}{2} t} \abs{k_w(V_{cl}(s),u)}  \int_{\abs{u'}\leq 3N}
	 \abs{\tilde{k}_2(V'_{cl}(s'),u')}du'  \, duds' ds\notag\\
	&\quad+C \int_0^t \int_0^s \int_{\abs{u}\leq 2N} 
	e^{-\frac{\nu_0}{2}(t-s')} \abs{k_w(V_{cl}(s),u)} \notag\\
	&\quad \quad \times \int_{\abs{u'}\leq 3N} \abs{\tilde{k}_2(V'_{cl}(s'),u')} \left[\int_0^{s'} 
	e^{-\frac{\nu_0}{2}(s'-s'')}\int_{\R^3} \abs{k_w(V'_{cl}(s''),u'')h(s'',X'_{cl}(s''),u'')} du''ds''\right]^2du' 
	duds'ds\notag\\ 
	&\quad+C \int_0^t \int_0^s \int_{\abs{u}\leq 2N} 
	e^{-\frac{\nu_0}{2}(t-s')} \abs{k_w(V_{cl}(s),u)}\notag\\
	& \quad \quad \times   \int_{\abs{u'}\leq 3N} \abs{\tilde{k}_2(V'_{cl}(s'),u')} 
	\int_0^{s'}e^{-\frac{\nu_0}{2}(s'-s'')} \norm{h(s'')}_{L^\infty}^2 \int_{\R^3} (1+\abs{u''})^{-4\beta+4} \abs{h(s'',X'_{cl}(s''),u'')}^2 du''ds''du' duds'ds\notag\\
	&:=I_{231}+I_{232}+I_{233}.
\end{align}} 
From \eqref{k_prop.3(n)}, we have 
\begin{align} \label{k_prop.5(n)}
	\int_{\R^3} \abs{\tilde{k}_2(V'_{cl}(s'),u')} \, du' \leq C.
\end{align}
It is direct to verify that 
\begin{align} \label{I_231}
	I_{231} \leq C e^{-\frac{\nu_0}{4} t} \norm{h_0}_{L^\infty}^2,
\end{align}
where we have used \eqref{k_prop.1} and \eqref{k_prop.5(n)}. We split the velocity integration into $\{\abs{u''}\geq 4N\} \cup \{\abs{u''}\leq 4N\}$ and we can further bound $I_{232}$ over $\{\abs{u''} \geq 4N\}$  as 
\begin{align} \label{I_232.1}
	I_{232} \textbf{1}_{\{\abs{u''} \geq 4N \}} \leq \frac{C}{N}\sup_{0\leq s \leq t} \norm{h(s)}_{L^\infty}^2, 
\end{align}
from \eqref{k_prop.1(n)}, \eqref{k_prop.1} and \eqref{k_prop.5(n)}. We apply H\"{o}lder's inequality to $I_{232}$ over $\{\abs{u''}\leq 4N\}$ to obtain
{\footnotesize\begin{align} \label{I_232.2}
	&I_{232}\textbf{1}_{\{\abs{u''} \leq 4N\}} \nonumber \\
	&\leq C\int_0^t \int_0^s \int_{\abs{u}\leq 2N} 
	e^{-\frac{\nu_0}{2}(t-s')} \abs{k_w(V_{cl}(s),u)}
	\int_{\abs{u'}\leq 3N} \abs{\tilde{k}_2(V'_{cl}(s'),u')} \nonumber \\
	& \quad\times\left[\int_0^{s'} 
	e^{-\frac{\nu_0}{2}(s'-s'')}\left(\int_{\abs{u''}\leq 4N} \abs{k_w(V'_{cl}(s''),u'')}du''\right)^{\frac{1}{2}}
	\left(\int_{\abs{u''}\leq 4N} \abs{k_w(V'_{cl}(s''),u'')}\abs{h(s'',X'_{cl}(s''),u'')}^2 du''\right)^{\frac{1}{2}}ds''\right]^2
	 du' \nonumber \\
	&\leq C \int_0^t \int_0^s\int_0^{s'} \int_{\abs{u}\leq 2N} e^{-\frac{\nu_0}{2}(t-s'')} 
	 \abs{k_w(V_{cl}(s),u)}  \int_{\abs{u'}\leq 3N} \abs{\tilde{k}_2 (V'_{cl}(s'),u')} 
	 \int_{\abs{u''}\leq 4N} \abs{k_w(V'_{cl}(s'),u'')} \abs{h(s'',X_{cl}'(s''),u'')}^2,  
\end{align}}
where we have used \eqref{k_prop.1}. Notice that we have 
\begin{align} \label{k_prop.6(n)}
	\int_{\R^3} \abs{\tilde{k}_2(u,u')}^2 du' \leq C, 
\end{align}
similar as \eqref{k_prop.5(n)} which comes from \eqref{new kernel} and \eqref{new kernel_0}. We have $\tilde{k}_2 \in L^1(\R^3)\cap L^2(\R^3)$ such as the kernel $k_w$ by \eqref{k_prop.5(n)} and \eqref{k_prop.6(n)}. By using similar arguments  \eqref{G_f decay split.2(n)}, \eqref{B,R}, \eqref{I_222.2} and \eqref{entropy.1(n)}, it holds that
\begin{align} \label{I_232.3} 
	 &C \int_0^t \int_0^s\int_0^{s'} \int_{\abs{u}\leq 2N} e^{-\frac{\nu_0}{2}(t-s'')} 
	 \abs{k_w(V_{cl}(s),u)} \notag\\ 	 &\qquad\qquad\qquad\qquad\qquad\qquad\times\int_{\abs{u'}\leq 3N} \abs{\tilde{k}_2 (V'_{cl}(s'),u')} 
	 \int_{\abs{u''}\leq 4N} \abs{k_w(V'_{cl}(s'),u'')} \abs{h(s'',X_{cl}'(s''),u'')}^2 \nonumber \\
	 &\leq C_{\O} \left(\delta + \frac{1}{\sqrt{N}}\right) \sup_{0\leq s \leq t}
	\norm{h(s)}_{L^\infty}^2 + C_{N,\delta,\O} A_0^2.
\end{align}
From \eqref{I_232.1}, \eqref{I_232.2} and \eqref{I_232.3}, we have the following estimate 
\begin{align} \label{I_232}
	I_{232} \leq C_{\O}\left(\delta+\frac{1}{\sqrt{N}}\right) \sup_{0\leq s \leq t}\norm{h(s)}_{L^\infty}^2+C_{N,\delta,\O}A_0^2.
\end{align}
Likewise, we split 
\begin{align} \label{L_inf_split.6(n)}
	I_{233}=I_{233} \textbf{1}_{\abs{u''}\geq 4N} + I_{233} \textbf{1}_{\abs{u''}\leq 4N}. 
\end{align}
We deduce from \eqref{k_prop.1(n)} and \eqref{k_prop.5(n)} that 
\begin{align} \label{I_233.1} 
	I_{233} \textbf{1}_{\abs{u''}\geq 4N} \leq \frac{C}{N}  \sup_{0\leq s \leq t} \norm{h(s)}_{L^\infty}^4,  
\end{align}
where $\beta \geq 5/2$ has been used. For $I_{233} \textbf{1}_{\abs{u''}\leq 4N}$ in \eqref{L_inf_split.6(n)}, similarly as for treating $I_{223}$ before, it holds that 
\begin{align} \label{I_233.2}
	I_{233} \textbf{1}_{\abs{u''}\leq 4N} \leq C_{\O} \left(\delta+\frac{1}{\sqrt{N}}\right) \sup_{0\leq s \leq t} \norm{h(s)}_{L^\infty}^4 +C_{N,\delta,\O}A_0.
\end{align}
Thus, it follows from \eqref{I_233.1} and \eqref{I_233.2} that 
\begin{align} \label{I_233}
	I_{233} \leq C_{\O} \left(\delta+\frac{1}{\sqrt{N}}\right) \sup_{0\leq s \leq t} \norm{h(s)}_{L^\infty}^4 + C_{N,\delta,\O}A_0.
\end{align}
Combining \eqref{L_inf_split.4(n)}, \eqref{I_23.1}, \eqref{L_inf_split.5(n)}, \eqref{I_231}, \eqref{I_232} and \eqref{I_233} yields that
\begin{align} \label{I_23}
	I_{23} \leq Ce^{-\frac{\nu_0}{4}t} \norm{h_0}_{L^\infty}^2+C_{\O} \left(\delta +\frac{1}{\sqrt{N}}\right) \sup_{0\leq s \leq t} \norm{h(s)}_{L^\infty}^4 + C_{N,\delta,\O} A_0.
\end{align}
In summary, combining \eqref{I_2.1}, \eqref{L_inf_esti.3}, \eqref{L_inf_esti.4}, \eqref{I_22} and \eqref{I_23}, we obtain the estimate on $I_2$ in \eqref{L_inf_split.1(n)} by
\begin{align}
	\begin{split} \label{L_inf_esti.1(n)}
	I_2 
	\leq Ce^{-\frac{\nu_0}{4}t}[\norm{h_0}_{L^\infty}+\norm{h_0}_{L^\infty}^2]+C_{\O} \left(\delta+\frac{1}{\sqrt{N}}\right) \sup_{0\leq s \leq t}\norm{h(s)}_{L^\infty}^4 + C_{N,\delta,\O} A_0.
	\end{split}
\end{align}

\noindent\underline{\textit{Estimate  on $I_3$}}:  Likewise, we split the velocity $u$-integration of $I_3$ in \eqref{L_inf_esti.1} into 
$\{\abs{u}\leq 2N\} \cup \{ \abs{u} \geq 2N\}$: 
\begin{align} \label{L_inf_split.7(n)}
	\begin{split}
	C\int_0^t e^{-\frac{\nu_0}{2} (t-s)} \frac{\norm{h(s)}_{L^\infty}}{1+\abs{v}} 
	\bigg(\int_{\R ^3} (1+\abs{u})^{-4\beta+4} \abs{h(s,X_{cl}(s),u)}^2 \, &du\bigg)^{\frac{1}{2}}\, ds 
	=\int_0^t\int_{\abs{u}\geq 2N} + \int_0^t \int_{\abs{u}\leq 2N}. 
	\end{split}
\end{align}
By a direct computation, the $\{|u| \geq 2N\}$ part in \eqref{L_inf_split.7(n)} can be controlled by  
\begin{align} \label{L_inf_esti.2(n)}
	\begin{split}
	C\int_0^t e^{-\frac{\nu_0}{2} (t-s)} \frac{\norm{h(s)}_{L^\infty}}{1+\abs{v}} 
	\bigg(\int_{\abs{u}\geq 2N} (1+\abs{u})^{-4\beta+4} \abs{h(s,X_{cl}(s),u)}^2 \,  du \bigg)^{\frac{1}{2}} \, ds
	\leq \frac{C}{N} \sup_{0\leq s \leq t} \norm{h(s)}_{L^\infty}^2, 
	\end{split}
\end{align}	
where $\beta \geq 5/2$ has been used. Applying \eqref{L_inf_esti.3} to $h(s,X_{cl}(s),u)$ in the second term of \eqref{L_inf_split.7(n)} gives 
\begin{align} \label{L_inf_split.8(n)}
	&C\int_0^t e^{-\frac{\nu_0}{2} (t-s)} \frac{\norm{h(s)}_{L^\infty}}{1+\abs{v}} 
	\left(\int_{\abs{u}\leq 2N} (1+\abs{u})^{-4\beta+4} \abs{h(s,X_{cl}(s),u)}^2 \, du\right)^{\frac{1}{2}}\, ds \nonumber \\
	&\leq C\norm{h_0}_{L^\infty} \int_0^t e^{-\frac{\nu_0}{2} t}\norm{h(s)}_{L^\infty}
	\left(\int_{\abs{u}\leq 2N} (1+\abs{u})^{-4\beta+4} du\right)^{\frac{1}{2}}ds\nonumber \\
	&\quad+C \int_0^t e^{-\frac{\nu_0}{2}(t-s)} \norm{h(s)}_{L^\infty}\nonumber\\
	&\quad \quad \times \left(\int_{\abs{u}\leq 2N} (1+\abs{u})^{-4\beta+4} \left[\int_0^s e^{-\frac{\nu_0}{2}(s-s')}\int_{\R^3} \abs{k_w(V'_{cl}(s'),u')h(s',X'_{cl}(s'),u')}du'ds' \right]^2 du\right)^{\frac{1}{2}}ds \nonumber\\
	&\quad+C \int_0^t e^{-\frac{\nu_0}{2}(t-s)} \norm{h(s)}_{L^\infty} \nonumber\\
	&\quad \quad \times \left(\int_{\abs{u}\leq 2N} (1+\abs{u})^{-4\beta+4} \left[\int_0^s e^{-\frac{\nu_0}{2}(s-s')}\int_{\R^3}
	\abs{\tilde{k}_2(V'_{cl}(s'),u')h^2(s',X'_{cl}(s'),u')}du'ds' \right]^2 du\right)^{\frac{1}{2}}ds \nonumber \\
	&:=I_{31}+I_{32}+I_{33}. 
\end{align}
Firstly, for $I_{31}$, it holds that 
\begin{align} \label{I_31}
	I_{31} \leq C e^{-\frac{\nu_0}{4}t}\norm{h_0}_{L^\infty} \int_0^t \norm{h(s)}_{L^\infty} ds ,
\end{align}
where $\beta \geq 5/2$ has been used. Secondly, we divide the velocity $u$-integration domain into $\{\abs{u'}\geq 3N\} \cup \{\abs{u'}\leq 3N\}$ so as to get
\begin{align} \label{L_inf_split.9(n)} 
	I_{32} = \int_0^t \int_0^s \int_{\abs{u}\leq 2N}\int_{\abs{u'}\geq 3N}+ \int_0^t \int_0^s \int_{\abs{u}\leq 2N}\int_{\abs{u'}\leq 3N}. 
\end{align}
It holds from \eqref{k_prop.1(n)} that 
\begin{align} \label{I_32.1} 
	I_{32} \textbf{1}_{\abs{u'} \geq 3N} \leq \frac{C}{N}  \sup_{0\leq s \leq t} \norm{h(s)}_{L^\infty}^2. 
\end{align}
For the second part in \eqref{L_inf_split.9(n)}, we apply the Duhamel formula \eqref{Duhamel}. Then, using \eqref{gamma_gain 1} in Lemma \ref{Gamma+} and Lemma \ref{G_f decay} yields that 
\begin{align} \label{L_inf_split.10(n)}
	&I_{32}\textbf{1}_{\abs{u'}\leq 3N} \notag\\
	&\leq C\int_0^t e^{-\frac{\nu_0}{2}(t-s)}\norm{h(s)}_{L^\infty} \notag\\
	&\quad \quad \times \left(\int_{\abs{u}\leq 2N} (1+\abs{u})^{-4\beta+4} \left[ e^{-\frac{\nu_0}{2}s}\norm{h_0}_{L^\infty}\int_0^s 
	\int_{\abs{u'}\leq 3N} \abs{k_w(V'_{cl}(s'),u')} du'ds' \right]^2 du \right)^{\frac{1}{2}} ds\notag\\
	&\quad +C\int_0^t e^{-\frac{\nu_0}{2}(t-s)}\norm{h(s)}_{L^\infty}  
	 \left(\int_{\abs{u}\leq 2N} (1+\abs{u})^{-4\beta+4}  \right. \notag\\
	 &\left. \quad \quad \times \left[\int_0^s e^{-\frac{\nu_0}{2}(s-s')}
	\int_{\abs{u'}\leq 3N}\abs{k_w(V'_{cl}(s'),u')} \int_0^{s'} e^{-\frac{\nu_0}{2}(s'-s'')}\int_{\R^3}\abs{k_w(V'_{cl}(s''),u'')h(s'',X'_{cl}(s''),u'')}  \right]^2 \right)^{\frac{1}{2}}\notag\\
	&\quad +C\int_0^t e^{-\frac{\nu_0}{2}(t-s)}\norm{h(s)}_{L^\infty}  
	 \left(\int_{\abs{u}\leq 2N} (1+\abs{u})^{-4\beta+4}  \left[\int_0^s e^{-\frac{\nu_0}{2}(s-s')} \right. \right. \notag\\
	&\left. \left. \quad \quad \times \int_{\abs{u'}\leq 3N}\abs{k_w(V'_{cl}(s'),u')}\int_0^{s'} e^{-\frac{\nu_0}{2}(s'-s'')}\norm{h(s'')}_{L^\infty}\left(\int_{\R^3}(1+\abs{u''})^{-4\beta+4}\abs{h(s'',X'_{cl}(s''),u'')}^2\right)^{\frac{1}{2}}  \right]^2 \right)^{\frac{1}{2}}\notag\\
	&:=I_{321}+I_{322}+I_{323}. 
\end{align}
It can be directly deduced from \eqref{k_prop.2(n)} and $\beta \geq 5/2$ that
\begin{align} \label{I_321}
	I_{321} \leq C e^{-\frac{\nu_0}{4} t}\norm{h_0}_{L^\infty} \int_0^t \norm{h(s)}_{L^\infty} ds. 
\end{align}
We divide the velocity $u''$-integration region in $I_{322}$ into $\{\abs{u''} \geq 4N\} \cup \{\abs{u''}\leq 4N\}$. Likewise, in the region $\{\abs{u''} \geq 4N \}$, we can further bound $I_{322}$ as  
\begin{align} \label{I_322.1}
	I_{322} \textbf{1}_{\abs{u''} \geq 4N} \leq \frac{C}{N}\sup_{0\leq s \leq t} \norm{h(s)}_{L^\infty}^2,
\end{align}
where \eqref{k_prop.1(n)} and $\beta \geq 5/2$ have been used. We will concentrate on $I_{322}$ in the region $\{\abs{u''}\leq 4N\}$. The H\"{o}lder's inequality and \eqref{k_prop.1} yield that
\begin{align} \label{Holder}
	&\int_{\abs{u''}\leq 4N} \abs{k_w(V'_{cl}(s'),u'')h(s'',X'_{cl}(s''),u'')} du'' \notag\\
	&\leq 
	\left(\int_{\abs{u''}\leq 4N}  \abs{k_w(V'_{cl}(s'),u'')} du'' \right)^{\frac{1}{2}} 
	\left(\int_{\abs{u''}\leq 4N}  \abs{k_w(V'_{cl}(s'),u'')h^2(s'',X'_{cl}(s''),u'')} du'' \right)^{\frac{1}{2}}\notag\\
	&\leq C \left(\int_{\abs{u''}\leq 4N}  \abs{k_w(V'_{cl}(s'),u'')h^2(s'',X'_{cl}(s''),u'')} du'' \right)^{\frac{1}{2}},
\end{align}
from which we can further bound $I_{322}\textbf{1}_{\abs{u''}\leq 4N}$ from \eqref{k_prop.1} as 
{\footnotesize\begin{align} \label{I_322.2}
	\begin{split}
	&I_{322}\textbf{1}_{\abs{u''}\leq 4N} \\
	&\leq C \int_0^t e^{-\frac{\nu_0}{2}(t-s)}\norm{h(s)}_{L^\infty}  \\
	 &\quad \times \left(\int_{\abs{u}\leq 2N} (1+\abs{u})^{-4\beta+4} \left[\int_0^s \int_0^{s'} e^{-\frac{\nu_0}{2}(s-s'')}
	 \left(\int_{\abs{u'}\leq 3N,\abs{u''}\leq 4N}\abs{k_w(V'_{cl}(s'),u')k_w(V'_{cl}(s''),u'')}du'' du'\right)^{\frac{1}{2}}
	  \right. \right. \\
	  &\left. \left. \quad \times 
	\left(\int_{\abs{u'}\leq 3N,\abs{u''} \leq 4N}
	\abs{k_w(V'_{cl}(s'),u')k_w(V'_{cl}(s''),u'')h^2(s'',X'_{cl}(s''),u'')}du''du'\right)^{\frac{1}{2}}ds''ds' \right]^2 du\right)^{\frac{1}{2}}ds \\
	&\leq C\int_0^t e^{-\frac{\nu_0}{2}(t-s)}\norm{h(s)}_{L^\infty}  \\
         &\quad \times \left(\int_{\abs{u}\leq 2N} (1+\abs{u})^{-4\beta+4} \int_0^s \int_0^{s'} e^{-\nu_0(s-s'')}
	\int_{\abs{u'}\leq 3N,\abs{u''}\leq 4N}\abs{k_w(V'_{cl}(s'),u')k_w(V'_{cl}(s''),u'')h^2(s'',X'_{cl}(s''),u'')}\right)^{\frac{1}{2}}\\
	&\leq C\sup_{0\leq s \leq t} \norm{h(s)}_{L^\infty}\\
	&\quad \times \left(\int_0^t \int_0^s \int_0^{s'} e^{-\nu_0(t-s'')}\int_{\abs{u}\leq 2N,\abs{u'}\leq 3N,\abs{u''}\leq 4N} 
	(1+\abs{u})^{-4\beta+4}\abs{k_w(V'_{cl}(s'),u')k_w(V'_{cl}(s''),u'')h^2(s'',X'_{cl}(s''),u'')}\right)^{\frac{1}{2}}.
	\end{split}
\end{align}}
We treat the last term in \eqref{I_322.2}  as we have dealt with $I_{223}$. One then has
\begin{align} \label{I_322.3}
	I_{322}\textbf{1}_{\abs{u''}\leq 4N} \leq C_{\O} \left (\sqrt{\delta} + \frac{1}{\sqrt{N}} \right) \sup_{0 \leq s \leq t} \norm{h(s)}_{L^\infty}^2 + C_{N,\delta,\O} A_0^2.
\end{align} 
In sum, from \eqref{I_322.1}, \eqref{I_322.2} and \eqref{I_322.3}, we have the following estimate 
\begin{equation} \label{I_322}
	I_{322} \leq C_{\O} \left (\sqrt{\delta} + \frac{1}{\sqrt{N}} \right) \sup_{0 \leq s \leq t} \norm{h(s)}_{L^\infty}^2 + C_{N,\delta,\O} A_0^2.
\end{equation}
Similarly, we split the velocity $u''$-integration domain in  $I_{323}$, and then we can further bound $I_{323}$ in the region $\{\abs{u''}\geq 4N\}$ as 
\begin{align} \label{I_323.1}
	I_{323}\textbf{1}_{\abs{u''}\geq 4N} \leq \frac{C}{N}  \sup_{0\leq s \leq t} \norm{h(s)}_{L^\infty}^3, 
\end{align}
where $\beta \geq 5/2$ and \eqref{k_prop.1} have been used. It remains to estimate $I_{323}$ in the domain $\{\abs{u''}\leq 4N\}$ for $I_{32}$. Similar as \eqref{I_322.2}, it follows from \eqref{k_prop.1} that 
\begin{align}\label{I_323.2}
	&I_{323} \textbf{1}_{\abs{u''}\leq 4N} \notag\\
	&\leq C\sup_{0\leq s \leq t}\norm{h(s)}_{L^\infty}^2 \int_0^t e^{-\frac{\nu_0}{2}(t-s)} 
	 \left(\int_{\abs{u}\leq 2N} (1+\abs{u})^{-4\beta+4}  \left[ \int_0^s \int_0^{s'}e^{-\frac{\nu_0}{2}(s-s'')} \right. \right.  \notag\\
	&\left. \left. \quad \times \left (\int_{\abs{u'}\leq 3N}\abs{k_w(V'_{cl}(s'),u')}
	\int_{\abs{u''}\leq 4N}(1+\abs{u''})^{-4\beta+4}\abs{h(s'',X'_{cl}(s''),u'')}^2 \right)^{\frac{1}{2}} \right]^2\right)^{\frac{1}{2}} \notag\\
	&\leq C \sup_{0\leq s \leq t} \norm{h(s)}_{L^\infty}^2 \Big(\int_0^t\int_0^s \int_0^{s'} e^{-\nu_0(t-s'')}\int_{\abs{u}\leq 2N,\abs{u'}\leq 3N,\abs{u''}\leq 4N} 
	(1+\abs{u})^{-4\beta+4}\notag \\
&\qquad\qquad\qquad\qquad\qquad\qquad\qquad\qquad\abs{k_w(V'_{cl}(s'),u')}(1+\abs{u''})^{-4\beta+4} \abs{h(s'',X'_{cl}(s''),u'')}^2 \Big)^{\frac{1}{2}}.
\end{align}
Since $(1+\abs{u''})^{-4\beta+4} \in L^1(\R^3) \cap L^2(\R^3)$ for $\beta \geq 5/2$, we can deal with the last term in \eqref{I_323.2} as  similar arguments in $I_{322}$. Then, we can deduce that
\begin{align} \label{I_323.3}
	I_{323} \textbf{1}_{\abs{u''}\leq 3N} \leq C_{\O}\left(\sqrt{\delta}+\frac{1}{\sqrt{N}}\right)\sup_{0\leq s \leq t}\norm{h(s)}_{L^\infty}^3+C_{N,\delta,\O} A_0.
\end{align} 
By \eqref{I_323.1},\eqref{I_323.2} and \eqref{I_323.3}, we have 
\begin{align} \label{I_323}
	I_{323} \leq C_{\O}\left(\sqrt{\delta}+\frac{1}{\sqrt{N}}\right) \sup_{0\leq s \leq t} \norm{h(s)}_{L^\infty}^3 + C_{N,\delta,\O} A_0.
\end{align}
In summary, we can obtain from \eqref{L_inf_split.9(n)}, \eqref{I_32.1}, \eqref{L_inf_split.10(n)}, \eqref{I_321},\eqref{I_322} and \eqref{I_323} that 
\begin{align} \label{I_32}
	I_{32} \leq C e^{-\frac{\nu_0}{2}t}\norm{h_0}_{L^\infty} \int_0^t \norm{h(s)}_{L^\infty} ds +
	C_{\O}\left(\sqrt{\delta}+\frac{1}{\sqrt{N}}\right) \sup_{0\leq s \leq t} \norm{h(s)}_{L^\infty}^3 + C_{N,\delta,\O} A_0. 
\end{align}
Finally, we split the velocity $u'$-integration domain into $\{\abs{u'}\geq 3N\}\cup \{\abs{u'}\leq 3N\}$, and then it follows from \eqref{k_prop.4(n)} that 
\begin{align} \label{I_33.1}
	I_{33}\textbf{1}_{\abs{u'}\geq 3N} \leq \frac{C}{N}\sup_{0\leq s \leq t}\norm{h(s)}_{L^\infty}^3. 
\end{align}
Plugging the estimate \eqref{L_inf_esti.1} into $I_{33}\textbf{1}_{\abs{u'}\leq 3N}$, we get 
{\footnotesize\begin{align} \label{L_inf_split.11(n)}
	&I_{33}\textbf{1}_{\abs{u'}\leq 3N} \notag\\
	&\leq C \int_0^t e^{-\frac{\nu_0}{2} (t-s)} \norm{h(s)}_{L^\infty}\notag\\ 
	&\quad \quad \times \left(\int_{\abs{u}\leq 2N} (1+\abs{u})^{-4\beta+4} \left[\int_0^s e^{-\frac{\nu_0}{2}s}\norm{h_0}_{L^\infty}^2
	\int_{\abs{u'}\leq 3N} \abs{\tilde{k}_2(V'_{cl}(s'),u')} du'ds'\right]^2du\right)^{\frac{1}{2}}ds\notag\\
	&\quad +C\int_0^t e^{-\frac{\nu_0}{2} (t-s)} \norm{h(s)}_{L^\infty} \left(\int_{\abs{u}\leq 2N}
	(1+\abs{u})^{-4\beta+4} \right.\notag\\
	&\left. \quad \quad \times\left[ \int_0^s e^{-\frac{\nu_0}{2}(s-s')}\int_{\abs{u'}\leq 3N} \abs{\tilde{k}_2(V'_{cl}(s'),u')} \left \{\int_0^{s'}e^{-\frac{\nu_0}{2}(s'-s'')}
	\int_{\R^3} \abs{k_w(V'_{cl}(s''),u'')h(s'',X'_{cl}(s''),u'')} \right \}^2 \right]^2 \right)^{\frac{1}{2}}\notag\\
	&\quad + C \int_0^t e^{-\frac{\nu_0}{2} (t-s)} \norm{h(s)}_{L^\infty} \left(\int_{\abs{u}\leq 2N}
	(1+\abs{u})^{-4\beta+4}\left[ \int_0^s e^{-\frac{\nu_0}{2}(s-s')}\right. \right.\notag\\
	&\left. \left. 
	\quad \quad \times \int_{\abs{u'}\leq 3N} \abs{\tilde{k}_2(V'_{cl}(s'),u')} \left \{\int_0^{s'}e^{-\frac{\nu_0}{2}(s'-s'')}
	\norm{h(s'')}_{L^\infty} \left(\int_{\R^3} (1+\abs{u''})^{-4\beta+4}\abs{h^2(s'',X'_{cl}(s''),u'')} \right)^{\frac{1}{2}}
	 \right \}^2\right]^2\right)^{\frac{1}{2}}\notag\\
	&:=I_{331}+I_{332}+I_{333}. 
\end{align}}
By \eqref{k_prop.5(n)}, we have 
\begin{align} \label{I_331}
	I_{331} \leq Ce^{-\frac{\nu_0}{4}t}\norm{h_0}_{L^\infty}^2  \int_0^t \norm{h(s)}_{L^\infty} ds . 
\end{align}
As before, we divide $I_{332}$ into two cases $\{\abs{u''}\geq 4N\}\cup \{\abs{u''} \leq 4N\}$. The first case can be further bounded by
\begin{align} \label{I_332.1}
	I_{332} \textbf{1}_{\abs{u''} \geq 4N} \leq \frac{C}{N}\sup_{0\leq s \leq t} \norm{h(s)}_{L^\infty}^3, 
\end{align}
where we have used \eqref{k_prop.1(n)},\eqref{k_prop.5(n)} and $\beta \geq 5/2$. Similar to the arguments \eqref{Holder} and \eqref{I_322.2}, we apply the H\"{o}lder's inequality and then it holds that 
\begin{align} \label{I_332.2}
	&I_{332}\textbf{1}_{\abs{u''} \leq 4N} \notag\\
	&\leq C\int_0^t e^{-\frac{\nu_0}{2} (t-s)} \norm{h(s)}_{L^\infty} \left(\int_{\abs{u}\leq 2N}
	(1+\abs{u})^{-4\beta+4} \right.\notag\\
	&\left. \quad \times\left[ \int_0^s e^{-\frac{\nu_0}{2}(s-s')}\int_{\abs{u'}\leq 3N} \abs{\tilde{k}_2(V'_{cl}(s'),u')} \int_0^{s'}e^{-\nu_0(s'-s'')}
	\int_{\abs{u''}\leq 4N} \abs{k_w(V'_{cl}(s''),u'')h^2(s'',X'_{cl}(s''),u'')}  \right]^2 \right)^{\frac{1}{2}}\notag\\
	&\leq C\sup_{0\leq s \leq t} \norm{h(s)}_{L^\infty}^2 \Big(\int_0^t \int_0^s \int_0^{s'} e^{-\nu_0(t-s'')}\int _{\abs{u}\leq 2N,\abs{u'}\leq 3N,\abs{u''}\leq 4N}
	(1+\abs{u})^{-4\beta+4} \notag\\
&\qquad\qquad\qquad\qquad\qquad\qquad\qquad\qquad\qquad\qquad\abs{\tilde{k}_2(V'_{cl}(s'),u')k_w(V'_{cl}(s''),u'')h^2(s'',X'_{cl}(s''),u'')}\Big)^{\frac{1}{2}}.
\end{align}	
Similar to the way of treating $I_{233}$ before, it follows that 
\begin{align} \label{I_332.3}
	I_{332}\textbf{1}_{\abs{u''}\leq 3N} \leq C_{\O}\left(\sqrt{\delta}+\frac{1}{\sqrt{N}}\right) \sup_{0\leq s \leq t} \norm{h(s)}_{L^\infty}^3 + C_{N,\delta,\O}A_0.
\end{align}
Applying the above estimates \eqref{I_332.1},\eqref{I_332.2} and \eqref{I_332.3}, one then has   
\begin{align} \label{I_332}
	I_{332} \leq C_{\O} \left(\sqrt{\delta}+\frac{1}{\sqrt{N}}\right) \sup_{0\leq s \leq t}\norm{h(s)}_{L^\infty}^3+
	C_{N,\delta,\O} A_0.
\end{align}
Likewise, for $I_{333}$, we split the velocity $u''$-integration domain into $\{\abs{u''}\geq 4N\} \cup \{\abs{u''}\leq 4N\}$. We have the estimate for $I_{333}$ in $\{\abs{u''}\geq 4N\}$
\begin{align} \label{I_333.1}
	I_{333}\textbf{1}_{\abs{u''}\geq 4N} \leq \frac{C}{N}\sup_{0\leq s \leq t}
	\norm{h(s)}_{L^\infty}^5, 
\end{align}
due to $\beta \geq 5/2$ and \eqref{k_prop.5(n)}. Similar to the way \eqref{I_322.2}, it holds that
{\footnotesize\begin{align} \label{I_333.2}
	&I_{333}\textbf{1}_{\abs{u''}\leq 4N} \nonumber \\ 
	&\leq C \sup_{0\leq s \leq t} \norm{h(s)}_{L^\infty}^3  \int_0^t e^{-\frac{\nu_0}{2} (t-s)} \nonumber \\
	&\quad \times \left(\int_{\abs{u}\leq 2N}
	(1+\abs{u})^{-4\beta+4}\left[ \int_0^s \int_0^{s'}e^{-\frac{\nu_0}{2}(s-s'')} \int_{\abs{u'}\leq 3N} \abs{\tilde{k}_2(V'_{cl}(s'),u')}	 
	\int_{\abs{u''}\leq 4N} (1+\abs{u''})^{-4\beta+4}\abs{h^2(s'',X'_{cl}(s''),u'')} 
	 \right]^2\right)^{\frac{1}{2}} \nonumber  \\ 
	&\leq C \sup_{0\leq s \leq t} \norm{h(s)}_{L^\infty}^3  \int_0^t e^{-\frac{\nu_0}{2} (t-s)} \nonumber \\ 
	& \quad  \times \left( \int_{\abs{u} \leq 2N} (1+\abs{u})^{-4\beta+4} \left[  \int_0^s\int_0^{s'} e^{-\frac{\nu_0}{2}(s-s'')} 
	\left(\int_{\abs{u'}\leq 3N,\abs{u''}\leq 4N} \abs{\tilde{k}_2(V'_{cl}(s'),u')}(1+\abs{u''})^{-4\beta+4}\right)^{\frac{1}{2}} \right. \right. \nonumber \\
	&\left. \left. \hspace{5cm} \times \left(\int_{\abs{u'}\leq 3N,\abs{u''}\leq 4N} 
	\abs{\tilde{k}_2(V'_{cl}(s'),u')}(1+\abs{u''})^{-4\beta+4}\abs{h^4(s'',X_{cl}'(s''),u'')}\right)^{\frac{1}{2}}
	 \right]^2  \right)^{\frac{1}{2}}  \nonumber \\
	&\leq C\sup_{0\leq s \leq t} \norm{h(s)}_{L^\infty}^3  \int_0^t e^{-\frac{\nu_0}{2} (t-s)} \nonumber \\
	&\quad \times \left(\int_{\abs{u}\leq 2N} (1+\abs{u})^{-4\beta+4} 
	\int_0^s \int_0^{s'} e^{-\nu_0(s-s'')}\int_{\abs{u'}\leq 3N,\abs{u''}\leq 4N} \abs{\tilde{k}_2(V'_{cl}(s'),u')} (1+\abs{u''})^{-4\beta+4} \abs{h^4(s'',X_{cl}'(s''),u'')} \right)^{\frac{1}{2}} \nonumber \\ 
	&\leq C \sup_{0\leq s \leq t} \norm{h(s)}_{L^\infty}^4 \nonumber \\
	&\quad  \times \left( \underbrace{\int_0^t\int_0^s \int_0^{s'} e^{-\nu_0(t-s'')}
	\int_{\abs{u}\leq 2N,\abs{u'}\leq 3N,\abs{u''}\leq 4N} (1+\abs{u})^{-4\beta+4} \abs{\tilde{k}_2(V'_{cl}(s'),u')}
	(1+\abs{u''})^{-4\beta+4} \abs{h^2(s'',X'_{cl}(s''),u'')}}_{(\ref{I_333.2})_{MAIN}}\right)^{\frac{1}{2}},
\end{align}}
by using \eqref{k_prop.5(n)} and $\beta \geq 5/2$. For applying the Proposition \ref{prop_COV}, the integration $(\ref{I_333.2})_{MAIN}$ is rewritten by 
\begin{align} \label{L_inf_split.12(n)}
	&(\ref{I_333.2})_{MAIN} \notag\\
	&\leq  \sum_{n=1}^{[t / \delta]} \sum_{\abs{\vec{n}} \leq N} \sum_{m=1}^{[ t/\delta]} \sum_{\abs{\vec{m}} \leq N} 
		\sum_{k}^{K_{i_s}} \sum_{k'}^{K'_{j_{s,s'}}}\int_{(n-1)\delta}^{(n+1)\delta} \int_{t^{k+1}+\delta} ^{t^k-\delta} \int_{t^{k'+1}+\delta}^{t^{k'}-\delta} e^{-{\nu_0}(t-s'')} \notag\\ 
		&\quad \times \int_{\abs{u}\leq 2N, \abs{u'}\leq 3N, \abs{u''}\leq 4N}  (1+\abs{u})^{-4\beta+4}\abs{\tilde{k}_2(V'_{cl}(s'),u')}(1+\abs{u''})^{-4\beta+4}\abs{h^2(s'',X'_{cl}(s''),u'')}  \textbf{1}_{R_1^{c} \cap \cdots \cap R_6^{c}}\notag\\
		& \qquad+ \mathcal{B}_{2} + \mathcal{R}_{2},
\end{align}
where the $\mathcal{B}_{2}$ term corresponds to where the trajectory is near bouncing points and $\mathcal{R}_{2}$ corresponds to where $(u,s',u',s'')$ is in one of $R_1$ through $R_6$. So we have the following small estimates for $\mathcal{B}_{2}$ and $\mathcal{R}_{2}$: 
{\footnotesize\begin{align} \label{B,R_2}
		\mathcal{B}_{2} &\leq \int_0^t \int_0^s \int_0^{s'}e^{-\nu_0(t-s'')} \int_{\abs{u}\leq 2N}  \int_{\abs{u'} \leq 3N} 
		(1+\abs{u})^{-4\beta+4}\abs{\tilde{k}_2(V'_{cl}(s'),u')} \notag\\ 
		& \quad\times \int_{\abs{u''}\leq 4N} (1+\abs{u''})^{-4\beta+4}\abs{h^2(s'',X'_{cl}(s''),u'')} 
		\textbf{1}_{\{t^{k+1} \leq s' \leq t^{k+1}+\delta\} \cup \{t^k-\delta \leq s' \leq t^k\}} 
		\textbf{1}_{\{t^{k'+1} \leq s'' \leq t^{k'+1}+\delta\} \cup \{t^{k'}-\delta \leq s'' \leq t^{k'} \}} \notag\\
		&\leq C\delta \sup_{0 \leq s \leq t} \norm{h(s)}_{L^\infty}^2, \notag\\ 
		\mathcal{R}_{2} &\leq  \int_0^t \int_0^s\int_0^{s'} e^{-\nu_0(t-s'')} \int_{\abs{u}\leq 2N}  \int_{\abs{u'}\leq 3N} (1+\abs{u})^{-4\beta+4}\abs{\tilde{k}_2(V'_{cl}(s'),u')} \notag\\
		 & \quad \times \int_{\abs{u''} \leq 4N} (1+\abs{u''})^{-4\beta+4}\abs{h^2(s'',X'_{cl}(s''),u'')} \textbf{1}_{R_1 \cup R_2 \cup \cdots \cup R_6}\notag\\ 
		&\leq C \delta \sup_{0\leq s \leq t} \norm{h(s)}_{L^\infty}^2. 
\end{align}}
Using H\"{o}lder's inequality and \eqref{k_prop.6(n)} yields that the remaining part in \eqref{L_inf_split.12(n)} can be further bounded by
{\footnotesize\begin{align}  \label{I_333.3}
		&\sum_{n=1}^{[t / \delta]} \sum_{\abs{\vec{n}} \leq N} \sum_{m=1}^{[ t/\delta]} \sum_{\abs{\vec{m}} \leq N} 
		\sum_{k}^{K_{i_s}} \sum_{k'}^{K'_{j_{s,s'}}}\int_{(n-1)\delta}^{(n+1)\delta} \int_{t^{k+1}+\delta} ^{t^k-\delta} \int_{t^{k'+1}+\delta}^{t^{k'}-\delta} e^{-{\nu_0}(t-s'')}\notag \\&
		\quad \times \int_{\abs{u}\leq 2N, \abs{u'}\leq 3N, \abs{u''}\leq 4N}  (1+\abs{u})^{-4\beta+4}\abs{\tilde{k}_2(V'_{cl}(s'),u')}(1+\abs{u''})^{-4\beta+4}\abs{h^2(s'',X'_{cl}(s''),u'')}  \textbf{1}_{R_1^{c} \cap \cdots \cap R_6^{c}}\notag\\
		&\leq \sum_{n=1}^{[t / \delta]} \sum_{\abs{\vec{n}} \leq N} \sum_{m=1}^{[ t/\delta]} \sum_{\abs{\vec{m}} \leq N} 
		\sum_{k}^{K_{i_s}} \sum_{k'}^{K'_{j_{s,s'}}}\int_{(n-1)\delta}^{(n+1)\delta} \int_{t^{k+1}+\delta} ^{t^k-\delta} \int_{t^{k'+1}+\delta}^{t^{k'}-\delta} e^{-{\nu_0}(t-s'')} \notag\\
		&\quad \times \left(\int_{\abs{u}\leq 2N,\abs{u'}\leq 3N, \abs{u''}\leq 4N} (1+\abs{u})^{-8\beta+8}
		\abs{\tilde{k}_2(V'_{cl}(s'),u')}^2(1+\abs{u''})^{-8\beta+8} \right)^{1/2} \notag\\ 
		&\quad \times \left(\int_{\abs{u}\leq 2N,\abs{u'}\leq 3N, \abs{u''}\leq 4N} \abs{h^4(s'',X'_{cl}(s''),u'')}\textbf{1}_{R_1^{c} \cap \cdots \cap R_6^{c}}\right)^{1/2}\notag\\
		&\leq C \sum_{n=1}^{[t / \delta]} \sum_{\abs{\vec{n}} \leq N} \sum_{m=1}^{[ t/\delta]} \sum_{\abs{\vec{m}} \leq N} 
		\sum_{k}^{K_{i_s}} \sum_{k'}^{K'_{j_{s,s'}}}\int_{(n-1)\delta}^{(n+1)\delta} \int_{t^{k+1}+\delta} ^{t^k-\delta} \int_{t^{k'+1}+\delta}^{t^{k'}-\delta} e^{-{\nu_0}(t-s'')} \notag\\
		&\times \left( \int_{u'' \in \R^3} \int_{\hat{u}_2,\xi_3,\xi_4} \textbf{1}_{\{ \abs{u}\leq 2N, \abs{u'}\leq 3N, \abs{u''}\leq 4N \}} \; d\hat{u}_2d\xi_3 d\xi_4
	  \right)^{\frac{1}{2}}
	  \left ( \int_{\abs{u'},\xi_1,\xi_2} \abs{h(s'',X'_{cl}(s''),u'')}^4  \; d\abs{u'}d\xi_1d\xi_2
	   \right) ^{\frac{1}{2}} ds''ds'ds \notag\\ 
	   &\leq C_{N,\delta,\O} \sum_{n=1}^{[t/\delta]} \sum_{\abs{\vec{n}} \leq N} \sum_{m=1}^{[t/\delta]} \sum_k^{K_{i_s}} \sum_{k'}^{K'_{j_{s,s'}}}
	    \int_{(n-1)\delta}^{(n+1)\delta}\int_{t^{k+1}+\delta}^{t^k-\delta} \int_{t^{k'+1}+\delta} ^{t^{k'}-\delta}  e^{-\nu_0(t-s'')}\notag\\
	    &\times \left(\int_{\abs{u''}\leq 4N } \int_{B(X(s''),r_{\delta,n,\vec{n},i,k,m,\vec{m},j,k'})} \abs{h(s'',y,u'')}^4 
	    \frac{1}{\epsilon'_{\O,N,\delta}}  \; dydu'' \right)^{\frac{1}{2}}\notag\\
	    &\leq C_{N,\delta,\O}\sup_{0\leq s \leq t} \norm{h(s)}_{L^\infty} \int_0^t \int_0^s\int_0^{s'} e^{-\nu_0(t-s'')} \left( \int_{y \in \O} \int_{\abs{u''}\leq 4N} \abs{h(s'',y,u'')}^2 \; du'' dy\right)^{\frac{1}{2}} \; ds''ds'ds.
\end{align}}
We can obtain that
\begin{align} \label{I_333.4}
	I_{333} \textbf{1}_{\abs{u''}\leq 4N} \leq C_{\O}\left(\sqrt{\delta}+\frac{1}{\sqrt{N}}\right)\sup_{0\leq s \leq t}\norm{h(s)}_{L^\infty}^5+C_{N,\delta,\O} A_0,
\end{align} 
due to \eqref{entropy.1(n)}, \eqref{I_333.2}, \eqref{L_inf_split.12(n)}, \eqref{B,R_2} and \eqref{I_333.3}. Then, it follows from \eqref{I_333.1}, \eqref{I_333.2} and \eqref{I_333.4} that 
\begin{align} \label{I_333} 
	I_{333} \leq C_{\O}\left(\sqrt{\delta}+\frac{1}{\sqrt{N}}\right) \sup_{0\leq s \leq t} \norm{h(s)}_{L^\infty}^5 + C_{N,\delta,\O} A_0.
\end{align}
Combining \eqref{I_33.1}, \eqref{L_inf_split.11(n)}, \eqref{I_331}, \eqref{I_332} and \eqref{I_333}, one obtains that 
\begin{align} \label{I_33} 
	I_{33} \leq Ce^{-\frac{\nu_0}{4}t}\norm{h_0}_{L^\infty}^2  \int_0^t \norm{h(s)}_{L^\infty} ds+C_{\O}\left(\sqrt{\delta}+\frac{1}{\sqrt{N}}\right) \sup_{0\leq s \leq t} \norm{h(s)}_{L^\infty}^5+
	C_{N,\delta,\O} A_0.
\end{align}
It follows from \eqref{L_inf_esti.2(n)}, \eqref{L_inf_split.8(n)}, \eqref{I_31}, \eqref{I_32} and \eqref{I_33} that \eqref{L_inf_split.7(n)} is bounded as 
\begin{align} \label{L_inf_esti.3(n)}
	\begin{split}
	I_3
	\leq Ce^{-\frac{\nu_0}{4}t}\int_0^t \norm{h(s)}_{L^\infty} ds \left[ \norm{h_0}_{L^\infty}+\norm{h_0}_{L^\infty}^2 \right]+C_{\O}\left(\sqrt{\delta}+\frac{1}{\sqrt{N}}\right) \sup_{0\leq s \leq t} \norm{h(s)}_{L^\infty}^5+C_{N,\delta,\O} A_0.
	\end{split}
\end{align}
Consequently, plugging \eqref{I_2,I_3}, \eqref{L_inf_esti.1(n)}, \eqref{L_inf_esti.3(n)} into \eqref{L_inf_esti.1} establishes the following estimate
\begin{align} \label{L_inf_esti.4(n)}
	\abs{h(t,x,v)} &\leq Ce^{-\frac{\nu_0}{4}t} \left(\int_0^t \norm{h(s)}_{L^\infty} +1 \right)\left[\norm{h_0}_{L^\infty}+\norm{h_0}_{L^\infty}^2 \right]\notag\\
&\quad	 +C_{\O}\left(\delta+\frac{1}{\sqrt{N}}\right)\sup_{0\leq s \leq t }\norm{h(s)}_{L^\infty}^5 + C_{N,\delta,\O}A_0. 
\end{align}
The desired estimate \eqref{L_inf_esti.result} follows by taking the $L^\infty_{x,v}$ norm to the above estimate \eqref{L_inf_esti.4(n)}. We complete the proof of Lemma \ref{L_inf_esti}.
\end{proof}

\section{Proof of the global existence}\label{sec4}

Before giving the proof of Theorem \ref{main thm}, we recall the following known important result in \cite{KL} for the global existence of solutions of small amplitude in $L^\infty$ setting.

\begin{proposition} \label{LeeKim}
Let $w(v):= (1+\rho^2 \abs{v}^2)^\beta e^{\varpi\abs{v}^2}$ for $\rho>0$, $\beta> 5/2$, and $0\leq\varpi \leq 1/64$. Assume that $\O$ is a $C^3$ uniformly convex bounded domain and $F_0(x,v)=\mu(v)+\sqrt{\mu(v)}f_0(x,v)\geq 0$ satisfies the mass-energy conservations
\begin{align*}
	\int_{\O}\int_{\R^3} \sqrt{\mu(v)}f_0(x,v) \,dvdx = 0, \quad \int_{\O}\int_{\R^3} \abs{v}^2 \sqrt{\mu(v)}f_0(x,v)\, dvdx=0,
\end{align*}
as well as the angular-momentum conservation \eqref{angular conserv} in addition if $\O$ has rotational symmetry \eqref{rot sym}. 
Then, there exists $\delta_1>0$ such that if $\norm{wf_0}_{L^\infty} \leq \delta_1$, there exists a unique global solution $F(t,x,v)=\mu(v)+\sqrt{\mu(v)}f(t,x,v)$ to the Boltzmann equation \eqref{def.be} with the specular boundary condition \eqref{specular} such that
\begin{align*}
	\norm{wf(t)}_{L^\infty}\leq C\norm{wf_0}_{L^\infty} e^{-\lambda t},
\end{align*}	
for all $t\geq 0$, where $C>0$ is a generic constant and $\lambda=\lambda(\O)>0$. Moreover, if $f_0(x,v)$ is continuous except on $\gamma_0$ and satisfies initial-boundary compatibility condition,
\[
	f_0(x,v) = f_0(x, R_x v),\quad \forall x\in\p\O,
\] 
then $f(t,x,v)$ is also continuous in $[0,\infty)\times \{\overline{\O}\times \R^{3}\backslash \gamma_0\}$. 
\end{proposition}



With the help of the above result, we are ready to apply Lemma \ref{L_inf_esti} to give the  

\medskip
\noindent\textbf{Proof of Theorem \ref{main thm}.} We first continue to focus on the {\it a priori} $L^\infty$ estimate of solutions starting from \eqref{L_inf_esti.result}. Recall that  
 \begin{align*}
 	 \sup_{0\leq t \leq T_0} \norm{h(t)}_{L^\infty} \leq \bar{M} \quad \textrm{and} \quad \mathcal{E}(F_0) \leq \epsilon_1=\epsilon_1(\bar{M},T_0).
 \end{align*}
Here, $\bar{M}$ and $T_0$ will be explicitly determined in terms of $M_0$ later. By Lemma \ref{L_inf_esti}, it holds that 
 \begin{align} \label{L_inf_h}
 	\norm{h(t)}_{L^\infty} \leq Ce^{-\frac{\nu_0}{4}t} \left(1+ \int_0^t \norm{h(s)}_{L^\infty} \right)[M_0+M_0^2] +D, 
 \end{align}
 for all $0\leq t \leq T_0$, where we have denoted the small part $D$ by
 \begin{align} \label{def_D}
 	D:= C_{\O}\left(\delta+\frac{1}{\sqrt{N}}\right)\bar{M}^5 + C_{N,\delta,\O} [\sqrt{\mathcal{E}(F_0)}+\mathcal{E}(F_0)].
 \end{align}
Define 
\begin{align*}
	G(t):=1+ \int_0^t \norm{h(s)}_{L^\infty} \,ds.
\end{align*}
By using $G(t)$ above, we can rewrite \eqref{L_inf_h} as  
\begin{align} \label{G_1}
	G'(t) \leq Ce^{-\frac{\nu_0}{4}t }(M_0+M_0^2)G(t)+D.
\end{align}
By multiplying an integrating factor, one then has 
\begin{align*}
	\left(G(t)\exp\left\{-\frac{\nu_0}{4}C[M_0+M_0^2](1-e^{-\frac{\nu_0}{4}t}) \right\}\right)' \leq D, 
\end{align*}
for all $0\leq t \leq T_0$. Taking the integration over $[0,t]$, we get 
\begin{align} \label{G_2}
	G(t)\leq (1+Dt) \exp\left\{ \frac{\nu_0}{4}C [M_0+M_0^2]\left(1-e^{-\frac{\nu_0}{4}t}\right)\right\} \leq (1+Dt) \exp\left\{ \frac{\nu_0}{4}C [M_0+M_0^2]\right\}.
\end{align}	
Defining 
\begin{equation}
\label{def.bM}
\bar{M} := 4C(M_0+M_0^2) \exp\left\{ \frac{4}{\nu_0}C [M_0+M_0^2]\right\},
\end{equation}
and inserting \eqref{G_2} into \eqref{G_1} yield that
\begin{align} \label{main_bound_1}
	\norm{h(t)}_{L^\infty} &\leq C\left(M_0+M_0^2\right) \exp\left \{\frac{\nu_0}{4}C[M_0+M_0^2]\right\}\left(1+Dt\right)e^{-\frac{\nu_0}{4}t}+D\notag \\
	&\leq \frac{1}{4}\bar{M}\left(1+Dt\right) e^{-\frac{\nu_0}{4}t}+D\notag \\
	&\leq \frac{1}{4}\bar{M}\left(1+\frac{8}{\nu_0}D\right)e^{-\frac{\nu_0}{8}t}+D, 
\end{align}
for all $0\leq t \leq T_0$. Recall the definition of $D$ in \eqref{def_D}. First  we choose sufficiently small $\delta>0$ and large $N$, both depending on $\bar{M}$. Since $\bar{M}$ has been written as a function of $M_0$ in \eqref{def.bM}, both $\delta$ and $N$ can be chosen to depend only on $M_{0}$. Then we take $\epsilon_1>0$ sufficiently small, depending only on $M_0$ similarly, so that 
\begin{align} \label{D}
	D\leq \min\left \{\frac{\nu_0}{16},\frac{\delta_1}{8},\frac{\bar{M}}{8} \right\},
\end{align}
where $\delta_1>0$ is introduced in Proposition \ref{LeeKim}. Hence, it follows from \eqref{main_bound_1} that 
\begin{align} \label{main_bound_2}
	\norm{h(t)}_{L^\infty} \leq \frac{3}{8}\bar{M}e^{-\frac{\nu_0}{8}t} +\frac{1}{8}\bar{M}\leq \frac{1}{2} \bar{M},
\end{align}
for $0\leq t \leq T_0$. Therefore, we have shown that 
\begin{align}\label{ad.dp1}
	\sup_{0\leq t\leq T_0}\norm{h(t)}_{L^\infty} \leq \frac{1}{2}\bar{M},
\end{align}	
provided that the {\it a priori} assumption \eqref{a priori assumption} holds true and also $\mathcal{E}(F_0) \leq \epsilon_1$ is satisfied.

Next, our goal is to extend the local-in-time solution of the Boltzmann equation up to the time $T_0$ and check that 
\begin{equation*} 
	\norm{h(T_0)}_{L^\infty} < \delta_1,
\end{equation*}
which allows that the solution can be further extended from $[0,T_0]$ to $[0,\infty)$ by applying Proposition \ref{LeeKim}; see Figure \ref{fig1}. Indeed, in terms of Lemma \ref{local},  there exists a time $\hat{t}_0>0$ such that the solution $f(t,x,v)$ of the Boltzmann equation exists for $0\leq t \leq \hat{t}_0$ and satisfies
\begin{align*}
	\sup_{0\leq t \leq \hat{t}_0} \norm{wf(t)}_{L^\infty} \leq 2 \norm{wf_0}_{L^\infty} \leq \frac{1}{2} \bar{M}.
\end{align*}
Considering $\hat{t}_0$ as the initial time, one can also extend the solution $f(t,x,v)$ for $0\leq t \leq \hat{t}_0+\tilde{t}$ for some $\tilde{t}>0$ and satisfies 
\begin{align*}
	\sup_{\hat{t}_0\leq t \leq \hat{t}_0+\tilde{t}} \norm{wf(t)}_{L^\infty} \leq 2 \norm{wf(\hat{t}_0)}_{L^\infty} \leq \bar{M}.
\end{align*}
Thus, the solution satisfies the {\it a priori} assumption \eqref{a priori assumption} for $0\leq t \leq \hat{t}_0+\tilde{t}$, which leads in terms of \eqref{ad.dp1} to 
\begin{align*}
	\sup_{0\leq t \leq \hat{t}_0+\tilde{t}} \norm{h(t)}_{L^\infty} \leq \frac{1}{2} \bar{M}. 
\end{align*}
Define 
\begin{align*}
	T_0:= \frac{8}{\nu_0}\left[ \ln \bar{M}+ \left \vert \ln \delta_1 \right \vert \right ],
\end{align*}
and then one can obtain the existence of the solution for $0\leq t \leq T_0$ via this process repeatedly. Moreover, it follows from \eqref{main_bound_1} and \eqref{D} that 
\begin{align*}
	\norm{wf(T_0)}_{L^\infty} \leq \frac{3}{8} \bar{M} e^{-\frac{\nu_0}{8}T_0} + \frac{\delta_1}{8}  < \delta_1.
\end{align*}
Therefore, we prove the global-in-time existence and uniqueness of solutions further with the help of Proposition \ref{LeeKim}. It remains to show an exponential decay in time of $\norm{wf(t)}_{L^\infty}$. For $t\geq T_0$, applying Proposition \ref{LeeKim} yields that 
\begin{align} \label{main_bound_3}
	\norm{h(t)}_{L^\infty} \leq C \norm{h(T_0)}_{L^\infty}e^{-\lambda(t-T_0)}\leq C\delta_1 e^{-\lambda(t-T_0)}.
\end{align}	
Taking $\vartheta:= \min\{\frac{\nu_0}{8}, \lambda\}$, it holds that 
\begin{align*}
	\norm{h(t)}_{L^\infty} \leq C \bar{M} e^{-\vartheta t}\leq C(M_0+M_0^2)\exp \left\{\frac{4}{\nu_0}C[M_0+M_0^2]\right\}e^{-\vartheta t},
\end{align*}
for all $t\geq 0$, where we have used \eqref{main_bound_2} and \eqref{main_bound_3}. Then, \eqref{thm.td} follows. If $f_0(x,v)$ is continuous except on $\gamma_0$ and satisfies the specular boundary condition, the continuity of $f(t,x,v)$ on $[0,\infty) \times \{ \O \times \R^3\backslash \gamma_0\}$ comes from Lemma \ref{local}. Therefore, we complete the proof of Theorem \ref{main thm}.\qed

\section{Appendix: Local-in-time existence}\label{sec5}

In this section, for the initial-boundary value problem \eqref{def.be}, \eqref{id} and \eqref{specular}, we consider the proof of existence of local-in-time solutions in $L^\infty$ as long as $\|wf_0\|_{L^\infty}$ is finite. From the lemma below, we notice that $\rho>0$ can be generally chosen independent of $\|wf_0\|_{L^\infty}$, the parameter $\varpi$ in the exponential part of the weight function can be zero, and the constant coefficient on the right-hand side of the estimate \eqref{lem.lx} does not involve the parameter $\rho$.

\begin{lemma} \label{local}
Let $\rho>0$, $\beta >5/2$ and $0\leq \varpi \leq \frac{1}{64}$ be fixed in the weight function \eqref{weight}. Under the conditions that $F_0(x,v)= \mu(v) + \sqrt{\mu(v)} f_0(x,v)\geq 0$ and $\norm{wf_0}_{L^\infty} <\infty$, there exists a time $\hat{t}_0>0$ such that the initial-boundary value problem \eqref{def.be} and \eqref{id} with specular reflection boundary condition \eqref{specular} admits a unique solution $F(t,x,v)=\mu(v)+\sqrt{\mu(v)} f(t,x,v)\geq 0$ for $t \in [0,\hat{t}_0]$, satisfying
\begin{align}\label{lem.lx}
	\sup_{0\leq t \leq \hat{t}_0} \norm{wf(t)}_{L^\infty} \leq 2 \norm{wf_0}_{L^\infty}. 
\end{align}
Moreover, if $f_0(x,v)$ is continuous except on $\gamma_0$ and satisfies initial-boundary compatibility condition,
\[
	f_0(x,v) = f_0(x, R_x v),\quad \forall x\in\p\O,
\] 
then $f(t,x,v)$ is also continuous in $[0,\hat{t}_0]\times \{\overline{\O}\times \R^{3}\backslash \gamma_0\}$.  
\end{lemma}
\begin{proof}
For the local existence of solutions to the Boltzmann equation \eqref{def.be}, we use the following iteration 
\begin{align} \label{mild_f.1}
	\begin{split}
		\begin{cases}
		\displaystyle
			\partial_t F^{n+1} +v \cdot \nabla_x F^{n+1} +F^{n+1} \int_{\R^3} \int_{\S^2} B(v-u,\omega) F^n(t,x,u)\, d\omega du =Q_+(F^n,F^n), \\
			F^{n+1}(t,x,v)\Bigr\rvert_{t=0} = F_0(x,v) \geq 0,
		\end{cases} 
	\end{split}
\end{align}
with the specular reflection boundary condition $F^{n+1}(t,x,v)=F^{n+1}(t,x,v-2(n(x)\cdot v)n(x))$ for $x \in \partial\O$ and $F^0(t,x,v) \equiv \mu(v)$. Recall that 
\begin{align*}
	R(f)(t,x,v) := \int_{\R^3} \int_{\S^2} B(v-u,\omega)\left[\mu(u) +\sqrt{\mu(u)}f(t,x,u) \right]\,d\omega du.
\end{align*}
If we set 
\begin{align*} 
	\displaystyle
	f^{n+1}(t,x,v) =\frac{F^{n+1}(t,x,v) - \mu(v)}{\sqrt{\mu(v)}} \quad \textrm{and} \quad h^{n+1}(t,x,v) =w(v) f^{n+1}(t,x,v), 
\end{align*}	
we can rewrite the above iteration as, for $n=0,1,2,\cdots$,
\begin{align} \label{mild_f.2}
	\begin{split}
		\begin{cases}
			\displaystyle
			\partial_t h^{n+1} + v\cdot \nabla_x h^{n+1} + h^{n+1} R(f^n) = K_w h^n + w(v) \Gamma_+(f^n,f^n), \\
			h^{n+1}(t,x,v) \Bigr\rvert_{t=0} = h_0(x,v), 
		\end{cases}
	\end{split}
\end{align} 
with $h^{n+1}(t,x,v)=h^{n+1}(t,x,v-2(n(x)\cdot v)n(x))$ for $x \in \partial\O$ and $h^0(t,x,v) \equiv 0$. By induction on $n$, we can prove that there exists a positive time $\hat{t}_1>0$ such that \eqref{mild_f.1} or \eqref{mild_f.2} has a unique solution over $[0,\hat{t}_1]$, and the following estimate for $h^{n+1}$ and the positivity of $F^{n+1}$ hold:
\begin{align} \label{h_0}
	\norm{h^n(t)}_{L^\infty} \leq 2 \norm{h_0}_{L^\infty},\quad F^n(t,x,v)\geq 0, 
\end{align}
for $0\leq t \leq \hat{t}_1$.  

When we solve \eqref{mild_f.2} with $n=0$, we can directly have 
	$h^{1} (t,x,v) =e^{-\nu(v)t} h_0(X_{cl}(0),V_{cl}(0))$, 
 where we have used $f^0(t,x,v)\equiv 0$. Furthermore, we have 
 	$\norm{h^1(t)}_{L^\infty} \leq 2 \norm{h_0}_{L^\infty}$ 
 for $t \geq 0$. Similarly, for the positivity of $F^1$, we solve \eqref{mild_f.1} to get 
 \begin{align*}
 	F^1(t,x,v) = e^{-\nu(v)t}F_0(X_{cl}(0),V_{cl}(0)) +\int_0^t e^{-\nu(v)(t-s)} \nu(v) \mu(v) ds\geq 0. 
 \end{align*}
 Now, we suppose that \eqref{h_0} holds for $n=0,1,2,\cdots,k$. We consider the same iteration \eqref{mild_f.2} with specular boundary condition for $n=k+1$. We define the linear solution operator $I^{n}(t,s)$ for \eqref{mild_f.2} by 
 \begin{align*}
 	I^n(t,s) &:=  \exp\left\{ -\int_s^t R(f^n)(\tau,X_{cl}(\tau),V_{cl}(\tau) d\tau\right\}\\
	&\,=\exp\left\{ -\int_s^t \int_{\R^3}\int_{\S^2} B(V_{cl}(\tau)-u,\omega)F^n(\tau,X_{cl}(\tau),u)\, d\omega du d\tau \right\}.
 \end{align*}
 Applying the Duhamel principle, one has 
 \begin{align} \label{Duhamel_1}
 	h^{k+1} (t,x,v) = I^{k}(t,0) h_0(X_{cl}(0),V_{cl}(0)) + \int_0^t I^k(t,s)\left[K_wh^k(s) + w(v) \Gamma_+(f^k,f^k)(s) \right]\,ds.
 \end{align}
From Lemma 2.2 and Lemma 2.3, we further bound \eqref{Duhamel_1} as: 
\begin{align*}
	\abs{h^{k+1}(t,x,v)} &\leq \norm{h_0}_{L^\infty} + C \int_0^t \norm{h^k(s)}_{L^\infty} + \norm{h^k(s)}_{L^\infty}^2 \, ds \\ 
	&\leq \norm{h_0}_{L^\infty} + 4Ct \norm{h_0}_{L^\infty}(1+\norm{h_0}_{L^\infty}), 
\end{align*}
where $C$ is some positive constant. Taking $\hat{t}_1:=(4C[1+\norm{h_0}_{L^\infty}])^{-1}$, then we induce 
\begin{align}\label{lem.lx.pa1}
	\norm{h^{k+1}(t)}_{L^\infty} \leq 2 \norm{h_0}_{L^\infty},
\end{align}
for $0\leq t \leq \hat{t}_1$. For the positivity of $F^{k+1}$, we express $F^{k+1}(t,x,v)$ as
\begin{align*}
		F^{k+1}(t,x,v) = I^{k}(t,0)F_0(X_{cl}(0),V_{cl}(0)) + \int_0^t I^k(t,s)Q_+(F^k,F^k)\, ds .
\end{align*} 
By $F^k\geq 0$, we directly deduce $F^{k+1} \geq 0$. 

We need to show that the sequence $\{h^n\}$ is convergent. To obtain the convergence, we introduce a new weight function $\hat{w}(v)= \frac{w(v)}{\sqrt{1+\rho^2 \abs{v}^2}}$. We denote 
\begin{align*}
\hat{h}^{n+1}(t,x,v) := \hat{w}(v) f^{n+1}(t,x,v)=(1+\rho^2 \abs{v}^2)^{-1/2} h^{n+1}(t,x,v), \quad n=0,1,2,\cdots.
\end{align*}
Then, $\hat{h}^{n+1}(t,x,v)$ satisfies 
\begin{align} \label{new h}
	\begin{split}
		\begin{cases}
	\partial_t \hat{h}^{n+1}+ v\cdot \nabla_x \hat{h}^{n+1} + \hat{h}^{n+1} R(f^n) = K_{\hat{w}}\hat{h}^{n} + \hat{w}(v) \Gamma_+(f^n,f^n),\\
	\hat{h}^{n+1}(t,x,v)\Bigr\rvert_{t=0} = \hat{h}_0(x,v). 
		\end{cases}
	\end{split}
\end{align}	
By the Duhamel principle to \eqref{new h}, it follows that 
\begin{align*}
	\hat{h}^{n+1} = I^n(t,0) \hat{h}_0(X_{cl}(0),V_{cl}(0)) + \int_0^t I_n(t,s) \left[ K_{\hat{w}} \hat{h}^n (s) + \hat{w}(v)\Gamma_+(f^n,f^n)(s)\right] \, ds.
\end{align*}
One then has 
\begin{align} \label{local_split.1}
	&\abs{(\hat{h}^{n+2} - \hat{h}^{n+1})(t,x,v)}\notag\\
	 &\leq \abs{I^{n+1}(t,0)-I^n(t,0)}.\abs{\hat{h}_0(X_{cl}(0),V_{cl}(0))}\notag\\ 
 	&\quad+\int_0^t \abs{I^{n+1}(t,s) -I^n(t,s)} \abs{K_{\hat{w}}\hat{h}^{n+1}(s)+\hat{w}(v) \Gamma_+(f^{n+1},f^{n+1})(s)}\,ds\notag\\
	&\quad+\int_0^t \abs{I^n(t,s)} \abs{K_{\hat{w}}\hat{h}^{n+1}(s)+\hat{w}(v) \Gamma_+(f^{n+1},f^{n+1})(s)-K_{\hat{w}}\hat{h}^{n}(s)-\hat{w}(v) \Gamma_+(f^n,f^n)(s)}\, ds\notag\\
	&:= K_1^n +K_2^n +K_3^n. 
\end{align}
For $K_1^n$ and $K_2^n$, we consider $\abs{I^{n+1}(t,s) -I^n(t,s)}$. By direct computations, we can derive 
\begin{align} \label{soln_operator}
	&\abs{I^{n+1}(t,s)-I^n(t,s)} \notag\\
	&=\left\vert \exp\left \{-\int_s^t R(f^{n+1})(\tau,X_{cl}(\tau),V_{cl}(\tau))\,d\tau\right \}-\exp \left\{-\int_s^t R(f^n)(\tau,X_{cl}(\tau),V_{cl}(\tau)) \,d\tau\right \}\right \vert \notag\\
	&\leq \left \vert \int_s^t \left[R(f^{n+1})-R(f^n)\right](\tau,X_{cl}(\tau),V_{cl}(\tau))\,d\tau\right \vert \notag\\
	&\leq C\int_s^t \int_{\R^3}\int_{\S^2} B(V_{cl}(\tau)-u,\omega) \frac{\sqrt{\mu(u)}}{\hat{w}(u)}\abs{(\hat{h}^{n+1}-\hat{h}^n)(\tau,X_{cl}(\tau),u)} \,d\omega dud\tau\notag\\
	&\leq C\nu(v) \int_s^t \norm{(\hat{h}^{n+1}-\hat{h}^n)(\tau)} \,d \tau. 
\end{align}
First consider the first term $K_1^n$ in \eqref{local_split.1}. It follows from \eqref{soln_operator} that 
\begin{align} \label{K_1^n}
		K_1^n &\leq C \nu(v) \abs{\hat{h}_0(X_{cl}(0),V_{cl}(0))}  \int_0^t \norm{(\hat{h}^{n+1}-\hat{h}^n)(\tau)} \, d\tau\notag\\
		&\leq C t\norm{h_0}_{L^\infty}  \sup_{0\leq \tau \leq t} \norm{(\hat{h}^{n+1}-\hat{h}^n)(\tau)}.
\end{align}
For $K_2^n$ in \eqref{local_split.1}, using Lemma \ref{Lemma_nega}, Lemma \ref{Gamma+} and \eqref{soln_operator}, one obtains that 
\begin{align} \label{K_2^n} 
	K_2^n &\leq C \int_0^t \norm{(\hat{h}^{n+1}-\hat{h}^n)(\tau)}_{L^\infty} \, d\tau \int_0^t \nu(v) \abs{K_{\hat{w}}\hat{h}^{n+1}(s)+\hat{w(v)} \Gamma_+(f^{n+1},f^{n+1})(s)} \,ds \notag \\ 
	&\leq C \int_0^t \left[\norm{h^{n+1}(s)}_{L^\infty} + \norm{h^{n+1}(s)}_{L^\infty}^2\right] \, ds \int_0^t \norm{(\hat{h}^{n+1}-\hat{h}^n)(\tau)}_{L^\infty} \, d\tau\notag\\
	&\leq Ct\norm{h_0}_{L^\infty}[4t(1+\norm{h_0}_{L^\infty})] \sup_{0\leq \tau \leq t} \norm{(\hat{h}^{n+1}-\hat{h}^n)(\tau)}_{L^\infty}\notag\\
	&\leq Ct\norm{h_0}_{L^\infty} \sup_{0\leq \tau \leq t} \norm{(\hat{h}^{n+1}-\hat{h}^n)(\tau)}_{L^\infty},
\end{align}
for $0\leq t \leq \hat{t}_1$. Finally, it suffices to establish the estimate for $K_3^n$ in \eqref{local_split.1}. Notice that 
\begin{align*} 
		&\abs{K_{\hat{w}}\hat{h}^{n+1}(s)+\hat{w}(v) \Gamma_+(f^{n+1},f^{n+1})(s)-K_{\hat{w}}\hat{h}^{n}(s)-\hat{w}(v) \Gamma_+(f^n,f^n)(s)} \\
		&\leq \abs{K_{\hat{w}}(\hat{h}^{n+1}-\hat{h}^n)(s)}\\
		 &\quad + \hat{w}(v)\left \vert \Gamma_+(f^{n+1}-f^{n}, f^{n+1})\right \vert + \hat{w}(v)\left \vert \Gamma_+(f^{n+1}, f^{n+1}-f^{n})\right \vert\\
		 &\leq C(1+\norm{h^{n}(s)}_{L^\infty}+\norm{h^{n+1}(s)}_{L^\infty})\norm{(\hat{h}^{n+1}-\hat{h}^n)(s)}_{L^\infty},
\end{align*}
where we have used Lemma \ref{Gamma+}.
Then, $K_3^n$ can be further bounded by 
\begin{align} \label{K_3^n}
		K_3^n &\leq C\int_0^t (1+ \norm{h^n(s)}_{L^\infty} +\norm{h^{n+1}(s)}) \norm{(\hat{h}^{n+1}-\hat{h}^n)(s)}_{L^\infty} \,ds\notag\\ 
		&\leq 4Ct (1+\norm{h_0}_{L^\infty}) \sup_{0\leq \tau \leq t} \norm{(\hat{h}^{n+1}-\hat{h}^n)(\tau)}_{L^\infty}. 
\end{align}
By substituting \eqref{K_1^n}, \eqref{K_2^n} and \eqref{K_3^n} to \eqref{local_split.1}, we can deduce
\begin{align*}
	\left \vert(\hat{h}^{n+2}-\hat{h}^{n+1})(t,x,v) \right \vert \leq 4Ct[1+\norm{h_0}_{L^\infty}]\sup_{0\leq \tau \leq t} \norm{(\hat{h}^{n+1} -\hat{h}^n)(\tau)}_{L^\infty}. 
\end{align*}
Hence, one obtains that
\begin{align*}
	\sup_{0\leq \tau \leq t}\norm{(\hat{h}^{n+2}-\hat{h}^{n+1})(\tau)}_{L^\infty} \leq 4Ct[1+\norm{h_0}_{L^\infty}] \sup_{0\leq \tau \leq t} \norm{(\hat{h}^{n+1}-\hat{h}^n)(\tau)}_{L^\infty}.
\end{align*}
If we take the time
\begin{align*}
	\hat{t}_0 := \frac{1}{8C[1+\norm{h_0}_{L^\infty}]} \quad (\leq \hat{t}_1),
\end{align*}
we have the contraction property for the sequence $\{\hat{h}^{n+1}\}$: 
\begin{multline*}
	\sup_{0\leq \tau \leq \hat{t}_0} \norm{(\hat{h}^{n+2}-\hat{h}^{n+1})(\tau)} _{L^\infty} \leq \frac{1}{2} \sup_{0\leq \tau \leq \hat{t}_0} \norm{(\hat{h}^{n+1}-\hat{h}^{n})(\tau)} _{L^\infty}\\
	\leq \cdots \leq \left(\frac{1}{2}\right)^n \sup_{0\leq \tau \leq \hat{t}_0} \norm{\hat{h}^1(\tau)}_{L^\infty}\leq \norm{h_0}_{L^\infty} \left(\frac{1}{2}\right)^{n-1}.
\end{multline*}
We thus prove that $\{f^{n+1} \}$ is a Cauchy sequence. In other words, there exists a function $f(t,x,v)$ satisfying the Boltzmann equation\eqref{f_eqtn} such that  
\begin{align} \label{seq_converge}
	\lim_{n\rightarrow \infty} \sup_{0\leq \tau \leq \hat{t}_0} \left \Vert \hat{w} (f^{n+1}-f)(\tau)\right \Vert_{L^\infty}=0. 
\end{align}
Moreover, the desired estimate \eqref{lem.lx} follows by \eqref{lem.lx.pa1} and also the non-negativity of the solution follows by $F^n(t,x,v)\geq 0$ for any $n$.
To prove the uniqueness of solutions, we suppose that there is another solution $g$ to the Boltzmann equation with the same initial and specular boundary conditions as $f$. Assume that 
\begin{align*}
	\sup_{0\leq \tau \leq \hat{t}_0} \norm{g(\tau)}_{L^\infty}<\infty.
\end{align*}
Then it holds that
\begin{align*}
	\abs{\hat{w}(v)(f-g)(t)} &\leq \int_0^t \abs{K_{\hat{w}}(\hat{w}f-\hat{w}g )(s)+\hat{w}(v)\Gamma_+(f,f)(s)-\hat{w}(v)\Gamma_+(g,g)(s)}\;ds  \\
	&\leq C\left(1+\sup_{0\leq \tau \leq \hat{t}_0}\left[\norm{wf(\tau)}_{L^\infty} + \norm{wg(\tau)}_{L^\infty}\right]\right)\int_0^t 
	\norm{\hat{w}(f-g)(\tau)}_{L^\infty} d\tau,
\end{align*}
which yields that
\begin{align} \label{gronwall}
	\norm{\hat{w}(f-g)(t)}_{L^\infty} \leq C \left(1+\sup_{0\leq \tau \leq \hat{t}_0}\left[\norm{wf(\tau)}_{L^\infty} + \norm{wg(\tau)}_{L^\infty}\right]\right)\int_0^t 
	\norm{\hat{w}(f-g)(\tau)}_{L^\infty} d\tau.
\end{align} 
Applying the Gronwall's inequality to \eqref{gronwall}, the uniqueness of solutions can be obtained. 

It remains to prove the continuity of $f(t,x,v)$ if $f_0$ is continuous except on $\gamma_0$ and satisfies the specular boundary condition. Because of $h^0\equiv 0$, we get the continuity of $h^1(t,x,v)$ on $[0,\hat{t}_0] \times \{\bar{\O}\times \R^3 \backslash \gamma_0\}$ from Lemma 21 in \cite{Guo}. Using Lemma 21 in \cite{Guo} and induction arguments, the continuity of $h^{n+1}(t,x,v)$ on $[0,\hat{t}_0] \times \{\bar{\O}\times \R^3 \backslash \gamma_0\}$ can be derived. By \eqref{seq_converge}, we know that $f^{n+1}(t,x,v)$ converges to $f(t,x,v)$ uniformly. Thus, we have the continuity of $f(t,x,v)$ on on $[0,\hat{t}_0] \times \{\bar{\O}\times \R^3 \backslash \gamma_0\}$ and then the proof of this lemma \ref{local} is completed.
 \end{proof}

\medskip

\noindent{\bf Acknowledgments.}  RJD is partially supported by the General Research Fund (Project No.~14302817) from RGC of Hong Kong and a Direct Grant from CUHK. DL and GK are supported by the National Research Foundation of Korea(NRF) grant funded by the Korea government(MSIT)(No. NRF-2019R1C1C1010915). DL is also supported by the POSCO Science Fellowship of POSCO TJ Park Foundation.

\bibliographystyle{plain}

\end{document}